\documentclass[11pt, reqno, a4paper]{amsart}
\usepackage[top = 1in, bottom = 1in, left=1in, right=1in]{geometry}

\usepackage[utf8]{inputenc}
\usepackage[activeacute,british]{babel}
\usepackage{microtype}

\allowdisplaybreaks
\usepackage{amssymb,amsmath,amsthm,graphicx,xcolor,mathtools,textcomp,mathrsfs}
\usepackage[mathscr]{euscript} 
\usepackage[shortlabels]{enumitem} 
\setlist[enumerate]{label=\normalfont{(\roman*)}}

\usepackage[
backend=biber,
style=alphabetic,
isbn=false,
maxbibnames=9]{biblatex}
\usepackage{csquotes}
\renewbibmacro*{doi+eprint+url}{
	\iftoggle{bbx:url}
	{\iffieldundef{doi}{\usebibmacro{url+urldate}}{}}{}
	\newunit\newblock
	\iftoggle{bbx:eprint}
	{\usebibmacro{eprint}}
	{}
	\newunit\newblock
	\iftoggle{bbx:doi}
	{\printfield{doi}}
	{}
}

\renewbibmacro{in:}{}
\DeclareFieldFormat{pages}{#1}
\DeclareFieldFormat{title}{\mkbibquote{#1}}
\DeclareFieldFormat[misc]{date}{Preprint (#1)}

\usepackage{hyperref}
\hypersetup{colorlinks, linkcolor={red!50!black}, citecolor={green!50!black}, urlcolor={blue!50!black}}

\usepackage[noabbrev, capitalise]{cleveref}
\newcommand{\refandname}[1]{\cref{#1} (\nameref*{#1})}
\newcommand{\nameandref}[1]{\nameref*{#1} (\cref{#1})}

\definecolor{DarkDesaturatedBlue}{HTML}{3A3556}
\definecolor{VividOrange}{HTML}{F15918}
\definecolor{PureOrange}{HTML}{FFBA00}
\definecolor{LightGrayishPink}{HTML}{EEC5D5}
\definecolor{VerySoftBlue}{HTML}{B5AFDB}

\makeatletter
\def\moverlay{\mathpalette\mov@rlay}
\def\mov@rlay#1#2{\leavevmode\vtop{%
		\baselineskip\z@skip \lineskiplimit-\maxdimen
		\ialign{\hfil$\m@th#1##$\hfil\cr#2\crcr}}}
\newcommand{\charfusion}[3][\mathord]{
	#1{\ifx#1\mathop\vphantom{#2}\fi
		\mathpalette\mov@rlay{#2\cr#3}
	}
	\ifx#1\mathop\expandafter\displaylimits\fi}
\makeatother

\newcommand{\cupdot}{\charfusion[\mathbin]{\cup}{\cdot}}
\newcommand{\bigcupdot}{\charfusion[\mathop]{\bigcup}{\cdot}}

\newcommand{\sta}[1]{\mathop{\mkern 0mu\mathrm{sta}}\nolimits(#1)}
\newcommand{\ter}[1]{\mathop{\mkern 0mu\mathrm{ter}}\nolimits(#1)}
\renewcommand{\int}[1]{\mathop{\mkern 0mu\mathrm{int}}\nolimits(#1)}

\usepackage{tikz}
\usetikzlibrary{math}
\usetikzlibrary{calc,decorations.pathmorphing,through}
\tikzset{snake it/.style={decorate, decoration=snake}}

\definecolor{DarkDesaturatedBlue}{HTML}{3A3556}
\definecolor{VividOrange}{HTML}{F15918}
\definecolor{PureOrange}{HTML}{FFBA00}
\definecolor{LightGrayishPink}{HTML}{EEC5D5}
\definecolor{VerySoftBlue}{HTML}{B5AFDB}

\newcommand{\triple}[7]{
	
	\ifx\relax#4\relax
	\def\qoffs{0pt}
	\else
	\def\qoffs{#4}
	\fi
	
	\def\qhedge{
		($#1+#3!\qoffs!-90:#2-#3$) --
		($#2+#1!\qoffs!-90:#3-#1$) --
		($#3+#2!\qoffs!-90:#1-#2$) -- cycle}
	
	\coordinate (12) at ($#1!\qoffs!90:#2$);
	\coordinate (13) at ($#1!\qoffs!-90:#3$);
	\coordinate (23) at ($#2!\qoffs!90:#3$);
	\coordinate (21) at ($#2!\qoffs!-90:#1$);
	\coordinate (31) at ($#3!\qoffs!90:#1$);
	\coordinate (32) at ($#3!\qoffs!-90:#2$);
	
	\def\nqhedge{
		(13) let \p1=($(13)-#1$), \p2=($(12)-#1$) in
		arc[start angle={atan2(\y1,\x1)}, delta angle={atan2(\y2,\x2)-atan2(\y1,\x1)-360*(atan2(\y2,\x2)-atan2(\y1,\x1)>0)}, x radius=\qoffs, y radius=\qoffs] --
		(21) let \p1=($(21)-#2$), \p2=($(23)-#2$) in
		arc[start angle={atan2(\y1,\x1)}, delta angle={atan2(\y2,\x2)-atan2(\y1,\x1)-360*(atan2(\y2,\x2)-atan2(\y1,\x1)>0)}, x radius=\qoffs, y radius=\qoffs] --
		(32) let \p1=($(32)-#3$), \p2=($(31)-#3$) in
		arc[start angle={atan2(\y1,\x1)}, delta angle={atan2(\y2,\x2)-atan2(\y1,\x1)-360*(atan2(\y2,\x2)-atan2(\y1,\x1)>0)}, x radius=\qoffs, y radius=\qoffs] --
		cycle}
	
	\ifx\relax#5\relax
	\def\qlwidth{1pt}
	\else
	\def\qlwidth{#5}
	\fi
	
	\ifx\relax#7\relax
	\fill \nqhedge;
	\else
	\fill[#7]\nqhedge;
	\fi
	
	\ifx\relax#6\relax
	\draw[line width=\qlwidth,rounded corners=\qoffs]\nqhedge;
	\else
	\draw[line width=\qlwidth,#6]\nqhedge;
	\fi
}

\definecolor{darkgreen}{rgb}{0.1,0.6,0.0}

\definecolor{ruby1}{RGB}{226,47,39}
\definecolor{ruby2}{RGB}{177,47,39}

\theoremstyle{plain}
\newtheorem{theorem}{Theorem}[section]
\newtheorem{proposition}[theorem]{Proposition}
\newtheorem{corollary}[theorem]{Corollary}
\newtheorem{lemma}[theorem]{Lemma}
\newtheorem{conjecture}[theorem]{Conjecture}
\newtheorem{claim}[theorem]{Claim}

\theoremstyle{remark}

\theoremstyle{definition}
\newtheorem{definition}[theorem]{Definition}

\numberwithin{equation}{section}

\DeclareMathOperator{\st}{st}
\DeclareMathOperator{\hc}{hc}

\DeclareMathOperator{\prob}{Pr}
\DeclareMathOperator{\expectation}{\mathbf{E}}
\def\eps{\varepsilon}
\def\calT{\mathcal{T}}

\def\calA{\mathcal{A}}

\def\calB{\mathcal{B}}
\def\le{\leqslant}
\def\leq{\leqslant}
\def\ge{\geqslant}
\def\geq{\geqslant}

\def\kdos{K^{(k)}(2)}
\def\degkdos{d^K}

\renewcommand{\dotsc}{...}
\renewcommand{\phi}{\varphi}

\title[]{Dirac-type conditions for spanning \\ bounded-degree hypertrees}
\date{}

\author[M.~Pavez-Sign\'e]{Mat\'ias Pavez-Sign\'e}
\address[M.~Pavez-Sign\'e]{Mathematics Institute, Zeeman Building, University of Warwick, Coventry CV4 7AL, UK}
\email{matias.pavez-signe@warwick.ac.uk}

\author[N.~Sanhueza-Matamala]{Nicol\'as Sanhueza-Matamala}
\address[N.~Sanhueza-Matamala]{Departamento de Ingeniería Matemática, Facultad de Ciencias Físicas y Matemáticas, Universidad de Concepción.}
\email{nicolas@sanhueza.net}

\author[M. Stein]{Maya Stein}
\address[M.~Stein]{Departamento de Ingenier\'ia Matem\'atica y Centro de Modelamiento Matemático (CNRS IRL 2807), Universidad de Chile, Beauchef 851, Santiago, Chile.}
\email{mstein@dim.uchile.cl}

\thanks{MPS was supported by ANID Doctoral scholarship ANID-PFCHA/Doctorado Nacional/2017-21171132 while he was affiliated to the Universidad de Chile. NSM acknowledges support by the Czech Science Foundation, grant number GA19-08740S with institutional support RVO: 67985807. MS was supported by ANID Regular Grant 1221905, by MathAmSud 20MATH-01, by FAPESP-ANID Investigaci\'on Conjunta grant 2019/13364-7, and by ANID/BASAL ACE210010 y FB210005.}

\begin{document}

\begin{abstract}
	We prove that for fixed $k$, every $k$-uniform hypergraph on $n$ vertices and of minimum codegree at least $n/2+o(n)$ contains every spanning tight $k$-tree of bounded vertex degree as a sub\-graph.
	This generalises a well-known result of Koml\'os, S\'ark\"ozy and Szemer\'edi for graphs.
	Our result is asymptotically sharp.
	We also prove an extension of our result to hypergraphs that satisfy some weak quasirandomness conditions. 
\end{abstract}

\maketitle
\thispagestyle{empty}
\vspace{-0.4cm}

\section{Introduction}
Forcing spanning substructures with minimum degree conditions  is a central topic in extremal graph theory.
For instance, a classic result of Dirac~\cite{Dirac1952} from~1952 asserts that any graph on $n\ge 3$ vertices with minimum degree at least $n/2$ contains a Hamilton cycle.
In the same spirit,  Bollob\'as~\cite{BollobasEGT} conjectured in the 1970s that graphs on $n$ vertices with minimum degree at least $n/2+o(n)$ contain every $n$-vertex tree of bounded maximum degree as a subgraph.
Koml\'os, S\'ark\"ozy and Szemer\'edi~\cite{KSS95} proved this conjecture in 1995. 
\begin{theorem}[Koml\'os, S\'ark\"ozy and Szemer\'edi~\cite{KSS95}]\label{thm:KSSoriginal} For all $\gamma>0$ and $\Delta\in\mathbb N$, there is $n_0$ such  that every graph $G$ on $n\ge n_0$ vertices with 
	$\delta(G)\ge (1/2+\gamma)n$ contains every $n$-vertex tree $T$ with  
	$\Delta(T)\le \Delta$.
\end{theorem}
In recent years, many efforts have been made to extend Dirac's theorem to $k$-uniform hypergraphs, also called \emph{$k$-graphs}. 
The \emph{minimum codegree} of a $k$-graph $H$, denoted $\delta_{k-1}(H)$, is the largest number $m$ such that every set of $k-1$ vertices from $H$ is contained in at least $m$ edges of $H$.
A notable result by R\"odl, Ruci\'nski and Szemer\'edi~\cite{RRS08} states that $k$-graphs on $n$ vertices and minimum codegree at least $n/2+o(n)$ contain a \emph{tight  Hamilton cycle}, where a tight Hamilton cycle consists of a cyclic ordering of the vertices of the $k$-graph such that every $k$ consecutive vertices in this ordering form an edge. More Dirac-type results for Hamilton cycles can be found in the survey~\cite{SimonovitsSzemeredi2019} by Simonovits and Szemer\'edi and the references therein.

In the present paper, we extend the Koml\'os--S\'ark\"ozy--Szemer\'edi theorem to $k$-graphs (Theorem~\ref{theorem:spanning-codegree}). To the best of our knowledge, our result is the first Dirac-type result for tightly connected spanning structures other than  {\it tight paths}, {\it tight cycles}, or \textit{triangulations of $2$-spheres}. The structures considered in our result are referred as to \emph{hypertrees} or \emph{tight $k$-trees}, which are defined next.

A \emph{tight $k$-tree} is a $k$-graph defined iteratively as follows: a single $k$-uniform edge is a tight $k$-tree; any $k$-graph obtained from a tight $k$-tree $T$ by adding a new vertex $v$ and a new edge $e$, such that $v\in e$ and $|e\cap e'|=k-1$ for some edge $e'\in E(T)$, is also a tight $k$-tree.
Observe that tight $2$-trees are the usual trees in graphs, since trees in graphs can be defined by successively adding leaves.
Also, the well-known $k$-uniform tight paths are tight $k$-trees.
Since no other kinds of trees will be considered here, we  usually just write \emph{$k$-tree} to refer to a tight $k$-tree.  

Extremal problems for $k$-trees have a long history.
In 1984, Kalai conjectured~\cite[Conjecture 3.6]{FranklFuredi87} that every $k$-graph on $n$ vertices with more than $\frac{t-1}{k} \binom{n}{k-1}$ edges
contains every $k$-tree with $t$ edges. For all $k$, this conjecture is tight for infinitely many $t$ and $n$.
In general, Kalai's conjecture is open, but there are partial and asymptotic results for special families of $k$-trees~\cite{FranklFuredi87, FurediJiang15, FJKMV2019, FJKMV17, FJKMV18}, among these are the (linear sized) tight paths~\cite{ABCM17},  for $k$-partite host $k$-graphs~\cite{Stein2019}, and for the case  $k=2$ (see~\cite{tree-survey} for references), where Kalai's conjecture reduces to the Erd\H os--S\'{o}s conjecture~\cite{Erdos64}.
Regarding spanning $k$-trees, it is also worth to mention the work of Georgakopoulos, Haslegrave, Montgomery, and Narayanan~\cite{GHMN2022}, who proved that large $n$-vertex $3$-graphs with minimum codegree at least $n/3 + o(n)$ have a spanning triangulation of a $2$-sphere, which in particular contains \emph{some} spanning $3$-tree.
In contrast, they also show that there are $3$-graphs with minimum codegree at least $ \lfloor n/3 \rfloor - 1$ where the largest tight $3$-tree has size at most $2 \lceil n/3 \rceil$.

Before stating our main result, we need a definition.
For a  $k$-graph $H$,  the \emph{maximum 1-degree} of $H$, denoted  $\Delta_1(H)$, is the maximum number $m$ such that some vertex of $H$ is contained in $m$ edges.
\begin{theorem} \label{theorem:spanning-codegree}
	For all $k, \Delta\ge 2$ and $\gamma > 0$, there is $n_0$ such that
	every $k$-graph $H$ on $n \ge n_0$ vertices with $\delta_{k-1}(H) \ge (1/2 + \gamma)n$
	contains every $k$-tree $T$ on $n$ vertices with $\Delta_1(T) \leq \Delta$.
\end{theorem}

The condition imposed on  $\delta_{k-1}(H)$ in Theorem~\ref{theorem:spanning-codegree} is best possible up to the term $\gamma n$ and an additive term depending on $k$, as shown by the next proposition (we postpone its proof until Section~\ref{section:lowerbounds}).
\begin{proposition} \label{proposition:lowerbound}
	For every $k \ge 2$ and for every $k$-tree $T$ on $n\ge k$ vertices, there is  a $k$-graph~$H$ on $n$ vertices not containing $T$, with $\delta_{k-1}(H) \ge \lfloor n/2\rfloor - f(T)$, where $f(T) \leq 2^k + k - 1$.
	Moreover, there are $k$-trees $T$ with $f(T)=k-1$.
\end{proposition}
Theorem~\ref{theorem:spanning-codegree} generalises to host graphs that have certain quasirandom properties.
For a $k$-graph $H$ and a  set  $F$ of distinct $(k-1)$-subsets of $V(H)$, we define the \emph{joint degree of $F$} as 
\begin{align}
\deg_H(F) = \big| \{v\in V(H) : \text{$f \cup \{ v \}\in H$ for each $f\in F$}\}\big|.\label{equation:jointdegree}
\end{align}
We say that $H$ is \emph{$(\varrho, h, \eps)$-typical} if $\left| \deg_H(F) - \varrho^{|F|} n \right| \leq \eps n$ for every set $F$ of $(k-1)$-sets such that $|F|\le h$.
We show that, for suitable choices of  $\varrho$ and $\eps$, every large $(\varrho, 2, \eps)$-typical $k$-graph contains every spanning $k$-tree with bounded degree.

\begin{theorem} \label{theorem:spanning-typical}
	For all $k, \Delta \ge 2$ and $\varrho > 0$, there are $n_0$ and $\eps_0>0$
	such that
	the following holds for all $0<\eps \leq \eps_0$.
	If $H$ is a $(\varrho, 2, \eps)$-typical $k$-graph on $n \ge n_0$ vertices,
	then $H$ contains every $k$-tree $T$ on $n$ vertices with $\Delta_1(T) \leq \Delta$.
\end{theorem}

Ehard and Joos~\cite{EhardJoos2020} showed very general results for finding (almost perfect packings of) spanning bounded-degree hypergraphs in host hypergraphs satisfying certain strong quasirandom conditions.
However, their results are incomparable with ours, as our quasirandomness conditions are much weaker.

We deduce Theorem~\ref{theorem:spanning-typical} from a slightly more general statement (Corollary~\ref{corollary:spanning-quasirandom}).
Weaker notions of quasirandomness and {\it minimum $1$-degree} (defined in analogy to the maximum $1$-degree) of order $\Theta(n^{k-1})$
 are not sufficient to guarantee the existence of any spanning $k$-tree 
  in dense $k$-graphs on $n$ vertices.
This follows from  examples of Ara\'ujo, Piga and Schacht~\cite{AraujoPigaSchacht2020} for $k \ge 3$.
See Section~\ref{section:quasirandom} for more details.

The paper is organised as follows. In Section~\ref{section:notation}, we introduce some notation and terminology that we will use throughout the paper. In Section~\ref{section:lowerbounds}, we prove Proposition~\ref{proposition:lowerbound} and in Section~\ref{section:sketch} we give an overview of the proof of Theorem~\ref{theorem:spanning-codegree} (which can be read independently of Sections~\ref{section:notation} and~\ref{section:lowerbounds}).
In Section~\ref{section:hypertrees}, we prove several results about tight $k$-trees which are higher-uniformity analogues of well-known results for trees.
In Section~\ref{section:tools}, we introduce some tools that will be used in the proof of our main result (Theorem~\ref{theorem:spanning-codegree}).
In Section~\ref{section:connecting}, we prove that $k$-graphs with large codegree are `well-connected', meaning that pairs of disjoint $(k-1)$-sets can be connected by many short walks of fixed length. In Section~\ref{section:embeddinglemma}, we prove an embedding result for $k$-trees into dense $k$-partite $k$-graphs, and in Section~\ref{section:absorption} we introduce a suitable absorption method for $k$-trees with bounded degree. In Section~\ref{section:main-theorem}, we use the results form Sections~\ref{section:hypertrees}-\ref{section:absorption} to prove Theorem~\ref{theorem:spanning-codegree}, and in Section~\ref{section:quasirandom} we extend Theorem~\ref{theorem:spanning-codegree} to hypergraphs satisfying certain quasirandomness conditions, in particular, proving Theorem~\ref{theorem:spanning-typical}. 
Section~\ref{section:remarks} contains some concluding remarks and open questions.

\section{Notation} \label{section:notation}
We introduce here some basic notation used throughout the paper.
More specific notions will be introduced where first needed.
Throughout this section, let $H$ be a $k$-graph, with $k\ge 2$.

\smallskip \noindent {\bf Hypertrees.} As stated in the introduction, a $k$-tree is a $k$-graph which can be defined iteratively as follows: 
\begin{enumerate}[(i)]
	\item  a single $k$-uniform edge is a $k$-tree;
	\item any $k$-graph obtained from a $k$-tree $T$ by adding a new vertex $v$ and a new edge $e$, such that $v\in e$ and $|e\cap e'|=k-1$ for some edge $e'\in E(T)$, is also a $k$-tree. 
\end{enumerate}
By definition, 
every $k$-tree with $n$ vertices has  $n - k + 1\ge 1$ edges, and hence $n\ge k$.
Also by definition,
 every $k$-tree $T$ on $n$ vertices has orderings $e_1, \dotsc, e_{n-k+1}$  and $v_{1}, \dotsc, v_{n}$ of its edges and vertices, respectively, such that $e_1 = \{ v_{1}, \dotsc, v_{k} \}$ and, for all $i \in \{k+1, \dotsc, n \}$,
\begin{enumerate}[\upshape (T1)]
    \item\label{hypertree:leafbyleaf} $\{v_{i}\} = e_{i-k+1}\setminus \bigcup_{1\le j < i-k+1} e_j$, and 
    \item\label{hypertree:joinedcorrectly} there exists $j \in [i-1]$ such that $e_{i-k+1} \setminus \{ v_{i}  \} \subseteq e_j$,  
\end{enumerate}
hold. Any ordering of $E(T)$ or $V(T)$ satisfying properties \ref{hypertree:leafbyleaf} and \ref{hypertree:joinedcorrectly} will be called a \emph{valid ordering}. {Sometimes, while referring to a valid ordering of $E(T)$, we will also understand an ordering of $V(T)$ as implicitly given by the ordering of $E(T)$, except for the ordering of the first $k$ vertices, which can be arbitrary. Similarly, we may refer to a valid ordering of $V(T)$, and then the ordering of $E(T)$ is implicit.}

If~$j\in[i-1]$ is the smallest index for which \ref{hypertree:joinedcorrectly} holds for $e_{i-k+1}$ and $e_{j}$, then we say that $e_j$ is the \emph{parent} of $e_{i-k+1}$ and that $e_{i-k+1}$ a \emph{child} of $e_j$. For $i\ge k$, the \emph{anchor of $v_i$} is $\alpha(v_i):=e_{{i-k+1}} \setminus \{ v_i \}$.

A \emph{$k$-subtree} of $T$ is a $k$-tree $T'$ such that $T' \subseteq T$. For instance, $e_1, \dotsc, e_{r}$ induces a  $k$-subtree of~$T$, for any $1\le r\le n-k+1$. Also, the tree $T-v_{n}$ obtained by removing $v_{n}$ and $e_{n-k+1}$ from $T$ is a $k$-subtree.

\smallskip \noindent {\bf $\ell$-partition}
We say that $H$ is \emph{$\ell$-partite} if there is a partition $\{V_1, \dotsc, V_\ell\}$ of $V(H)$ such that $|e\cap V_i|\le 1$ for every $e\in E(H)$ and $i\in[\ell]$.
It is easy to show (by induction on the number of vertices) that every $k$-tree is $k$-partite and, moreover, the $k$-partition of its vertices is unique.
 

\smallskip \noindent {\bf Homomorphisms and embeddings.}
If $H_1, H_2$ are hypergraphs, a \emph{hypergraph homomorphism of $H_1$ in $H_2$} is a function $\phi: V(H_1) \rightarrow V(H_2)$ that preserves edges.
If furthermore, $H_1\subseteq H'_1$ and $\phi': V(H'_1) \rightarrow V(H_2)$ is a hypergraph homomorphism such that $\phi'$ agrees with $\phi$ restricted to $V(H_1)$, then we call $\phi'$
an \emph{extension of $\phi$ to $H'_1$}.
If $\phi$ is injective, we say that $\phi$ is an \emph{embedding} of $H_1$ into $H_2$. An {extension of $\phi$} is then an extension which is also an embedding.

\smallskip \noindent {\bf Shadows and ordered shadows.} 
The \emph{shadow}  of $H$, denoted  $\partial H$, is the $(k-1)$-graph on vertex set $V(H)$ and whose edges are all the $(k-1)$-sets which are contained in some edge of $H$. The \emph{ordered shadow}  of $H$, denoted  $\partial^\circ H$, is defined as the set of all tuples $\mathbf{v}=(v_1,...,v_{k-1})$ with $\{v_1,...,v_{k-1}\}\in\partial H$, and we will use bold letters to denote its elements. If $\phi: \partial^\circ H \rightarrow \partial^\circ H'$ is a function, $\mathbf a= (a_1, \dotsc, a_{k-1}) \in\partial^\circ H$ and $\mathbf b = (b_1, \dotsc, b_{k-1}) \in\partial^\circ H'$, then $\phi(\mathbf a) = \mathbf b$ means that $\phi(a_i) = b_i$ for all $i \in [k-1]$. Furthermore, elements from $\partial^\circ H $ will be used as their underlying set for set-theoretical operations. For instance, if $\mathbf a=(a_1,\dotsc,a_{k-1}),\mathbf b=(b_1,\dotsc,b_{k-1})\in \partial^\circ H $, then $\mathbf a\cup\mathbf b=\{a_1,\dotsc,a_{k-1}\}\cup\{b_1,\dotsc,b_{k-1}\}$.

\smallskip \noindent {\bf Neighbourhoods and degrees.}
For $S \subseteq V(H)$, the \emph{neighbourhood of $S$ in $H$} is defined as $N_H(S) = \{ F \subseteq V(H) \setminus S : S \cup F \in H \}$.
If $x_1, \dotsc, x_{\ell} \in V(H)$,
we will write $N_H(x_1, \dotsc, x_{\ell})$ instead of $N_H(\{x_1, \dotsc, x_{\ell} \} )$.
The degree of $f\in \partial H$, denoted $\deg_H(f)$, is the number of edges of $H$ containing $f$ and equals $|N_H(f)|$.

\smallskip \noindent {\bf Walks.}
An ordered sequence $(x_1, \dotsc, x_n)$ of vertices from $H$ is a \emph{walk} if every $k$ consecutive vertices  form an edge of $H$.
We will often just write $x_1 \dotsb x_n$ instead of $(x_1, \dotsc, x_n)$ when using walks.
The \emph{length} of a walk is the number of its edges, e.g.~a walk on $n$ vertices has length $n - k + 1$.
The order of the vertices is important: $x_1 x_2 \dotsb x_n$ will generally be a different walk than $x_n \dotsb x_2 x_1$ (even though they use the same vertices and edges).
Note that a walk in which every vertex appears exactly once is a tight path in $H$.
At some points we will use that a walk corresponds naturally to a subgraph of $H$ (instead of a sequence of vertices).

Let $W = x_1 \dotsb x_n$ be a walk in $H$.
The \emph{start of $W$} is denoted by $\sta{W}:=(x_1, \dotsc, x_{k-1})$, and the \emph{end of $W$} is denoted $\ter{W}:=(x_{n-k+1}, \dotsc, x_{n})$. Both $\sta{W}$ and $\ter{W}$ belong to $\partial^o H$.
If $\sta{W}=\mathbf a$ and $\ter{W}=\mathbf b$, we also say that $W$ \emph{goes from $\mathbf a$ to $\mathbf b$}.
The \emph{interior}  of $W$, denoted $\int{W}$, is the set $V(W)\setminus(\sta{W}\cup \ter{W})$.
Thus $|\int{W}| \leq n - 2k - 2$, strict inequality holds if $W$ is not a path.
If $|\int{W}| = q$, we also say that $W$ \emph{has $q$ internal vertices}.
If $\int{W}\cap S=\emptyset$, we call $W$  \emph{internally disjoint} from~$S$.

\smallskip \noindent {\bf Numbers and hierarchies.} Given real numbers $x,y,z$, we write $x=y\pm z$ to denote that $x\in [y-z,y+z]$. Also, we write $a \ll b$ to mean that for $b > 0$, there exists $a_0 > 0$ such that for all $a \leq a_0$ the subsequent statements hold.
Hierarchies with more constants are defined analogously, and should always be read from right to left.
Implicitly, we assume that all constants appearing in a hierarchy are positive, and moreover if $1/m$ appears in a hierarchy then $m$ is an integer.

\section{Extremal example}\label{section:lowerbounds}
In this short section, we prove Proposition~\ref{proposition:lowerbound}.
The construction which witnesses the lower bound is similar in flavour to many other constructions in extremal hypergraph theory, as it is an example of a standard `parity obstruction'. 
We will use the following family of $k$-graphs.

\begin{definition} \label{definition:hab}
For disjoint sets $A,B$, and $0\le i\le k$, let $H_i \coloneqq \{ e \subseteq A \cup B:  |e|=k, |e \cap A| = i \}$,
	and $I \coloneqq \{ i \in \{ 0, \dotsc, k \} : i \not\equiv \lfloor k/2 \rfloor \bmod 2  \}$.
	Define $H(A,B) := \bigcup_{i \in I} H_i$.		
\end{definition}

Assuming that $|A\cup B|\ge k$, note that $\delta_{k-1}(H(A,B)) \ge \min\{ |A|,|B| \} - k + 1$.
The following lemma asserts that there are not too many ways to embed a $k$-tree into $H(A,B)$.
Recall that each $k$-tree admits a unique $k$-partition of its vertices. 

\begin{proposition} \label{proposition:lowerboundtool}
	Let $k, n\in\mathbb N$,  let $H(A,B)$ be as in Definition~\ref{definition:hab}, 
	with $|A \cup B| = n\ge k$.
	Let $T$ be a $k$-tree with $k$-partition $V_1 \cup \dotsb \cup V_k$,
	and an embedding $\phi: V(T) \rightarrow V(H(A,B))$.
	Then, for each $i\in [k]$ either  $\phi(V_i) \subseteq A$ or $\phi(V_i) \subseteq B$.
\end{proposition}

\begin{proof}
	We proceed by induction on $|E(T)|$; the base case $|E(T)| \le 1$ is trivial.
	Let $v$ and $e$ be the last vertex and edge of some valid ordering of $T$, let $e'$ be the parent edge of $e$ and let $v'$ be the vertex in $e'\setminus e$.
	Note that there exists $j \in [k]$ such that $v,v'\in V_j$.
	Applying the induction hypothesis to $T-v$ and to $\phi$ restricted to $V(T')$, we see that we only need to show that $\phi (v)\in A$ if and only if $\phi (v')\in A$.
	This is true, for otherwise, $\big||A\cap e|-|A\cap e'|\big|=1$, which contradicts the definition of $H(A,B)$.
\end{proof}

Now we are ready for the proof of Proposition~\ref{proposition:lowerbound}. 

\begin{proof}[Proof of Proposition~\ref{proposition:lowerbound}]
	Given $T$, together with the unique $k$-partition $\{ V_1, \dotsc, V_k \}$ of $V(T)$,
	choose $a(T)$ as the largest integer such that $a(T)\leq n/2$ and $a(T)\neq | \bigcup_{j \in J} V_j |$ for all $J \subseteq [k]$.
	Since $a(T)$ needs to avoid at most $2^k$ different values, it holds that $a(T) \ge \lfloor n/2\rfloor - 2^k$.
	Set $f(T)=\lfloor n/2\rfloor-a(T) + k-1$.
	
	Let $A$, $B$ be disjoint sets such that $|A|=a(T)$ and $|A\cup B|=n$,
	and consider the $k$-graph $H(A,B)$ as in Definition~\ref{definition:hab}.
	Then  $\delta_{k-1}(H(A,B)) \ge a(T) - k + 1 =\lfloor n/2\rfloor - f(T)$ (by the observation after Definition~\ref{definition:hab}), and $T$ does not embed into $H(A,B)$ because of Proposition~\ref{proposition:lowerboundtool} and by the choice of $a(T)$.
\end{proof}

\section{Overview of the proof of Theorem~\ref{theorem:spanning-codegree}}\label{section:sketch}

Let $H$ be an $n$-vertex $k$-graph with $\delta_{k-1}(H)\ge (1/2+\gamma)n$,  and let $T$ be an $n$-vertex $k$-tree with $\Delta_1(T)\le \Delta$.  We start by partitioning $T$ into three edge-disjoint subgraphs $T_1,T_2,T_3$ such that \begin{enumerate}
	\item $T_1$ and $T_2$ are subtrees of $T$,
	\item $|V(T_1)|\approx \alpha n$ and $|V(T_3)|\approx \nu n$, for some $0<\nu\ll\alpha \ll \gamma$,
	\item $V(T_1)\cap V(T_2)\in\partial T$, and 
	\item $T_2$ is obtained from $T-T_1$ by removing `leaves' one by one. 	
\end{enumerate}
We call $T_2$ the \textit{bulk} of $T$, which is the subgraph containing most vertices from $T$. \\

\noindent \textit{Building the absorbing structures:} We use $T_1$ to build gadgets in $H$ which will allow us extend a partial embedding of $T$ by adding vertices one by one. As we will see in Section~\ref{section:hypertrees}, the \textit{link graph} of a vertex in $T$ is a $(k-1)$-tree with $O(\Delta_1(T))$ vertices (Proposition~\ref{proposition:linkgraph}). Since $|V(T_1)|\approx \alpha n$, there is a $(k-1)$-tree $X$ such that linearly many vertices from $T_1$ have $X$ as its link graph (here is crucial that $\Delta_1(T)=O(1)$). Our gadgets in $H$, called $X$-tuples, consists of a copy $\tilde X$ of $X$ and a special vertex $u^*$ such that $\tilde X$ is contained in the link graph of $u^*$. The idea behind this gadget is that any vertex  whose link graph contains $\tilde X$ can be swapped with $u^*$ in a potential embedding of $T_1$ (see Section~\ref{section:absorption} for details). 

Using the large codegree in $H$, we can embed $T_1$ while covering a set of $\delta n$ disjoint $X$-tuples, with $\nu\ll \delta\ll \alpha $, which will be possible since $T_1$ contains linearly many vertices with $X$ as its link graph. Each $X$-tuple will be capable to `absorb' one arbitrary vertex at the time, and so, in total, this family will be able to absorb one by one any sequence of $\delta n$ vertices.\\

\noindent \textit{Decomposing the bulk of $T$:} In this step, we decompose $T_2$ into a constant number of smaller subtrees in a similar way as it has been done  for trees in graphs~\cite{AKS95}.   This is accomplished in Section~\ref{section:hypertrees}, where, in particular, we discuss {\it rooting} a $k$-tree at a $(k-1)$-set of its vertices and also develop the notion of \emph{layerings} of hypertrees, which resemble BFS-layerings of  rooted graphs.
Using these notions, we show (\cref{lemma:cut:trees}) that for any $\beta > 0$ one can decompose $T_2=D_1\cup\dotsb\cup D_p$, with $p=O(\beta^{-1})$, so that each $D_i$ is a $k$-tree of size $O(\beta n)$. Moreover, these parts are edge-disjoint and $V(D_i)\cap V(D_j)$ is either empty or is an element of $\partial T_2$. \\

\noindent \textit{Embedding of $T_2$:} The parts $D_1,\dotsc,D_p$ can be ordered and each of them can be rooted so that the first $\ell$ parts, for any $\ell\le p$, form a subtree of $T_2$ containing the `root' of part $\ell+1$. We will then
embed the parts successively, embedding  in each step one part (except its root, which is already embedded). Each $D_i$ will be embedded  into a suitable part of the host hypergraph using the regularity method. Fortunately, the  {\it weak regularity lemma for hypergraphs} (Theorem~\ref{theorem:weakregularity}) is sufficient for our purposes here, which simplifies the technical details of the proof and also gives better bounds for $n_0$. We only use this lemma in order to find an almost perfect matching $\mathcal M$ in the corresponding \textit{reduced graph}, which is a vertex disjoint collection of dense `regular $k$-tuples' covering most of $H$.  

Suppose we are about to embed the part $D_i$ which has its root already embedded. We first find an edge of $\mathcal M$ with sufficient free space, which spans a dense $k$-partite $k$-graph $F_i$ in $H$ where we will embed most of $D_i$. That is, we embed all but the first few layers of $D_i$ into $F_i$, because these layers will be needed to make the connection between the already embedded root of $D_i$ and $F_i$. It will be crucial here the bound on $\Delta_1(T)$. This will ensure that the number of vertices in the first few layers of $D_i$ is small and so most of $D_i$ is embedded into $F_i$. 

In order to connect the root of $D_i$ with $F_i$, we will use a part of $H$ that we have separated earlier, before applying regularity, and that we will only use for the connections.
This is the {\it reservoir},  a very small set $R\subseteq V(H)$ having (amongst others) the property  that every $(k-1)$-set has many neighbours in $R$. The reservoir is found  using a standard probabilistic argument (Lemma~\ref{lem:reservoir}).
A {\it connecting lemma} (Lemma~\ref{lemma:strengthenedconnectinglemma}) allows us to find many short walks between arbitrary pairs of ordered $(k-1)$-sets, whose internal vertices are all  inside the reservoir, and an enhanced version of this  lemma (Lemma~\ref{lem:embedding}) allows us to embed not only walks or paths, but instead bounded-size $k$-trees of bounded degree into the reservoir, joining given pairs of $(k-1)$-sets. This is what we need to finish the embedding described in the previous paragraph.\\

\noindent\textit{Absorption:} Recall that $T_2$ was obtained from $T-T_1$ by removing leaves one by one, which implies that $T_3$ is spanned by the last $\nu n$ vertices in a valid ordering of $T-T_1$. 

In order to embed $T_3$, we will use the collection of $X$-tuples we covered at the beginning of the proof, which is capable to absorb any sequence of $\delta n$ vertices. Since $T_3$ is formed by a sequence of  $\nu n\ll \delta n$ vertices, we can incorporate the vertices of $T_3$ one by one, using one $X$-tuple at the time, and thus finishing the embedding of $T$.
\section{Hypertrees}\label{section:hypertrees}

In this section, we establish some structural results about hypertrees.
Most importantly, we show any large hypertree can be decomposed into smaller hypertrees
(see \Cref{lemma:cut:trees}).

\subsection{Link graph of a $k$-tree} 
For a $k$-graph $H$ and $v \in V(H)$, the \emph{restricted link graph of $v$ with respect to $H$}, denoted $H(v)$, is the $(k-1)$-graph whose vertex set is $\bigcup\{e\setminus \{v\}:v\in e\}$ and its edge set is $\{e\setminus\{v\}:v\in e\}$.

\begin{proposition} \label{proposition:linkgraph}
	Let $k\ge 2$, let $T$ be a $k$-tree, and let $v \in V(T)$. Then $T(v)$ is a $(k-1)$-tree on at most $\Delta_1(T) + k - 1$ vertices.
\end{proposition}

\begin{proof}
	Let $e_1, \dotsc, e_m$ be a valid ordering of the edges of $T$, and let $I = \{ i \in [m] : v \in e_i \}$.
	Then $E(T(v))=\{ e_i \setminus \{v\} : i \in I \}$, and $I$ induces a valid ordering of $E(T(v))$, with $\alpha(v)$ being the first edge.
	So $T(v)$ is a $(k-1)$-tree.
	Since $v$ belongs to at most $\Delta_1(T)$ edges in $T$,
	we know that $T(v)$ has at most $\Delta_1(T)$ edges, and thus at most $\Delta_1(T) + k - 1$ vertices.
\end{proof}

\subsection{Layerings}
It is often convenient to root a $2$-tree $T$  at some vertex $r\in V(T)$, which gives rise to a rooted tree $(T,r)$.
Then one can define the \emph{$i$-th layer $L_i$ of $(T,r)$} as the set of all vertices at distance exactly $i$ from $r$ in $T$.
The layers partition $V(T)$, and any vertex in layer $i+1$ is joined to some vertex in layer $i$.

We now introduce a generalisation of these notions to higher uniformities.

\begin{definition} \label{definition:rooted}
	A \emph{rooted $k$-tree} is a pair $(T,\mathbf x)$ where $T$ is a $k$-tree and $\mathbf x\in \partial^\circ T$.
\end{definition}

\begin{definition}[Layering]\label{def:layering}
	Let $(T,\mathbf x)$ be a rooted $k$-tree {with $\mathbf x=(x_1,\dotsc,x_{k-1})$}.
	A \emph{layering} for $(T,\mathbf x)$ is a tuple $\mathcal L=(L_1,\dotsc,L_m)$, for some $m\in\mathbb N$, such that $\{L_1, \dotsc, L_{m} \}$ is a partition of $V(T)$, and
	\begin{enumerate}[(L1)]
		\item {$\mathbf x \cap L_i=\{x_i\}$ for all $i \in [k-1]$,
		\label{item:flatten-root}
		and $L_1 = \{x_1\}$,}
		\label{item:flatten-initial}
		\item for each  $v \in L_{i+1}$ with $1\le i< m$ there are $w \in L_{i}$, $e \in E(T)$  such that $\{v, w\} \subseteq e$, and
		\label{item:flatten-degree}
		\item for each $e \in E(T)$, there is $j\in[m]$ such that $|e \cap L_{i}| = 1$ for each $j\le i <j+k$.
		\label{item:flatten-correct}
	\end{enumerate}
	We call the tuple $(T,\mathbf x,\mathcal{L})$ a \emph{layered} $k$-tree.
\end{definition}

Note that a layering $(L_1, \dotsc, L_m)$ of $(T,\mathbf x)$ 
is the preimage of the tight path on~$m$ vertices under a homomorphism
that maps all of $L_i$ to the $i$-th vertex of the tight path. 

  \begin{figure}[h!]
	\centering
\begin{tikzpicture}[thick, scale=1]
	
	\draw (0,5.4) node {$v_1$};
	\draw (-1.3,4.65) node {$v_2$};
	\draw (1,4.2) node {$v_3$};
	\draw (0.05,2.25) node {$v_4$};
	\draw (-1.9,3.55) node {$v_5$};
	\draw (-3.5,3.5) node {$v_6$};
	\draw (-3.25,1.95) node {$v_7$};
	\draw (-1,1.6) node {$v_8$};
	\draw (-1.55,1) node {$v_9$};
	\draw (2.55,3) node {$v_{10}$};
	\draw (2.45,1.8) node {$v_{11}$};
	\draw (2.45,4.5) node {$v_{12}$};
	
	\draw (-3,5) node {$T$};
	\draw (5,5) node {$\mathcal L$};
	
	\draw (5,3) node {$v_1$};
	\draw (5.9,3.1) node {$v_2$};
	\draw (7,3) node {$v_3$};
	\draw (8,3) node {$v_4$};
	\draw (7,3.5) node {$v_5$};
	\draw (8,3.5) node {$v_6$};
	\draw (9,3) node {$v_7$};
	\draw (9,3.5) node {$v_8$};
	\draw (8,4) node {$v_9$};
	\draw (9,4) node {$v_{10}$};
	\draw (10,3) node {$v_{11}$};
	\draw (6,3.5) node {$v_{12}$};
	
	\draw (5,2) node {$L_1$};
	\draw (6,2) node {$L_2$};
	\draw (7,2) node {$L_3$};
	\draw (8,2) node {$L_4$};
	\draw (9,2) node {$L_5$};
	\draw (10,2) node {$L_6$};
	
	\draw (4.5,2.5) -- (10.5,2.5);
	\draw (5.5,1.65) -- (5.5,4.4);
	\draw (6.5,1.65) -- (6.5,4.4);
	\draw (7.5,1.65) -- (7.5,4.4);
	\draw (8.5,1.65) -- (8.5,4.4);
	\draw (9.5,1.65) -- (9.5,4.4);

	
	\coordinate (v1) at (0,5);
	\coordinate (v2) at (-1,4.3);
	\coordinate (v3) at (0.7,3.9);
	\coordinate (v4) at (0,2.8);
	\coordinate (v5) at (-1.7,3.1);
	\coordinate (v6) at (-3,3.5);
	\coordinate (v7) at (-2.8,2);
	\coordinate (v8) at (-1,2);
	\coordinate (v9) at (-2,1);
	\coordinate (v10) at (2,3);
	\coordinate (v11) at (1.9,1.9);
	\coordinate (v12) at (1.9,4.5);	
	
	\tikzstyle{every node}=[circle, draw, fill, inner sep=0pt, minimum width=2pt]
	\draw (v1) node {};
	\draw (v2) node {};
	\draw (v3) node {};
	\draw (v4) node {};
	\draw (v5) node {};
	\draw (v6) node {};
	\draw (v7) node {};
	\draw (v8) node {};
	\draw (v9) node {};
	\draw (v10) node {};
	\draw (v11) node {};
	\draw (v12) node {};
	
	\triple{(v2)}{(v1)}{(v3)}{6pt}{1pt}{DarkDesaturatedBlue,opacity=0.8}{VerySoftBlue,opacity=0.2};
	\triple{(v2)}{(v3)}{(v4)}{6pt}{1pt}{DarkDesaturatedBlue,opacity=0.8}{VerySoftBlue,opacity=0.2};
	\triple{(v2)}{(v4)}{(v5)}{6pt}{1pt}{DarkDesaturatedBlue,opacity=0.8}{VerySoftBlue,opacity=0.2};
	\triple{(v2)}{(v5)}{(v6)}{6pt}{1pt}{DarkDesaturatedBlue,opacity=0.8}{VerySoftBlue,opacity=0.2};
	\triple{(v6)}{(v5)}{(v7)}{6pt}{1pt}{DarkDesaturatedBlue,opacity=0.8}{VerySoftBlue,opacity=0.2};
	\triple{(v5)}{(v4)}{(v8)}{6pt}{1pt}{DarkDesaturatedBlue,opacity=0.8}{VerySoftBlue,opacity=0.2};
	\triple{(v7)}{(v5)}{(v9)}{6pt}{1pt}{DarkDesaturatedBlue,opacity=0.8}{VerySoftBlue,opacity=0.2};
	\triple{(v4)}{(v3)}{(v10)}{6pt}{1pt}{DarkDesaturatedBlue,opacity=0.8}{VerySoftBlue,opacity=0.2};
	\triple{(v4)}{(v10)}{(v11)}{6pt}{1pt}{DarkDesaturatedBlue,opacity=0.8}{VerySoftBlue,opacity=0.2};
	\triple{(v3)}{(v1)}{(v12)}{6pt}{1pt}{DarkDesaturatedBlue,opacity=0.8}{VerySoftBlue,opacity=0.2};
	
\end{tikzpicture}
	\caption{On the left, a $3$-tree $T$ with valid ordering $v_1, \dotsc, v_{12}$.
		On the right, a table shows the layering $\mathcal L=(L_1,\dotsc, L_6)$ of $T$ when it is rooted at $\mathbf x=(v_1,v_2)$.} 
\end{figure}

\begin{lemma}\label{lemma:flattening}
	Every rooted $k$-tree $(T,\mathbf x)$ has a unique layering.
\end{lemma}
\begin{proof}The proof is by induction on $|E(T)|$. Let us note that if $T$ is a single edge, then the first $k-1$ levels of the layering of $(T,\mathbf x)$ correspond to $\mathbf x$, and the last level corresponds to the unique vertex in $T-\mathbf x$. So, we assume that $|E(T)|\ge 2$ and, in some valid ordering of $V(T)$, let $v$ and $e$ be the last vertex and last edge, respectively.
Let $e'$ be the parent of $e$, and let $w$ be the only vertex in $e'\setminus e$. 

We shall argue first about the existence of the layering for $(T, \mathbf{x})$, we will argue about uniqueness later.
Let $\mathbf x =(x_1,\dotsc,x_{k-1})$.
We consider three cases.
Suppose first that $v\notin \mathbf x$.
Let $\mathbf{x}' = \mathbf{x}$.
By induction, $(T-v,\mathbf x')$ has a unique layering $\mathcal L'=(L'_1,\dotsc,L'_m)$ which we can extend to a layering of $(T,\mathbf x)$ by either adding $v$ to the layer $L'_i$ containing $w$, or (if all other vertices of $e'$ lie in later layers than $w$) by adding $v$ to  $L'_{i+k}$.
Next, suppose $v=x_j$ for some $j\in [k-1] \setminus \{1\}$.
Then we set $\mathbf x'=(x_1',\dotsc,x_{k-1}')\in \partial^\circ T$, where $x'_\ell =x_\ell$ for $\ell\not =j$ and $x'_j=w$.
By induction, $(T-v,\mathbf x')$ has a unique layering $(L'_1, \dotsc, L'_m)$.
We extend this layering to a layering of $(T, \mathbf x)$ by adding $v$ to the layer that hosts $w$.
It is easy to check that~\ref{item:flatten-initial}--\ref{item:flatten-correct} hold for our layering of $T$.
Finally, suppose $v =x_1$.
In this case, set $\mathbf{x}' = (x_2, x_3, \dotsc, x_{k-1}, w)$.
Again, by induction $(T-v, \mathbf{x}')$ has a unique layering $(L'_1, \dotsc, L'_m)$.
We set $L_1=\{v\}$ and $L_i = L'_{i-1}$ for all $2 \leq i \leq m+1$.
Again,~\ref{item:flatten-initial}--\ref{item:flatten-correct} hold for $(L_1, \dotsc, L_{m+1})$.

It is also straightforward to check that, in all cases, the obtained layering $\mathcal L$ must be unique, as the layering obtained from $\mathcal{L}$ by removing $v$ will yield a layering of $(T - v, \mathbf{x}')$, which is must be unique by induction. 
\end{proof}

Note that Definition~\ref{def:layering}~\ref{item:flatten-degree} 
gives that $|L_{i+1}| \leq \Delta_1(T) |L_i|$ for all $i \in [m-1]$.
So, by~\ref{item:flatten-initial} we have the following easy observation.

\begin{proposition} \label{proposition:boundedclusters}
	Let $(T,\mathbf x, \mathcal{L})$ be a layered $k$-tree with $\mathcal{L}=(L_1, \dotsc, L_m )$.
	Then, $|L_i| \leq (\Delta_1(T))^{i-1}$ for all $i \in [m]$. \hfill $\qed$
\end{proposition}

Recall that $k$-trees are $k$-partite, and that each $k$-tree $T$ admits a unique $k$-partition $\{V_1, \dotsc, V_k\}$ of $V(T)$.
Given the layering $\mathcal L=(L_1, \dotsc, L_m)$ of $(T,\mathbf x)$, it is clear from  Definition~\ref{def:layering}~\ref{item:flatten-correct} that (after relabelling the partition classes) each $V_i$ contains all layers $L_{i+k\mathbb N}$.
We use this to deduce that the sizes of the partition classes of a $k$-tree cannot differ too much, as detailed in the following lemma.

\begin{proposition} \label{proposition:boundedsizecolourclasses}
	Let $\Delta, k\ge 2$ and let $T$ be a $k$-tree with $k$-partition $\{V_1, \dotsc, V_k \}$ and with
	$\Delta_1(T) \leq \Delta$.
	Then $|V_i| \leq \Delta^{k-1}(1+|V_j|)$ for each $i, j \in [k]$.
\end{proposition}

\begin{proof}
	Let $\mathcal{L} = (L_1, \dotsc, L_m)$ be a layering of $T$.
	By~\ref{item:flatten-initial} and~\ref{item:flatten-degree} we have that $|L_1| = 1$ and $|L_{i+1}| \leq \Delta_1(T) |L_i|\leq \Delta |L_i|$ for all $i \in [m-1]$.
	We can assume that $V_1, \dotsc, V_k$ are  so that for each $i \in [k]$, the set $V_i$ contains precisely the layers $\{ L_{i + jk} : j \geq 0 \}$.
	
	So, for $1\le i \leq k$, we have $|V_i| = \sum_{j \geq 0} |L_{i+jk}| \leq \sum_{j \geq 0} \Delta |L_{i-1+jk}| = \Delta |V_{i-1}|$.
	Secondly, note that $|V_1| = \sum_{j \geq 0} |L_{1+jk}| = |L_1| + \sum_{j \geq 1} |L_{1+jk}| \leq 1 + \sum_{j \geq 1} \Delta |L_{jk}| = 1 + \Delta |V_k|$.
	The desired bound follows by applying these inequalities repeatedly.
\end{proof}

\subsection{Pseudopaths in $k$-trees}
A basic fact about $2$-trees is that every two vertices are joined by a unique path.
We now introduce pseudopaths, which play a similar role in hypertrees.

\begin{definition}[Pseudopath]
A $k$-tree $P$ is a \emph{pseudopath} (of uniformity $k$) if there exists a valid ordering $e_1, \dotsc, e_t$ of $E(P)$ such that for every $i < t$, the only child of edge $e_i$ is $e_{i+1}$, in which case we say that $e_1,\dotsc, e_t$ is a \textit{path-ordering} for $P$.\end{definition}

\begin{figure}[h!]
	\centering
	\begin{tikzpicture}[thick, scale=.7]
		\draw (0,.5) node {$v_0$};
		\draw (180:2.6) node {$v_1$};
		\draw (210:2.6) node {$v_2$};
		\draw (240:2.6) node {$v_3$};
		\draw (270:2.6) node {$v_4$};
		\draw (300:2.6) node {$v_5$};
		\draw (330:2.6) node {$v_6$};
		\draw (360:2.6) node {$v_7$};		
		\tikzstyle{every node}=[circle, draw, fill, inner sep=0pt, minimum width=2pt]
		\draw (0,0) node {};
		\draw (180:2) node {};
		\draw (210:2) node {};
		\triple{(0,0)}{(210:2)}{(180:2)}{6pt}{1pt}{DarkDesaturatedBlue,opacity=0.8}{VerySoftBlue,opacity=0.2};
		\draw (240:2) node {};
		\triple{(0,0)}{(240:2)}{(210:2)}{6pt}{1pt}{DarkDesaturatedBlue,opacity=0.8}{VerySoftBlue,opacity=0.2};
		\draw (270:2) node {};
		\triple{(0,0)}{(270:2)}{(240:2)}{6pt}{1pt}{DarkDesaturatedBlue,opacity=0.8}{VerySoftBlue,opacity=0.2};
		\draw (300:2) node {};
		\triple{(0,0)}{(300:2)}{(270:2)}{6pt}{1pt}{DarkDesaturatedBlue,opacity=0.8}{VerySoftBlue,opacity=0.2};
		\draw (330:2) node {};
		\triple{(0,0)}{(330:2)}{(300:2)}{6pt}{1pt}{DarkDesaturatedBlue,opacity=0.8}{VerySoftBlue,opacity=0.2};
		\draw (0:2) node {};
		\triple{(0,0)}{(360:2)}{(330:2)}{6pt}{1pt}{DarkDesaturatedBlue,opacity=0.8}{VerySoftBlue,opacity=0.2};
	\end{tikzpicture}
	\caption{The $3$-graph $F_{3,7}$.}
	\label{figure:fan}
\end{figure}
Observe that a pseudopath $P$ can have many different valid orderings of its edges which make it a $k$-tree, but not necessarily all valid orderings will be path-orderings.
As an example, consider the $k$-tree $F_{k,t}$ (see Figure~\ref{figure:fan}) with vertex set $\{v_0,v_1,\dotsc, v_t\}$ and edges $e_i=\{v_0,v_i,\dotsc, v_{i+k-2}\}$, for $1\le i\le t-k+2$.
We easily see that $e_1,\dotsc, e_{t-k+2}$ is a path-ordering for $F_{k,t}$.
However, rooting $F_{k,t}$ at any $(k-1)$-set $f\subseteq e_i$, for $2\le i\le t-k-3$, gives a valid ordering which is not a path-ordering.

\begin{definition}Given a $k$-graph $H$ and distinct $f,f' \in \partial H$, an \emph{$(f,f')$-pseudopath} in $H$ is a pseudopath $P\subseteq H$ with a path-ordering $e_1, \dotsc, e_t$ such that $f \subseteq e_i$ and $f' \subseteq e_j$ if and only if $(i,j)=(1,t)$. \end{definition}

Observe that if an $(f,f')$-pseudopath $P$ has at least two edges, then exactly two vertices of $V(P)$ have degree $1$,
and these vertices are contained in $f$ and $f'$ respectively.
We use this observation to show existence and uniqueness of pseudopaths in hypertrees.

\begin{lemma}
	Let $T$ be a $k$-tree. Then, for any distinct $f,f' \in \partial T$ there is a unique $(f,f')$-pseudopath in~$T$. 
\end{lemma}

\begin{proof}
    We use induction on $|E(T)|$.
    If $|E(T)|\le 1$, the statement clearly holds.
    Otherwise, let~$v$ and $e$ be the last vertex and last edge in a valid ordering of $V(T)$. If $v \notin f\cup f'$, then by induction, the tree $T-v$ contains a unique $(f,f')$-pseudopath $P$.
    This path remains unique in~$T$.
    Indeed, note that in any $(f,f')$-pseudopath $P'$ in $T$ with at least two edges the only vertices with degree $1$ in $P'$ must be included in the first or last edge, and included in $f \cup f'$.
    Since $v$ has degree $1$ in $T$, if there were a $(f,f')$-pseudopath $P'$ in $T$ with at least two edges including $v$, this would imply that $v$ is in the first or last edge of $P'$, and also $v \in f \cup f'$, a contradiction.
    It also cannot happen that $P'$ is an $(f,f')$-pseudopath consisting of a single edge and including~$v$, since then the edge is equal to $f \cup f'$ and again it would imply that $v \in f \cup f'$.

    So assume $v\in f\cup f'$.
    If $v\in f\cap f'$, then $f, f'\subseteq e$, and therefore, $e$ is a $(f,f')$-pseudopath, and it is unique.
    We can thus suppose that $v\in f\setminus f'$, which implies $f\subseteq e$.
    By induction, $T-v$ contains a unique $(e\setminus \{v\},f')$-pseudopath $P'$, which can be extended to an $(f,f')$-pseudopath $P$ by adding $e$.
    Since $v$ has degree~$1$ in $T$, any $(f,f')$-pseudopath in $T$ contains $e$.
    So, as~$P'$ was unique, $P$ is unique too.
\end{proof}

A set $f$ of $k-1$ vertices is said to {\it lie on} a pseudopath~$P$, if either $P$ is an $(f,f')$-pseudopath for some $f'$, or $f$ is contained in exactly two of the edges of $P$.

\begin{definition}[Distance] \label{definition:distance}
	Given a $k$-tree $T$ and distinct $f,f'\in \partial T$, the \emph{distance} $d_T(f,f')$ between $f$ and $f'$ is  the number of edges in the unique $(f,f')$-pseudopath connecting $f$ with $f'$.
	If $f = f'$, we let $d_T(f,f') = 0$.
\end{definition}

Note that $d_T(f,f')\ge 1$ for all distinct $f,f'\in\partial T$, with equality if and only if $f\cup f'\in E(T)$.
Given tuples $\mathbf x,\mathbf y\in\partial^\circ T$ and $f\in\partial T$, we write $d_T(\mathbf x,\mathbf y)$ for the distance between the underlying $(k-1)$-sets of $\mathbf x$ and $\mathbf y$, and let $d_T(\mathbf x,f)$ denote the distance between $f$ and the underlying set of~$\mathbf x$.

\begin{lemma} \label{lemma:layeringpseudopath}
	Let $P$ be an $(f,f')$-pseudopath of uniformity $k$ with a path-ordering $e_1, \dotsc, e_t$.
	Let $\mathbf{x} \in \partial^\circ P$ be such that $\mathbf{x} \subseteq e_1$,
	and let $\mathcal{L} = (L_1, \dotsc, L_m)$ be the unique layering of $(P, \mathbf{x})$.
	Then, setting $r(j) = \min \{ i : L_i \cap e_j \neq \emptyset \}$ for $j=1,\dots,t$, we have
	\begin{enumerate}
		\item \label{item:layerpseudo-increasing}  $r(j+1) - r(j) \in\{0,1\}$ for all $j\in[t-1]$, and
		\item \label{item:layerpseudo-bounded} $|L_i| \leq k \Delta_1(P)$ for all $i \in [m]$.
	\end{enumerate}
\end{lemma}

\begin{proof}
	We begin by describing explicitly how can one construct $\mathcal{L}$ by adding edges iteratively, as follows.
	First, start with all $L_1, \dotsc, L_m$ empty.
	Let $\mathbf{x} = (x_1, \dotsc, x_{k-1})$ and $x_k$ is the unique vertex in $e_1 \setminus \mathbf{x}$.
	Begin by adding $x_i$ to $L_i$ for all $1 \leq i \leq k$.
	Now, given $2 \leq j \leq t$, assume that $e_j$ already has been included,
	it has one vertex in each of the layers $L_{i+1}, \dotsc, L_{i+k}$,
	and we need to allocate $e_{j+1}$.
	Let $x, y$ be the unique vertices in $e_j \setminus e_{j+1}$ and $e_{j+1} \setminus e_j$, respectively.
	If $x \notin L_{i+1}$, we add $y$ to the same layer which contains $x$; otherwise we add $y$ to $L_{i+k+1}$.
	It is straightforward to show by induction that this construction satisfies \ref{item:flatten-root}--\ref{item:flatten-correct}, and since there is a unique layering by \cref{lemma:flattening}, this construction precisely describes $\mathcal{L}$.
	
	Now, we show that~\ref{item:layerpseudo-increasing} holds.
	Let $1 \leq j < t$.
	Note that since $|e_j \cap e_{j+1}| = k-1$, together with~\ref{item:flatten-correct} it must hold that $|r(j+1) - r(j)| \leq 1$.
	Thus we only need to show that $r(j+1) \geq r(j)$, but this follows immediately from the iterative construction for $\mathcal{L}$ which we described before.

	For~\ref{item:layerpseudo-bounded}, set $\Delta:= \Delta_1(T)$ and observe that since $P$ is a pseudopath,  for every $x \in V(P)$ there are $j \leq |E(P)|$ and $d < \Delta$ such that $x \in e_{i}$ if and only if $j \leq i \leq j + d$. In particular, because of~\ref{item:layerpseudo-increasing} and  \ref{item:flatten-correct}, we have
\begin{equation}\label{ray}
\text{$e_i\cap L_{r(j)}=\emptyset$ for all $1 \leq j \leq t$ and all $j+\Delta \leq i \leq t$.}
\end{equation}
Now assume for contradiction that there is an index $i\in [m]$ with $|L_i| >  k \Delta$. Note that each vertex in $L_i$ belongs to an edge that by~\ref{item:flatten-correct} meets the $k$ levels $L_\ell, L_{\ell+1},..., L_{\ell+k-1}$ for some $\ell\in\{i-k+1,  ..., i\}$. So, there is an index $\ell\in\{i-k+1, ..., i\}$ such that more than $\Delta$ edges meet all of the levels $L_\ell, L_{\ell+1},..., L_{\ell+k-1}$. Let $j\in[t]$ be minimum with the property  that $r(j)=\ell$. Then by~\eqref{ray}, 
only edges $e_j, e_{j+1}, ..., e_{j+\Delta-1}$ may meet $L_\ell$. As these are only $\Delta$ edges, we arrive at 
a contradiction, as desired.	
\end{proof}

\subsection{Cutting  $k$-trees}
We will now show how to partition a $k$-tree into smaller $k$-subtrees of controlled size.
Given a layered $k$-tree $(T,\mathbf x,\mathcal{L})$, with $\mathcal{L} = (L_1, \dotsc, L_m)$,
and given $\mathbf s=(s_1,\dotsc,s_{k-1})\in\partial^\circ T$, we say $\mathbf s$ is \emph{$\mathcal{L}$-layered} if  $\mathbf s \cap L_{i} = \{s_i\}$ for each $i =j, ..., j+k-2$ for some $j \in [m]$,
that is, $\mathbf s$ meets $k-1$ consecutive layers of $\mathcal{L}$.
In that case we say that $j$ is the \emph{rank} of~$\mathbf s$.

\begin{definition}[Induced $k$-subtree]
    Let $(T,\mathbf x,\mathcal{L})$ be a layered $k$-tree, with $\mathcal{L} = (L_1, \dotsc, L_m)$, and let $\mathbf s\in\partial ^\circ T$ be $\mathcal{L}$-layered.
    The tree $T_{\mathbf s}$ \textit{induced by $\mathbf s$} is the $k$-subtree of $T$ spanned by $\bigcup_{i\ge 0}E_i$ where $E_0:=\{\mathbf s\cup\{v\}:\alpha(v)=\mathbf s\}$ and $E_{i+1}$ contains all children of edges in $E_i$. Observe that $T_{\mathbf s}$ might be edgeless.
    Write $T-T_{\mathbf s}$ for the tree obtained from $T$ by deleting all edges in $E(T_{\mathbf s})$, and all vertices in $V(T_{\mathbf s})\setminus {\mathbf s}$.
\end{definition}

Clearly, if $T$ is rooted at $\mathbf x$, then $T_{\mathbf x}=T$.
Observe that if $(T_{\mathbf s}, {\mathbf s})$ is an induced $k$-subtree of $(T,\mathbf x)$,
and $\mathbf s' \in \partial^\circ T_{\mathbf s'}$ is $\mathcal{L}$-layered,
then the induced $k$-subtree $((T_{\mathbf s})_{\mathbf s'},\mathbf s')$ of $(T_{\mathbf s}, \mathbf s)$ is also an induced $k$-subtree of $(T,\mathbf x)$,
and we have $((T_{\mathbf s})_{\mathbf s'},\mathbf s') = (T_{\mathbf s'}, \mathbf s')$. Note that for each $f\in\partial T_{\mathbf s}$, the underlying set of $\mathbf s$ lies on the unique pseudopath from $f$ to the root in $T$.
Moreover,  $T_{\mathbf s}$ inherits a valid ordering and a layering from~$(T, \mathbf x, \mathcal{L})$, with layers $L_j \cap V(T_{\mathbf s})$, which we call  the \emph{inherited layering} of $(T_{\mathbf s},\mathbf s)$ and denote by $\mathcal{L}^{\mathbf s}$. 

The following observation will be useful in a moment.

\begin{proposition}\label{proposition:cut1}
	Let $(T,\mathbf x, \mathcal{L})$ be a layered $k$-tree with  $\mathcal L=(L_1,\dotsc,L_{m})$, $\Delta_1(T) \leq \Delta$ and $k\ge 2$. Let $F\subseteq E(T)$ be the set of all edges meeting $L_1$ and  and let $\mathbf{S}\subseteq\partial^\circ T$ consist of all the $\mathcal L$-layered tuples whose unordered vertices are in
	$\{e \setminus L_1 : e\in F\}$. Then,
	\begin{enumerate}[\upshape(i)]
	\item each $\mathbf s\in \mathbf S$ is $\mathcal L$-layered and has rank  2,
	 \item $|F|=|\mathbf S|\le\Delta$, and
	 \item $E(T)=F\cupdot\bigcupdot_{\mathbf s \in \mathbf S}E(T_\mathbf s)$.
	 \end{enumerate}
\end{proposition}

\begin{proof}
	 As $\Delta_1(T) \leq \Delta$ and $|L_1|=1$, we have $|F|\le\Delta$. The other properties are  easy to see.
\end{proof}

The next definition captures the previously mentioned partition of a $k$-tree.
Intuitively, it ensures that the small $k$-trees in the partition are of controlled size (no small $k$-tree is too large, and there are not many $k$-trees in the partition).
Also, the roots of each $k$-tree are ``far apart'' from each other, as measured by their rank.

\begin{definition}[$(\beta,d)$-decomposition]\label{beta-decomp}
	Let $\Delta,k\ge 2$,  and let $(T,\mathbf x, \mathcal{L})$ be a layered $k$-tree. For $\beta\in (0,1)$ and $d \ge 1$, a \emph{$(\beta, d)$-decomposition of $(T,\mathbf x,\mathcal{L})$} is a tuple $(D_i, \mathbf s_{i})_{1\le i\le m}$ of rooted $k$-subtrees of $T$ such that 
	\begin{enumerate}
	\item \label{item:beta-notsomany} $m\le 2\Delta^{d}/\beta$,
		\item \label{item:beta-edgedecomposition} $E(T)=\bigcupdot_{1\le i\le m} E(D_i)$,
		\item \label{item:beta-small} $|E(D_i)|\le \beta |E(T)|$ for each $1\le i\le m$,
				\item \label{item:beta-layeredroots} $\mathbf s_{1}=\mathbf x$ and each $\mathbf s_{i}$ is $\mathcal{L}$-layered, 
		\item \label{item:beta-singleparent} $(V(D_\ell) \setminus \mathbf s_\ell)\cap V(D_i) =\emptyset$ for all $1\le\ i < \ell \le m$, and
		\item \label{item:beta-distancerank} for each $2\le \ell\le m$, there is an 
		 $i < \ell$ such that $\mathbf s_{\ell} \in \partial^\circ D_i$,  and the rank of $\mathbf s_{\ell}$ in $(D_i,\mathbf s_i,\mathcal L^{\mathbf s_i})$ is at least $d$.
	\end{enumerate}
\end{definition}

\begin{lemma} \label{lemma:cut:trees}
	Let $\Delta, k\ge 2$, $d \ge 1$, $\beta\in (0,1)$, and let $(T,\mathbf x, \mathcal{L})$ be a layered $k$-tree with $t \ge 2 \Delta^{d} \beta^{-1}$ edges satisfying $\Delta_1(T)\le\Delta$.
	Then $T$ has a $(\beta, d)$-decomposition. 
\end{lemma}

\begin{proof}
We will find the trees $(D_i, \mathbf s_{i})_{1\le i\le m}$ inductively. At the end of each step $j\ge 0$, we will have found trees $D_1, D_2, \dotsc, D_j$ fulfilling properties \ref{item:beta-small}--\ref{item:beta-distancerank} from Definition~\ref{beta-decomp}, with $m$ replaced by $j$. Moreover, there will be a set $\mathbf S_j\subseteq \partial^\circ T$  such that 
\begin{enumerate}[(a)]
	\item \label{item:betapartial-decompos} $E(T)=\bigcupdot_{1\le i\le j} E(D_i)\cupdot\bigcupdot_{\mathbf s\in \mathbf S_j} E(T_\mathbf s)$, 
	\item \label{item:betapartial-distancerank} for each $\mathbf s\in \mathbf S_j\setminus \{\mathbf x\}$ , there is a unique $i\le j$ such that $\mathbf s \in \partial^\circ D_i$,  the rank of $\mathbf s$ in $D_i$ is at least $d$, and  $(V(T_\mathbf s) \setminus \mathbf s)\cap \bigcup_{1\le i\le j}V(D_i) =\emptyset$,
	\item \label{item:betapartial-largetrees} $|E(D_i)|\ge \beta t/(2\Delta^{d})$ for each $1\le i\le j$, and
	\item \label{item:betapartial-largenewroots} $|E(T_{\mathbf s})|\ge \beta t/(2\Delta^{d})$ for each $\mathbf s \in \mathbf S_j$.
\end{enumerate}
Note that \ref{item:betapartial-largetrees} guarantees that we stop  in some step $m\le 2\Delta^{d}/\beta$ with $\mathbf S_m=\emptyset$.
This, together with \ref{item:betapartial-decompos},
ensures \ref{item:beta-notsomany} and \ref{item:beta-edgedecomposition}  hold.

We start the procedure setting $\mathbf S_0=\{\mathbf x\}$, with all properties  trivially fulfilled.	
Now assume we are in step $j\ge 1$.
Choose any $\mathbf s\in \mathbf S_{j-1}$.
By \ref{item:betapartial-largenewroots}, we have $|E(T_{\mathbf s})| \ge \beta t/(2\Delta^{d})$.
If  $|E(T_{\mathbf s})| \leq \beta t$, then set $D_j := T_{\mathbf s}$, $\mathbf s_j := \mathbf s$ and $\mathbf S_{j} := \mathbf S_{j-1} \setminus \{ \mathbf s \}$ and end step $j$.
Otherwise, apply \cref{proposition:cut1} to $(T_{\mathbf s},\mathbf s)$, obtaining a set $F_1$ of edges, and a set $\mathbf S'_1$ of $\mathcal{L}$-layered elements of $\partial^\circ T_{\mathbf s}\subseteq \partial^\circ T$ of rank 2 in $T_{\mathbf s}$ (that is, $F_1$ and $S'_1$ are the sets $F$ and $S$ from the statement of \cref{proposition:cut1}).
Apply \cref{proposition:cut1} to all trees $T_{\mathbf s'}$ with $\mathbf s'\in \mathbf S'_1$, thus generating a set $F_2$ of edges and a set $\mathbf S'_2$, such that each $\mathbf s'\in \mathbf S'_2$ is $\mathcal{L}$-layered and has rank 3 in $T_{\mathbf s}$.
Continue in this manner until reaching a set $\mathbf S'_{d-1}$ of $\mathcal{L}$-layered elements of rank $d$, and set $F:=\bigcup_{1\le i\le d-1} F_i$.
Note that $|\mathbf S'_{d-1}|\le |F|\le \Delta^{d}$ and the edges in $F$ span a $k$-tree $T_F$ rooted at $\mathbf s$. 
Next, for each $\mathbf s'\in \mathbf S'_{d}$, in order, consider the tree $T_{\mathbf s'}$.
If $|E(T_{\mathbf s'})|< \beta t/(2\Delta^{d})$, then add $T_{\mathbf s'}$ to $T_F$ and delete $\mathbf s'$ from $\mathbf S'_d$, and continue to examine the next $\mathbf s' \in \mathbf S'_d$.
At the end of this process, we obtain a tree $B_1\supseteq T_F$ and a set $\mathbf Z_1\subseteq \mathbf S'_d$.
Note that 
\[|E(B_1)|\le |F|+|\mathbf S'_{d-1}| ( \beta t/(2\Delta^{d})) \leq |F|(1+\beta t/(2\Delta^d)) \le \beta t<|E(T_{\mathbf s})|,\]
which implies that $\mathbf Z_1\not=\emptyset$. Moreover, we have that $|E(T_{\mathbf z})| \ge \beta t/(2\Delta^{d})$ for each $\mathbf z \in \mathbf Z_1$. 
	
Let us note here that if $|B_1|\ge \beta t/2$, then we could set $D_j:=	B_1$, $\mathbf{s}_j:=\mathbf s$, and $\mathbf{S}_j=(\mathbf{S}_{j-1}\cup \mathbf{Z}_1)\setminus \{\mathbf{s}\}$, thus completing the inductive step.
So, let us suppose that $|B_1| < \beta t / 2$.
In what follows next we will, gradually, add edges to $B_1$ to make it have size between $\beta t / 2$ and $\beta t$.
To do this, successively, for $i\ge 1$, choose any $\mathbf z\in \mathbf Z_i$ and apply \cref{proposition:cut1} to $T_{\mathbf z}$.
Add the resulting edges given by \cref{proposition:cut1} to $B_i$, obtaining the set $B'_{i} \supseteq B_i$,
and  let $\mathbf S$ be the subset of $\partial T$ from the lemma. For each $\mathbf s'\in \mathbf S$, check whether $|E(T_{\mathbf s'}) |< \beta t/(2\Delta^{d})$, and if this is the case, then add $T_{\mathbf s'}$ to $B'_i$ and delete $\mathbf s'$ from $\mathbf S$.
After processing all $\mathbf s' \in \mathbf S$, this results in a set $\mathbf S'$, and a tree $B_{i+1}$.
Set $\mathbf Z_{i+1}:= (\mathbf Z_i\cup \mathbf S')\setminus \{\mathbf z\}$.
Then $|E(B_{i+1})|\le |E(B_{i})|+\Delta+|\mathbf S|(\beta t/(2\Delta^{d})) \leq |E(B_i)|+\beta t/(\Delta^{d-1})$
and  $|E(T_{\mathbf z})| \ge \beta t/(2\Delta^{d})$ for each $\mathbf z \in \mathbf Z_i$.
		
We continue until we reach the first index $h$ with $|E(B_h)|\ge \beta t/2$ (this must happen at some point, since in each step, at least one edge from $E(T_{\mathbf s})$ is added to $E(B_i)$, and $|E(T_{\mathbf s})| > \beta t$). Then $|E(B_h)|\le\beta t$.
Set $D_j:=B_h$, $\mathbf s_j := \mathbf s$, and set $\mathbf S_j:=(\mathbf S_{j-1}\cup \mathbf Z_h)\setminus \{\mathbf s\}$.
By construction,  \ref{item:betapartial-decompos}--\ref{item:betapartial-largenewroots} and \ref{item:beta-small}--\ref{item:beta-distancerank} from Definition~\ref{beta-decomp} hold for $\mathbf S_j$ and $D_1, \dotsc, D_j$.
\end{proof}	

\section{Tools}\label{section:tools}
In this section, we collect some tools that will be needed for the proof of Theorem~\ref{theorem:spanning-codegree}.

\subsection{The weak hypergraph regularity lemma}

Let $H$ be a $k$-graph and let $V_1, \dotsc, V_k$ be pairwise disjoint  subsets of  $V(H)$.
Let $H[V_1, \dotsc, V_k]$ be the $k$-partite subhypergraph of $H$ induced by all edges that intersect all sets $V_i$.
The \emph{density} of $H[V_1, \dotsc, V_k]$ is defined as \[ d(V_1, \dotsc, V_k) := \frac{e_H(V_1, \dotsc, V_k)}{|V_1| \dotsb |V_k|}, \]
where $e_H(V_1,\dotsc,V_k)$ denotes the number of edges in $H[V_1,\dotsc,V_k]$.
For $\eps, d > 0$, we say a $k$-tuple $(V_1, \dotsc, V_k)$ of pairwise disjoint non-empty subsets of $V(H)$ is \emph{$(\eps, d)$-regular} if \[ |d(W_1, \dotsc, W_k) - d| \leq \eps \]
for all $k$-tuples of subsets $W_i \subseteq V_i$ satisfying $|W_1| \dotsb |W_k| \ge \eps |V_1| \dotsb |V_k|$.
A $k$-tuple $(V_1, \dotsc, V_k)$ will be called \emph{$\eps$-regular} if it is $(\eps,d)$-regular for some $d \ge 0$.

The {\it weak regularity lemma for hypergraphs} ensures that the vertex set of every $k$-graph can be partitioned into a bounded number of clusters, such that almost all $k$-tuples of these clusters are $\eps$-regular. We will use the lemma in the following form (see~\cite[Theorem 9]{KNRS2010}).

\begin{theorem}[Weak Hypergraph Regularity Lemma] \label{theorem:weakregularity}
	Let $k \ge 2$ and let $1/n, 1/T_0 \ll 1/t_0, 1/k, \eps$.
	For every $k$-graph $H$ on $n$ vertices there exists a partition $\{ V_0, V_1, \dotsc, V_t \}$ of $V(H)$ such that
	\begin{enumerate}
		\item \label{item:weakregularity1} $t_0 \leq t \leq T_0$,
		\item \label{item:weakregularity2} $|V_0| \leq \eps n$ and $|V_1| = \dotsb = |V_t|$, and
		\item \label{item:weakregularity3} for all but at most $\eps \binom{t}{k}$ sets $\{ i_1, \dotsc, i_k \} \subseteq [t]$, the $k$-tuple $(V_{i_1}, \dotsc, V_{i_k})$ is $\eps$-regular.
	\end{enumerate}
\end{theorem}
Any partition $\mathcal{P} = \{ V_0, V_1, \dotsc, V_t \}$ of $V(H)$ satisfying \ref{item:weakregularity1}--\ref{item:weakregularity3} will be called an \emph{$\eps$-regular} partition of $H$.
Given $d > 0$, we define the \emph{$d$-reduced $k$-graph $R_d(H)$ of $H$ with respect to $\mathcal{P}$} as follows.
Its vertex set is $[t] = \{1, \dotsc, t \}$, and its edges are the $k$-sets $\{ i_1, \dotsc, i_{k} \}$ such that $d_H(V_{i_1}, \dotsc, V_{i_k}) \ge d$ and $(V_{i_1}, \dotsc, V_{i_k})$ is $\eps$-regular.
We will also refer to $R_d(H)$ as ``the'' \emph{$d$-reduced $k$-graph of $H$}. (Even if $R_d(H)$ depends on the choice of $\mathcal{P}$, we omit explicit reference to $\mathcal{P}$ in the notation for simplicity.)

We will need to find almost-perfect matchings in the reduced $k$-graph. For $k=2$, it is easy to find one using graph regularity, and for $k\ge 3$ its existence may be deduced from Claims 4.4 and 4.5 in~\cite{RRS08}.

\begin{lemma}\label{lemma:matching}
Let $k \ge 2$, $0 < 1/n \ll 1/t \ll \eps\ll 1/k, \gamma, \eta$, and let $d\ll \gamma$.
Let $H$ be a $k$-graph on $n$ vertices with $\delta_{k-1}(H) \ge (1/2 + \gamma)n$.
Let $\mathcal{P} = \{ V_0, V_1, \dotsc, V_t \}$ be an $\eps$-regular partition of $V(H)$. Then the $d$-reduced $k$-graph $R_{d}(H)$ has a matching covering at least $(1 - \eta) t$ vertices.
\end{lemma}

\subsection{Degenerate hypergraphs and extensible edges}

Given $k\ge 2$ and $s \in\mathbb N$, let $K^{(k)}(s)$ denote the complete $k$-partite $k$-graph with each class of size $s$.
To be precise, $V(K^{(k)}(s))$ is partitioned in $k$ clusters $V_1, \dotsc, V_k$ of size $s$ each, and its edges are precisely the $k$-sets which intersect each $V_i$ exactly once.
The following result, due to Erd\H os~\cite{Erdos1964}, is a hypergraph version of the classical K\H ov\'ari--S\'os--Tur\'an theorem~\cite{KST}.

\begin{lemma}\label{lemma:turanzero}
Let $1/n\ll 1/k, 1/s, \eps$.
	Let $H$ be a $k$-graph with $n$ vertices and at least $\eps n^{k}$ edges.
	Then $H$ contains a copy of $K^{(k)}(s)$ as a subgraph.
\end{lemma}
Note that for any $k \ge 2$,  the complete $k$-partite $k$-graph $\kdos$  has $2k$ vertices and $2^k$ edges.
Given a $k$-graph $H$ and an edge $e \in H$, let $\degkdos_H(e)$ be the number of copies of $\kdos$ in $H$ in which $e$ participates. Note that $\degkdos_H(e)\le \binom{n-k}{k}$ always.

\begin{definition}[$\theta$-extensible edge]
	Given an $n$-vertex $k$-graph $H$ and $\theta > 0$, we say an edge $e \in H$ is \emph{$\theta$-extensible} if $\degkdos_H(e) \ge \theta \binom{n-k}{k}$.
\end{definition}

Extensible edges will be useful in our embedding of $k$-trees.
We show that in an appropriately dense $k$-graph most edges are extensible.

\begin{lemma}\label{lem:extensible}
	Let $1/n, \theta \ll \eps, 1/k$.
	In any $k$-graph on $n$ vertices, all but at most $\eps \binom{n}{k}$ edges are $\theta$-extensible.
\end{lemma}

\begin{proof}
	Lemma~\ref{lemma:turanzero} implies that the Tur\'an density of $\kdos$ is zero.
	Hence, by standard supersaturation arguments~\cite[Lemma 2.1]{Keevash2011}, there exist $n_0$ and $\alpha > 0$ such that every $k$-graph on $n \ge n_0$ vertices and at least $\eps \binom{n}{k}$ edges has at least $\alpha \binom{n}{2k}$ copies of $\kdos$. To prove the lemma we shall use $n \ge n_0$ and $\theta \leq (k!)^2 2^k \alpha / (2k!)$.
	
	Indeed, let $H$ be any $k$-graph on $n$ vertices and let $H' \subseteq H$ be the $k$-graph formed by the non-$\theta$-extensible edges of $H$.
	To reach a contradiction, suppose that $H'$ has at least $\eps \binom{n}{k}$ edges.
	By the choice of $n_0$ and $\alpha$, we know that $H'$ contains at least $\alpha \binom{n}{2k}$ copies of $\kdos$.
	Note that~$H'$ has at most $\binom{n}{k}$ edges and recall that each copy of $\kdos$ has $2^k$ edges.
	Therefore, a double-counting argument shows that some edge $e$ in $H'$  participates in at least $2^k \alpha \binom{n}{2k} / \binom{n}{k} \ge \theta \binom{n-k}{k}$ copies of $\kdos$.
	So, $\degkdos_H(e) \ge \theta \binom{n-k}{k}$, or in other words, $e$ is a $\theta$-extensible edge of $H'$, and therefore of $H$.
	This contradicts the definition of $H'$.
\end{proof}

\subsection{Reservoirs}
Let $H$ be a $k$-graph $H$, and let $F\subseteq \partial H$.
Recall that $\deg_H(F)$ denotes the joint degree of $F$, as defined in \eqref{equation:jointdegree}.
For $U \subseteq V(H)$, we let
\begin{align}
\deg_H(F,U) = \big| \{v\in U : \text{$f \cup \{ v \}\in H$ for each $f\in F$}\}\big|.
\end{align}
Similarly, recalling that previously we defined $\degkdos_H(e)$  as the number of copies of~$\kdos$ in the $k$-graph~$H$ that contain the edge $e\in E(H)$, we define  $\degkdos_H(e, U)$ as the number of copies of $\kdos$ in $H[U\cup e]$ that contain $e$.

\begin{definition}[Reservoir]
Let $H$ be a $k$-graph on $n$ vertices, and let $\gamma, \mu > 0$.
We say that a set $U \subseteq V(H)$ is a \emph{$(\gamma, \mu, h)$-reservoir for $H$} if
\begin{enumerate}
	\item\label{res:i} $|U| = (\gamma \pm \mu)n$,
	\item\label{res:ii} for every  $F \subseteq \partial (H)$ with $|F| \leq h$ we have $\deg_H(F, U) \ge (\deg_H(F)/n - \mu) |U|$, and
	\item\label{res:iii} for every $e \in H$, we have $\degkdos_H(e, U) \ge (\degkdos_H(e)/\binom{n-k}{k} - \mu) \binom{|U|-k}{k}$.
\end{enumerate}
\end{definition}

\begin{lemma}[Reservoir Lemma]\label{lem:reservoir}
	Let $1/n \ll \mu \ll \gamma, 1/h \le 1$.
	Then every $k$-graph $H$ on $n$ vertices has a $(\gamma, \mu, h)$-reservoir.
\end{lemma}

The proof of Lemma~\ref{lem:reservoir} is probabilistic, and
we will use the following standard concentration inequalities for random variables.

\begin{theorem}[Chernoff's inequality {\cite[Theorem 2.1]{JLR2000}}]\label{theorem:chernoff}
	Let $0<\alpha<3 \expectation[X] / 2$ and $X \sim \emph{\text{Bin}}(n,p)$ be a binomial random variable. 
	Then $\Pr\left(|X - \expectation[X] | > \alpha \right) < 2\exp(-\alpha^2/(3 \expectation[X] ))$.
\end{theorem}

\begin{theorem}[McDiarmid's inequality {\cite{McDiarmid1989}}] \label{theorem:mcdiarmid} 
	Suppose $X_1, \dotsc, X_m$ are independent Bernoulli random variables and $b_1, \dotsc, b_m \in [0, B]$.
	Suppose $X$ is a real-valued random variable determined by $X_1, \dotsc, X_m$ such that changing the outcome of $X_i$ changes $X$ by at most $b_i$ for all $1 \leq i \leq m$.
	Then, for all $\lambda > 0$, we have
	\[ \Pr \left(|X - \expectation[X] | > \lambda \right) \leq 2 \exp \left( - \frac{2 \lambda^2}{B \sum_{i=1}^m b_i} \right).  \]
\end{theorem}

\begin{proof}[Proof of Lemma~\ref{lem:reservoir}]
	Choose a set $U \subseteq V(H)$ randomly by independently including each vertex of $V(H)$ with probability $p = \gamma$.
	With non-zero probability $U$ will satisfy all of the properties \ref{res:i}--\ref{res:iii} simultaneously, which shows the desired set $U$ exists.
	
	Indeed, $\expectation[|U|] = p n=\gamma n$.
	Thus, using Chernoff's inequality (Theorem~\ref{theorem:chernoff}) with $\alpha = (n^{1/3} \gamma)^{-1}$ we get that $|U| = \gamma n \pm n^{2/3}$ fails to hold with probability at most $2 \exp(- n^{1/3} / (3 \gamma))$.
	Since $n$ is sufficiently large, $|U| = \gamma n \pm n^{2/3}$ holds with probability at least $1 - 1/n$ and we will assume those bounds on $|U|$ from now on.
	Note also this implies~\ref{res:i} holds for $U$.
	
	Now we verify~\ref{res:ii} holds.
	Let $F \subseteq \partial H$ of size at most $h$ and note that $\expectation[\deg_H(F, U)] = p \deg_H(F)$.
	If $\deg_H(F, U) < \mu n$ then there is nothing to show, so we assume otherwise.
	In particular, $\expectation[\deg_H(F, U)] \ge \gamma \mu n$.
	If $\deg_H(F, U) < (\deg_H(F)/n - \mu)|U|$, then $\deg_H(F, U) \leq (\deg_H(F)/n - \mu)(\gamma n + n^{2/3}) \leq \expectation[\deg_H(F, U)] - \mu \gamma n / 2$ since $n$ is large.
	Apply Chernoff's inequality with $\alpha = \mu \gamma n / 2 \leq 3 \expectation[\deg_H(F, U)] / 2$ to get
	\begin{align*}
	\prob[ \deg_H(F, U) < (\deg_H(F)/n - \mu)|U| ]
	& \leq \prob[ |\deg_H(F, U) - \expectation[\deg_H(F, U)] | > \mu \gamma n / 2  ] \\
	& \leq 2 \exp \left(- \frac{(\mu \gamma)^2 n}{12} \right).
	\end{align*}
	Since $|\partial H| \leq n^{k-1}$ and $|F| \leq h$, there are at most $n^{h(k-1)}$ possible choices for $F$.
	Then a union bound shows that~\ref{res:ii} fails to hold with probability at most $2 n^{h(k-1)} \exp(-(\beta \gamma)^2 n/12) < 1/n$,
	where the last inequality holds since $n$ is large.
	
	To see~\ref{res:iii}, fix an edge $e \in H$.
	If $\degkdos_H(e) < \mu \binom{n-k}{k}$ then there is nothing to show, so assume otherwise.
	Let $X = \degkdos_H(e, U)$ and note that $\expectation[X] = p^k \degkdos_H(e)$.
	Since $|U| = p n \pm n^{2/3}$ and $1/n \ll 1/k$, we have $\binom{|U|-k}{k} = (1 + o(1)) \gamma^k \binom{n}{k}$.
	The presence of a vertex in $U$ can affect $X$ by at most $n^{k-1}$.
	Thus, we can apply McDiarmid's inequality (Theorem~\ref{theorem:mcdiarmid}) with $m = n$ and $B = b_i = n^{k-1}$ for all $1 \leq i \leq n$ to see that
	\begin{align*}
	\prob\left[ \deg^K_H(e, U) < (\deg^K_H(e)/\binom{n-k}{k} - \mu)\binom{|U|-k}{k} \right]
	& \leq \prob\left[ X < \expectation[X] - \mu \gamma^k \binom{n}{k} / 2 \right] \\
	& \leq 2 \exp \left( - \mu^2 \gamma^{2k} n^{2k} / ( 2 k^{2k} n^{2k-1} ) \right),
	\end{align*}
	where in the last inequality we also used $\binom{n}{k} \geq (n/k)^k$.
	The last term is less than $1/{n^{k+1}}$ since $n$ is sufficiently large.
	Since there are at most $n^k$ edges in $H$, we see~\ref{res:iii} fails with probability at most $1/n$, as required.
\end{proof}

\section{Connections}\label{section:connecting}

For this section, the following notion will be essential.
We say $k$-graph  $H$ is \emph{$\ell$-large} if every two $f,f'\in\partial H$ have at least $\ell$ common neighbours.
For instance, $k$-graphs $H$ on $n$ vertices with minimum codegree at least $(1/2+\gamma)n$ are $2 \gamma n$-large,
and $(\varrho,2,\eps)$-typical graphs $H$ on $n$ vertices are $(\varrho^2 - \eps)n$-large.
Note that a $k$-graph can be $\ell$-large and have \emph{isolated vertices} (i.e. vertices which do not lie in any edge), since the property only says something about tuples in $\partial H$.
For $U\subseteq V(H)$, we say that $H$ is $(\ell, U)$-large if every two $f,f'\in\partial H$ have at least $\ell$ common neighbours in $U$ (Later, $U$ will be a reservoir).

The main result of the current section is \cref{lemma:strengthenedconnectinglemma}, which essentially says that in any $\Omega(n)$-large $k$-graph $H$ we can connect any two elements of $\partial H$ by a walk (actually, many such walks) of length exactly $\ell$, where $\ell$ is a number only depending on~$k$.
We observe that \cref{lemma:strengthenedconnectinglemma} can be seen as a strengthening of the `Connecting lemma' of R\"odl, Ruci\'nski and Szemer\'edi~\cite[Lemma 2.4]{RRS08}. For our approach, however, it is crucial that we can control the precise length of the walk (instead of only having an upper bound). We will also control the number of internal vertices of the walk, and the order of the elements $f,f'$ we are connecting.

Recall that $\partial^\circ H$ denotes the ordered shadow of $H$, and that the interior $\int{W}$ of a walk $W$ corresponds to the set of vertices $V(W) \setminus (\sta{W} \cup \ter{W})$.

\begin{lemma}[Connecting Lemma]\label{lemma:strengthenedconnectinglemma}
   For integers $k, \ell$ with $k \ge 2$ and $\ell \ge (2k+1)\lfloor k /2 \rfloor + 2k$, there exists $q \leq \ell$ with the following property.
   Let $1/n \ll \gamma \ll 1/k, 1/\ell$. Let $H$ be a $k$-graph on $n$ vertices which is $(2 \gamma n, U)$-large for some $U \subseteq V(H)$,
   and let $\mathbf x, \mathbf x' \in \partial^\circ H$.
   Then there are at least $(\gamma n)^{q}$ many walks going from $\mathbf x$ to $\mathbf x'$,
   each of length $\ell$ and with $q$ internal vertices all from $U \setminus (\mathbf x \cup \mathbf x')$.
\end{lemma}
To prove \cref{lemma:strengthenedconnectinglemma}, it will be useful to find a walk which swaps the order in which the vertices of a given edge appear.
This will be achieved in the following two short lemmas.
The first lemma effectively swaps the position of two vertices in an ordered edge ($b_j$ and $b_{k-j+1}$ in the statement below).

\begin{lemma}\label{lemma:swap1}	
	Let $k\ge 2$ and let $(b_1, ..., b_k)$ be an ordered edge in a $k$-graph $H$.
	Let $1\le j\le \lfloor k/2\rfloor$, and let $u \in N_H( b_1, \dotsc, b_{j-1}, b_{j+1}, \dotsc, b_k)\cap N_H( b_1, \dotsc, b_{k-j}, b_{k-j+2}, \dotsc, b_k )$.
	Then $H[\bigcup_{1\le i\le k}\{b_i\}\cup \{u\}]$ contains a walk of length $2k+1$ from $(b_1, \dotsc, b_k)$ to
	\[(b_1,\dotsc, b_{j-1}, b_{k-j+1}, b_{j+1}, b_{j+2}, \dotsc, b_{k-j}, b_j, b_{k-j+2}, \dotsc, b_k).\]
\end{lemma}

\begin{proof}
	It suffices to consider the walk (each $k$ consecutive vertices form an edge)
		\[ b_1 \dotsb b_k b_1 \dotsb b_{j-1} u b_{j+1} \dotsb b_{k-j} b_j b_{k-j+2} \dotsb b_k b_1 \dotsb b_{j-1} b_{k-j+1} b_{j+1} \dotsb b_{k-j} b_j b_{k-j+2} \dotsb b_k,\]
	which uses $3k$ vertices and thus has length $2k+1$.
\end{proof}	

\begin{lemma}\label{lemma:swap}	
	Let $k\ge 2$ and let $(a_1, \dotsc, a_k)$ be an edge in a $k$-graph $H$.
	For each $j\le \lfloor k/2\rfloor$, let $u_j\in N_H( a_1, \dotsc, a_{j-1}, a_{j+1}, \dotsc a_k )\cap N_H(a_1, \dotsc, a_{k-j}, a_{k-j+2}, \dotsc a_k )$.
	Then $H[\bigcup_{1\le i\le k}\{a_i\}\cup\bigcup_{1\le j\le \lfloor k/2\rfloor} \{u_j\}]$ contains a walk of length $\lfloor k/2\rfloor (2k+1)$ from $(a_1, \dotsc, a_k)$ to  $(a_k, \dotsc, a_1)$.	
\end{lemma}
\begin{proof}
	We use Lemma~\ref{lemma:swap1} successively, for all $j\le \lfloor k/2\rfloor$, thus swapping the vertices $a_j$ and $a_{k-j+1}$ in the walk until we reach $(a_k, \dots, a_1)$.
	This gives a walk of length $\lfloor k/2\rfloor (2k+1)$.
\end{proof}

Now we can prove Lemma~\ref{lemma:strengthenedconnectinglemma}.
\begin{proof}[Proof of Lemma~\ref{lemma:strengthenedconnectinglemma}]
Let $\mathbf x=(x_1,\dotsc,x_{k-1})$ and $\mathbf x'=(x'_1,\dotsc,x'_{k-1})$, and set $\ell_0 := (2k+1)\lfloor k /2 \rfloor + 2k$.
	Let $\mathbf x'_R = (x'_{k-1}, \dotsc, x'_1)$ be the reverse of the tuple $\mathbf x'$.
	Greedily construct a tight path $P_1$ of length $\ell-\ell_0$,
	starting at $\mathbf x'_R$, ending at some $\mathbf y=(y_1, \dotsc, y_{k-1})$,
	and using only vertices in 
	$U \setminus \mathbf x$.
	This can be done: as every $(k-1)$-set has at least $ 2\gamma n$ neighbours in $U$,
	at each step we need to avoid at most $|\mathbf x \cup \mathbf x'| - (\ell - \ell_0) \leq 2k+\ell$ vertices, and
	so $1/n \ll 1/k, 1/\ell$ implies there are at least $\gamma n$ choices at each step.
	Set	 $Z:=V(P_1) \setminus \mathbf x$ and $q: = k+\lfloor k/2 \rfloor +\ell-\ell_0= k+\lfloor k/2 \rfloor + |Z|$.
    
    Building on $P_1$, we will first construct a single walk
    as promised in the lemma, afterwards we will estimate in how many ways this can be done. 
	Let $a_1\in (N_H(\mathbf x)\cap N_H(\mathbf x'')\cap U) \setminus (Z \cup \mathbf x \cup \mathbf x')$.
	Having defined $a_1, \dotsc, a_j$ for some $j \in [k-1]$, we choose an arbitrary unused vertex 
\[a_{j+1}\in N_H(x_{j+1}, \dotsc, x_{k-1}, a_1, \dotsc, a_j) \cap N_H(y_{j+1}, \dotsc, y_{k-1}, a_1, \dotsc, a_j)\cap U.\]
	Clearly $P_2=x_1 \dotsb x_{k-1} a_1 \dotsb a_{k-1} a_k$ and $P_3=a_k a_{k-1} \dotsb a_1 y_{k-1} \dotsb y_{1}$ are tight paths. 
	Applying Lemma~\ref{lemma:swap}, we find a walk $P_4$ which goes from $(a_1, \dotsc, a_{k})$ to $(a_k, a_{k-1}, \dots, a_1)$ only occupying unused vertices $u_1, \dotsc, u_{\lfloor k/2\rfloor}$ from $U$ (the  vertices $u_j$ exist because $H$ is $(2 \gamma n, U)$-large).Concatenating the walks $P_2$, $P_4$, $P_3$ and $P_1$ (the latter traversed in reverse order) gives a walk $W$ from $\mathbf x$ to $\mathbf x'$.
	Set $Q = Z\cup \{ a_i : 1 \leq i \leq k \} \cup \{ u_j : 1 \leq j \leq \lfloor k/2 \rfloor \}$.
	Then $Q = \int{W}$ and $|Q|  = q$, and the length of $W$ is $|E(P)|=k+2k\lfloor k/2\rfloor + \lfloor k/2\rfloor+k+\ell-\ell_0=\ell$,
	so $W$ is a walk which satisfies the required properties.
	
\begin{figure}[h!]
	\centering
	\begin{tikzpicture}[scale=.8]
	
	\draw node (x1l) at (.5,3.4)  {$x_1$};
	\draw node  (x2l) at (-1.6,1.5) {$x_2$};
	\draw node (y2l) at (4,3.4) {$y_2$};
	\draw node (y1l)  at (6,1.5){$y_1$};
	\draw node  (a1l) at (2,2.5){ $a_1$};
	\draw node  (a2l) at (.7,-.4){$a_2$};
	\draw node  (a3l) at (2.8,-2.1){$a_3$};
	\draw node (a4l) at (3.7,.3){$a_4$};

		\tikzstyle{every node}=[circle, draw, fill, inner sep=0pt, minimum width=4pt];
		
	\draw node (x1) at (.5,3)  {};
	\draw node  (x2) at (-1,1.5) {};
	\draw node (y2) at (4,3) {};
	\draw node (y1)  at (5.5,1.5){};
	\draw node  (a1) at (2,2){};
	\draw node  (a2) at (1,0){};
	\draw node  (a3) at (2.3,-2){};
	\draw node (a4) at (3.2,.3){};

		
		\triple{(x2)}{(x1)}{(a1)}{6pt}{1pt}{DarkDesaturatedBlue,opacity=0.8}{VerySoftBlue,opacity=0.2};
		\triple{(x2)}{(a1)}{(a2)}{6pt}{1pt}{DarkDesaturatedBlue,opacity=0.8}{VerySoftBlue,opacity=0.2};
		\triple{(y2)}{(y1)}{(a1)}{6pt}{1pt}{DarkDesaturatedBlue,opacity=0.8}{VerySoftBlue,opacity=0.2};
		\triple{(y2)}{(a2)}{(a1)}{6pt}{1pt}{DarkDesaturatedBlue,opacity=0.8}{VerySoftBlue,opacity=0.2};

		\triple{(a3)}{(a2)}{(a1)}{6pt}{1pt}{DarkDesaturatedBlue,opacity=0.8}{VerySoftBlue,opacity=0.2};
		
		\triple{(a2)}{(a1)}{(a4)}{6pt}{1pt}{DarkDesaturatedBlue,opacity=0.8}{VerySoftBlue,opacity=0.2};
		
		\triple{(a3)}{(a2)}{(a4)}{6pt}{1pt}{DarkDesaturatedBlue,opacity=0.8}{VerySoftBlue,opacity=0.2};

	\end{tikzpicture}
	\caption{This figure shows how to construct the walk $x_1x_2a_1a_2a_4a_3a_2a_1y_2y_1$ connecting $(x_1,x_2)$ with $(y_2,y_1)$ in the $3$-uniform case. }
\end{figure}
Note that by construction, each vertex in the interior of $P$ is chosen as an arbitrary unused vertex in the neighbourhood of one or two $(k-1)$-sets.	
	Since $H$ is $(2\gamma n, U)$-large, the common neighbourhoods in $U$ have size at least $2\gamma n$, and thus (here we use $1/n \ll 1/k, 1/\ell$) in every step we have at least $2 \gamma n - q \ge \gamma n$ possible choices in $U$.
	Thus by the previous discussion there are at least $(\gamma n)^q$ many different walks with the required properties.
\end{proof}

\section{Embedding large hypertrees}\label{section:embeddinglemma}

In this section, we  prove an embedding result (Lemma~\ref{lem:embedding}) that will allow us to embed a small rooted bounded degree tree into a dense $k$-partite graph, while at the same time controlling the location of the root and the bulk of the tree quite accurately.

Given positive integers $\ell_1\le \ell_2\le m$, and given a layering $\mathcal L=(L_1, \dotsc, L_{m})$ of a rooted $k$-tree $(T, \mathbf x)$,
we say that $V_{[\ell_1,\ell_2]}(\mathcal L):=\bigcup_{\ell_1\le i\le \ell_2} L_i$ is the {\em $[\ell_1, \ell_2]$-interval} of $T$. If the layering is clear from the context, we just write $V_{[\ell_1,\ell_2]}$.
If, moreover, $|L_i|\le M$ for each $\ell_1\le i\le\ell_2$, we say $V_{[\ell_1,\ell_2]}$ that is {\em $M$-bounded}.
We write $\mathbf x\in \partial^\circ T[V_{[\ell_1,\ell_2]}]$ if there is a $j\in [\ell_1, \ell_2-k+1]$ such that $\mathbf x=(x_1, ..., x_{k-1})\in \partial^\circ T$ with $x_i\in L_{j+i-1}$ for all $i\in [k-1]$.

We can now state the Embedding Lemma.

\begin{lemma}[Embedding Lemma]\label{lem:embedding}
	Let $\Delta, k \ge 2$, let $\ell \ge \lfloor k/2 \rfloor (2k+1) + 2k$ and let $1/n, \mu \ll \beta, \theta \ll 1/k, 1/\Delta, c, \gamma, d$.
	Let $H$ be a
	$\gamma n$-large 
	$k$-graph on $n$ vertices with a $(\gamma,\mu,2)$-reservoir $R$.
	Let $W_1,\dotsc,W_k\subseteq V(H)\setminus R$ be all pairwise disjoint, and  such that $d(W_1,\dotsc,W_k)\ge d$ and $|W_i|\ge cn$ for each $i\in [k]$.
	Let $(T,\mathbf x, \mathcal{L})$ be a layered $k$-tree on at most $\beta n$ vertices with $\mathcal{L} = (L_1, \dotsc, L_m)$ and $\Delta_1(T)\leq \Delta$.
	Then, for any $\theta$-extensible edge $e\in H$,
	$f\subseteq e$ of size $k-1$ and any ordering $\mathbf f$ of $f$, there exists an embedding $\varphi:V(T)\to f \cup R\cup W_1\cup\dotsb \cup W_k$ such that
	\begin{enumerate}[\upshape (E1)]
		\item\label{keylemma:1} $\varphi(\mathbf x)=\mathbf f$,
		\item\label{keylemma:2} $\phi^{-1}( R\cup f ) =\bigcup_{i=1}^\ell L_i$,
		\item\label{keylemma:3} $\phi(V_{[\ell +1, m]})\subseteq W_1\cup \dots\cup W_k$, with
	$|\varphi^{-1}(W_1)|\ge \dotsb \ge |\varphi^{-1}(W_k)|$,
		and
		\item\label{keylemma:4} $\varphi(e')$ is  $\theta$-extensible, for each $e' \in E(T[V_{[\ell +1, m]}])$.
	\end{enumerate}
\end{lemma} 

Roughly speaking, \cref{lem:embedding} states that we can embed any sufficiently small tree of bounded degree into any large $k$-graph using vertices only from a given extensible edge, a given  reservoir,  and a given dense $k$-partite subgraph.
Property~\ref{keylemma:1} says that we can map the root of the tree into any $(k-1)$-subset of an extensible edge of the host graph, in any order.
Property~\ref{keylemma:2} says that we only embed a fixed number of layers of $T$ in the reservoir (and thus only use a constant number of its vertices).
Property~\ref{keylemma:3} ensures the remaining levels of the tree are embedded into the $k$-partite subgraph $(W_1,\dotsc,W_k)$ and, moreover, we can decide which of these receives most (second most, etc) of the vertices from $T$.
Finally, Property~\ref{keylemma:4} states that all the edges from $T[V_{[\ell +1, m]}]$ are mapped to $\theta$-extensible edges of $H$.

In \cref{subsection:toolsforembedding}, we will gather some tools for the proof of \cref{lem:embedding},
which is postponed to \cref{subsection:proofofembeddinglemma}.

\subsection{Tools for embedding} \label{subsection:toolsforembedding}

 We will use a hypergraph version of the fact that every graph of average degree at least $d$ contains a subgraph of minimum degree at least $d/2$.
 This is achieved by the next lemma.

\begin{proposition}[Cleaning the graph]\label{proposition:mincodeg}
	Let $k,m \ge 2$, let $d\in (0,1)$ and
	let $H$ be a $k$-partite $k$-graph, with partition classes each of size at most $m$.
	If $H$ has at least $d m^k$ edges,
	then $H$ has a non-empty subgraph $H'$ such that $\deg_{H'}(f) \ge dm/k$ for every $f \in \partial H'$.
\end{proposition}

\begin{proof}
	Starting with $H_1:=H$, proceed as follows for $i\ge 1$.
	If there is an $f\in \partial H_i$ with $\deg_{H_i}(f)< dm/k$,
	then obtain $H_{i+1}$ from $H$ by removing all edges containing $f$.
	If there is no such $f$, we stop and set $H':=H_i$.
	It only remains to show that $H'\neq\emptyset$. For this, observe that the total number of deleted edges is less than $|\partial H| dm/k\le d m^k$.
\end{proof}

In the proof of \cref{lem:embedding}, we will use \cref{proposition:mincodeg} to clean $H[W_1,\dots,W_k]$ 
to 
 find a subgraph $H'\subseteq H''$	such that every $f \in \partial H'$ has large codegree.
The next lemma states that in such a $k$-graph, one can extend any (correctly located) partial embedding of a $k$-tree to a larger $k$-tree. 
The proof proceeds by mapping the remaining vertices successively, using the codegree condition. 

\begin{proposition}[Extending a partial tree embedding] \label{proposition:emb extension}
	Let $\Delta,k,m, n\in \mathbb N$ with $k \ge 2$,
	and let $\delta,\beta > 0$ with $\beta \le \delta /2$. Let $H$ be a $k$-graph and let $W_1, \dotsc, W_k\subseteq V(H)$ be pairwise disjoint, and  of size at most $n$ each.
	Let $H' \subseteq H[W_1, \dotsc, W_k]$ such that for each $f \in \partial H'$ we have	$\deg_{H'}(f) \ge \delta n$.
	Let $(T,\mathbf x,\mathcal L)$ be a layered $k$-tree with $|V(T)|\le \beta n$,
	$\Delta_1(T) \leq \Delta$ and $\mathcal L=(L_1,  \dotsc, L_{m})$.
	Let $1\le\ell\le m-k+1$, and suppose there is an embedding $\varphi_1$ of $T_1 := T[V_{[1,\ell + k - 1]}]$ in $H$ such that
	\begin{enumerate}
		\item $\phi_1 (L_{\ell + i})\subseteq W_i$ for every $i \in [k-1]$, and
		\item if $f\in\partial T [V_{[\ell + 1,\ell + k - 1]} ]$ then $\phi_1(f)$ has codegree at least $\delta m$ in $H'$.
	\end{enumerate}
	Then there is an embedding $\varphi$ of $T$ which extends $\varphi_1$, such that for each $i \in [m - \ell]$, $\phi(L_{\ell + i})\subseteq W_j$ if and only if $j \equiv i \bmod k$.
\end{proposition}

\begin{proof}
	In each $W_i$, at most $|V(T)|\le\beta n \le \delta n / 2$ vertices are used by the partial embedding $\varphi_1$,
	and at most $\delta n / 2$ new vertices need to be embedded in each $W_i$.
	Because of our condition on the codegrees of $H'$,
	we can extend $\varphi_1$ greedily, embedding the vertices of $V(T) \setminus V(T_1)$ one by one, following any valid ordering, and choosing an unused vertex in each step.
\end{proof}

The next lemma (Lemma~\ref{lemma:embeddingtrunk}) is designed to find an embedding of a short sequence of consecutive layers of a layered $k$-tree,
while fixing the location of the initial and final segments.
Its output will be the input for \cref{proposition:emb extension}, which is then used to prove Lemma~\ref{lem:embedding}. 

\begin{lemma}[Embedding the trunk of a tree] \label{lemma:embeddingtrunk}
	Let $1/n \ll 1/k, 1/q, 1/\ell, 1/M, \delta, \alpha$ and $2\le k \le (\ell-1)/2$.
	Let $\mathcal{L} = (L_1, \dotsc, L_{m})$ be a layering of a rooted $k$-tree $(T, \mathbf x)$, and assume $V_{[t, t+\ell]}$ is $M$-bounded. \\
	If $T_I \subseteq T[V_{[t, t+\ell]}]$, $H$ is a $k$-graph on $n$ vertices, $U \subseteq V(H)$ and $\mathbf F_1, \mathbf F_2 \subseteq \partial^\circ H$ are such that
	\begin{enumerate}[\rm (I)]
		\item $|\mathbf F_1|, |\mathbf F_2| \ge \delta n^{k-1}$, and \label{item:embedding-manyends}
		\item for every $\mathbf v_1 \in \mathbf F_1$ and $\mathbf v_2 \in \mathbf F_2$ there are at least $\alpha n^q$ many walks going from $\mathbf v_1$ to $\mathbf v_2$, each of length $\ell-k+1$, each with $q$ internal vertices all from $U\setminus (\mathbf v_1\cup \mathbf v_2)$, \label{item:embedding-manywalks}
	\end{enumerate}
	then there exists an embedding $\phi: V(T_I) \rightarrow V(H)$ such that
	\begin{enumerate}
		\item $\phi(\mathbf v)\in  \mathbf F_1$ for each $\mathbf v\in \partial^\circ T_I[V_{[t,t+k-2]}]$,
		\item $\phi(\mathbf v)\in  \mathbf F_2$ for each $\mathbf v\in \partial^\circ T_I[V_{[t+\ell - k + 2,t+\ell]}]$, and
		\item $\phi(v) \in U$ for each $v \in V_{[t+k-1, t+\ell-k+1]}$.
	\end{enumerate}	
\end{lemma}

The proof of Lemma~\ref{lemma:embeddingtrunk} can be summarised as follows.
In a first step, we define an auxiliary hypergraph $\mathcal{H}'$ whose edges correspond to the interior vertices of a walk from some $\mathbf v_1 \in \mathbf F_1$ to some $\mathbf v_2 \in \mathbf F_2$;
our assumptions will ensure that $\mathcal{H}'$ is sufficiently dense.
Secondly, we discard some edges of $\mathcal{H}'$ to ensure that the remaining hypergraph $\mathcal{H}$ encodes only walks which use $q$ distinct vertices and also use repeated vertices in precisely the same positions of the walk.
This is done by defining a $q$-vertex-colouring and only keeping the (edges which correspond to) walks which are $q$-coloured according to this colouring, in such a way that the colouring codifies the order of each walk.
In a final step, we use supersaturation in $\mathcal{H}$ in order to find a copy of a large complete multipartite subgraph $\mathcal K$ of $\mathcal{H}$. 
Now, as the walks were encoded in the colouring, we can use $\mathcal K$ to embed $T_I$ into the walks corresponding to the edges of $\mathcal K$.

\begin{proof}[Proof of Lemma~\ref{lemma:embeddingtrunk}]
	Define $k_1 = q + 2k-2$ and $k_2 = \ell - k + 1$.
	
	\medskip
	\noindent \emph{Step 1: Defining an auxiliary hypergraph.}	
	We begin by defining an auxiliary $k_1$-graph $\mathcal{H}'$.
	The vertices of $\mathcal H'$ are the vertices of $H$. The vertex set of a walk $W$ of length $k_2$ in $H$ is declared an edge of $\mathcal H'$ if there exist $\mathbf v_1, \mathbf v_2 \in \partial^\circ H$ such that the following conditions hold:
	\begin{itemize}
		\item $\sta{W}=\mathbf v_1 \in \mathbf F_1$ and $\ter{W}=\mathbf v_2\in \mathbf F_2$,

		\item $\int{W} \subseteq U$ and $|\int{W}| = q$, and

		\item $\sta{W}$, $\ter{W}$ and $\int{W}$ (viewed as sets) are pairwise disjoint.
	\end{itemize}
	We claim that 
	\begin{equation}\label{mathcalH'edges}
	\text{$\mathcal{H}'$ has at least $ \frac{\delta^2 \alpha}{2 \ell!} n^{k_1}$ edges.}
	\end{equation}
	Indeed, by \ref{item:embedding-manyends} there are at least $|\mathbf F_1| \ge \delta n^{k-1}$ possible choices for $\mathbf v_1$,
	which will correspond to the start of a walk $W$ defining an edge of $\mathcal{H}'$.
	Each such $\mathbf v_1$ intersects at most $(k-1)n^{k-2}$ elements of $\partial H$.
	So, as $|\mathbf F_2| \ge \delta n^{k-1}$ and $n$ is large, there are at least $|\mathbf F_2| - (k-1)n^{k-2} \ge \delta n^{k-1}/2$ ways to select an end $\mathbf v_2 \in \mathbf F_2$ disjoint from $\mathbf v_1$.
	Having chosen $\mathbf v_1$ and $\mathbf v_2$, by \ref{item:embedding-manywalks} there are at least $\alpha n^q$ many walks $W$ going from $\mathbf v_1$ to $\mathbf v_2$ that could define an edge of $\mathcal{H}'$.
	Having fixed $\mathbf v_1$ and $\mathbf v_2$, since all the given walks $W$ have length $k_2$, the set $\int{W}$ could coincide for at most $k_2! < \ell!$ many of the given walks,
	and thus at most $\ell!$ different walks from $\mathbf v_1$ to $\mathbf v_2$ yield the same edge of $\mathcal{H}'$. Thus the number of edges in $\mathcal{H}'$ is at least $\delta n^{k-1} \times (\delta  n^{k-1} / 2) \times \alpha n^q  \times (\ell!)^{-1}$, as claimed.
	\medskip
	
	\noindent \emph{Step 2: Cleaning the auxiliary hypergraph.}	
	We now define a colouring $c: V(H) \rightarrow \{1, \dotsc, k_1\}$ by choosing a colour for each vertex independently and uniformly at random.
	Let $\mathcal{H}^c \subseteq \mathcal{H}'$ be spanned by all edges $X \in \mathcal{H}'$ whose corresponding walk $W$ has the following properties:
	\begin{enumerate}[(a)]
		\item \label{item:embedding-cleaning-sta} if $\sta{W}=(x_1, \dotsc, x_{k-1})$, then $c(x_i)=i$ for all $i \leq k-1$,
		\item if $\ter{W}=(y_1, \dotsc, y_{k-1})$, then $c(y_i)=k_1 - k + i$ for all $i \leq k-1$, and
		\item \label{item:embedding-cleaning-int} no two vertices of $\int{W}$ have the same colour.
	\end{enumerate}
	For a fixed $X \in \mathcal{H}'$, the probability of belonging to $\mathcal{H}^c$ is at least $k_1^{-k_1}$.
	So the expected size of~$\mathcal{H}^c$ is at least $|E(\mathcal{H}')|/k_1^{k_1} \ge \delta^2 \alpha n^{k_1} / (2 \ell! k_1^{k_1})$ (where we used~\eqref{mathcalH'edges}).
	We can thus fix a colouring $c$ and $\mathcal{H}^c \subseteq \mathcal{H}'$ having properties \ref{item:embedding-cleaning-sta}-\ref{item:embedding-cleaning-int} and also fulfilling 
	\begin{equation}\label{sizeHc}
	|\mathcal{H}^c| \ge \delta^2 \alpha n^{k_1} / (2 \ell! k_1^{k_1}).
	\end{equation}
	Now	we further restrict $\mathcal{H}^c$ to make sure that, for all edges $X \in E(\mathcal{H}^c)$ corresponding to a walk $W$, the interiors of the walks are all consistently coloured.
	All of the walks $W$ have length $k_2$,
	and thus (seeing $W$ as a sequence of vertices),
	the vertices outside of $\sta{W}$ and $\ter{W}$ correspond to $k_2-k+1$ vertices with possible repetitions,
	whose underlying set $\int{W}$ is coloured with different colours from $\{k, ..., k_1 - k\}$.
	So tracking the colours received by the vertices of the walk defines a sequence of colours,
	with possible repetitions, chosen among $k_1 - 2k + 1$ available colours.
	As there are at most $(k_1 - 2k + 1)^{k_2 - k + 1} \leq k_1^{k_2}$ such sequences, and because of~\eqref{sizeHc},
	the pigeon-hole principle gives a subset $E(\mathcal{H}) \subseteq E(\mathcal{H}^c)$ of size at least 
	\begin{equation}\label{sizeH__}
	|E(\mathcal{H}^c)| / k_1^{k_2} \ge \delta^2 \alpha n^{k_1} / (2 \ell! k_1^{k_1 + k_2}),
		\end{equation}
	such that each walk $W$ corresponding to an edge $X$ of $\mathcal H$ is coloured in exactly the same way under $c$.
	
	\medskip
	
	\noindent \emph{Step 3: Using supersaturation.}	
	Let $\beta = \delta^2 \alpha / (2 \ell! k_1^{k_1 + k_2})$.
	By our assumptions, $1/n \ll \beta$
	and by~\eqref{sizeH__}, we know that~$\mathcal{H}$ is a $k_1$-graph with at least $\beta n^{k_1}$ edges.
	So, we can apply~\cref{lemma:turanzero}, with $k_1$ and $M \ell$ playing the roles of $k$ and $s$, to find that $\mathcal H$ contains a copy $\mathcal K$ of $K^{(k_1)}(M \ell)$, the complete $k_1$-partite $k_1$-graph with classes of size $M \ell$.
	Now, take any edge in $\mathcal K$, and recall it gives rise to  a walk $W=v_1 v_2 \dotsb v_{k_2}$ in~$H$.
	For all~$i$, let $V_i$ denote the partition class of $\mathcal K$ that contains $v_i$.
	Note that by construction, $V_i=V_j$ is only possible for $i,j\in\{k, ..., \ell-k+1\}$.
	As all walks corresponding to edges of $\mathcal K$ are coloured in the same way, each of them passes through the sets~$V_i$ in the same order. 
	
	Consider any injective function $h : V(T_I) \rightarrow V(\mathcal K)$ which maps all of $L_{t+i-1}$ to $V_i$, for each $i \in \{1, \dotsc, \ell \}$.
Such a function exists, since by assumption, each $L_{t+i-1}$ has size at most $M$, and since  $W$ repeats each vertex at most $k_2\le \ell$ times, whereas each $V_i$ has $M\ell$ vertices. As $\mathcal K$ is complete $k_1$-partite, we have found the desired embedding of $T_I$ in $H$.
\end{proof}

\subsection{Proof of \cref{lem:embedding}} \label{subsection:proofofembeddinglemma}
The proof of \cref{lem:embedding} proceeds by separating the input $T$ in $T_1$ induced by the first $\ell+k-1$ layers,
which we call the ``trunk of $T$'',
and the remaining $T_2 = T \setminus T_1$, which we call the ``crown of $T$''.

The proof is separated into three steps.
In the first step, we will prepare the host graph $H$ for the embedding.
This will be done by removing non-extensible edges, or edges with undesirable codegree properties from $H[W_1, \dotsc, W_k]$ to get to a `cleaned' subgraph $H' \subseteq H$, and finding a suitable ordering of the clusters $W_1, \dotsc, W_k$ so that the final embedding satisfies the required properties.
In a second step, we will apply \refandname{lemma:embeddingtrunk} to embed the trunk $T_1$.
In the third and final step we extend the embedding of $T_1$ to an embedding of the whole tree, which is done using \refandname{proposition:emb extension}.
Now come the details.

\begin{proof}[Proof of \cref{lem:embedding}]
	To begin, we introduce a new constant $\eps > 0$ such that $\theta \ll \eps \ll d, c, 1/k$, and set $\delta =dc^k/(2k)$.
	
	\medskip
	
	\noindent \emph{Step 1: Preparing $H$ and $T$ for the embedding.}
	Obtain $H''$ from $H[W_1,\dots,W_k]$  by removing all  non-$\theta$-extensible edges. Since $1/n, \theta \ll \eps$, Lemma~\ref{lem:extensible} implies that
	there are at most $\eps \binom{n}{k} \leq \eps n^k$ non-$\theta$-extensible edges.
	Hence, by our choice of $\eps \ll d, c, 1/k$, we have
	\[e(H''[W_1,\dots,W_k])\ge d|W_1|\dotsb |W_k|-\eps n^k \ge d c^k n^k -\eps n^k \ge \frac{dc^k}{2}n^k.\]
	Use \cref{proposition:mincodeg} to find a subgraph $H'\subseteq H''$	such that for every $f \in \partial H'$, $\deg_{H'}(f) \ge \frac{dc^k}{2k} n = \delta n$, by definition of $\delta$.
	
	Now turn to $T$.
	For each $i \in [k]$, let $\overline{L}_i = \bigcup_{j\ge 0}L_{\ell+jk+i}$. Note that $\{ \overline{L}_i \}_{i \in [k]}$  partitions $V_{[\ell+1, m]}$. Let $\sigma:[k] \to [k]$ be a permutation with
		$|\overline{L}_{\sigma^{-1}(1)}|\ge \dotsb \ge |\overline{L}_{\sigma^{-1}(k)}|$.
	Our plan is to embed $\overline L_{i}$ into $W_{\sigma(i)}$ for all $i \in [k]$, as this will ensure~\ref{keylemma:3}.
	\medskip
	
	\noindent \emph{Step 2: Embedding the trunk of the tree.}
	Let $T_1 := T[V_{[1,\ell + k-1]}]$. We will embed $T_1$, using vertices in $R$, starting from $f\subseteq e$ and ending in $W_1\cup \dotsc\cup W_k$.
	Formally, our goal is to find an embedding $\phi_1$ of $T_1$ such that
	\begin{enumerate}[\upshape (T1)]
		\item \label{item:f1}
		$\phi_1(\mathbf x) = \mathbf f$,
		\item \label{item:f2}  $\phi_1(V_{[1,\ell]} \setminus \mathbf x)\subseteq R$,
		\item \label{item:f3} for every $i \in \{ 1, \dotsc, k-1 \}$, $L_{\ell + i}$ is embedded in $W_{\sigma(i)}$, and
		\item \label{item:f4} if $f\in \partial T_1 [V_{\ell + 1,\ell + k - 1}]$ then  $\deg_{H''}(\phi_1(f)) \ge \delta n$.
	\end{enumerate}
	We initially set $\phi_1(\mathbf x) = \mathbf f$, thus ensuring~\ref{item:f1}.
	In order to embed the remaining vertices in $T'_1 = T_1 \setminus \mathbf x$, we will use Lemma~\ref{lemma:embeddingtrunk}, which we will apply in a suitably defined subgraph.
	For this, say that a $k$-tuple $\mathbf e' = (x'_1,\dots,x'_k) \in R^k$ is \emph{$e$-good} if $e \cup \mathbf e'$ induces a copy of $K^{(k)}(2)$ with partition classes $\{x_1,x_1'\},\dotsc,\{x_k,x_k'\}$.
	Define  sets $\mathbf F_1, \mathbf F_2 \subseteq \partial^\circ H$ as
	\begin{align*}
		\mathbf F_1 & = \{ (x'_2, \dotsc, x'_k)\in \partial^\circ H : \text{$(x'_1, x'_2, \dotsc, x'_k)$ is an $e$-good $k$-tuple for some $x'_1 \in R$}\}, \text{and } \\
		\mathbf F_2 & = \{ (y_1, \dotsc, y_{k-1}) \in W_{\sigma(1)}\times \dots \times W_{\sigma(k-1)} : \deg_{H'}(\{y_1,\dots,y_{k-1}\})\ge \delta n \}.
	\end{align*}
	Next, we show that between every $\mathbf h, \mathbf h'\in\partial^\circ  H$ there are many short walks of fixed length which pass through $R$.
	To this end, note that $H$ is $\gamma n$-large and $R$ is a $(\gamma, \mu, 2)$-reservoir,
	and therefore $H$ is $((\gamma^2 - 2 \mu)n, R)$-large.
	Moreover, by the choice of $\mu \ll \gamma$, $H$ is actually $(\gamma^2n/2, R)$-large. By $1/n \ll \gamma \ll 1/k$,
	we can apply Lemma~\ref{lemma:strengthenedconnectinglemma} with input $k, \gamma^2/2$, and $\ell$ in place of $k, \gamma, \ell$, to see the following.
	
	\begin{claim} \label{claim:applyingconnectinglemma}
		For all $\mathbf h, \mathbf h'\in\partial^\circ H$,
		there are $(\gamma^2 n/4)^q$ many walks $W$ of length $\ell$ in $H[\mathbf h \cup R \cup \mathbf h']$ from $\mathbf h$ to $\mathbf h'$,
		each with $q$ internal vertices all in $R\setminus (\mathbf h \cup \mathbf h')$.
	\end{claim}
	We wish to apply \cref{lemma:embeddingtrunk} with $\mathbf F_1, \mathbf F_2$, and $R$ in place of $U$.
	Let us check that its hypotheses are satisfied.
	First, we show that $|\mathbf F_1|$ is large.
	Since $e$ is $\theta$-extensible and $R$ is a $(\gamma, \mu, 2)$-reservoir,
	there are at least $(\theta - \mu) \binom{|U| - k}{k} \ge ( \theta \gamma^k / (2k!)) n^k$ $e$-good tuples in $R$ (we have used $\mu \ll \theta$). Keeping the last $(k-1)$-vertices of any such tuple, we deduce that  $|\mathbf F_1|\ge (\theta \gamma^k / (2 k!)) n^{k-1}$.	
	Secondly, we show that $|\mathbf F_2|$ is large.
	For this, recall that each $(k-1)$-tuple in $\partial H'$ has codegree at least $\delta n$.
	So, we see that $|\mathbf F_2|\ge \delta^{k-1} n^{k-1}$.
	Finally, by \cref{claim:applyingconnectinglemma}, for each choice $\mathbf v_1 \in \mathbf F_1$ and $\mathbf v_2\in \mathbf F_2$ there are at least $(\gamma^2 n/4)^q$ walks $W$ of length $\ell$ going from $\mathbf v_1$ to $\mathbf v_2$ in $H$,
	each satisfying $|\int{W}| = q$, internally disjoint from $\mathbf v_1\cup \mathbf v_2$ and $\int{W} \subseteq R$, as required.

	Note that $T'_1 \subseteq T[V_{[2, \ell +k - 1]}]$.	
	Since $\mathcal{L} = (V_1, \dotsc, V_m)$ is a layering of $(T,\mathbf x)$, and $\Delta_1(T) \leq \Delta$,
	\cref{proposition:boundedclusters} implies that
	$|L_i| \leq \Delta^{\ell + k - 1}$ for all $1 \leq i \leq \ell+k-1$,
	which implies that $V_{[2, \ell+k-1]}$ is $\Delta^{\ell + k - 1}$-bounded.
	Since $1/n \ll 1/k, 1/\Delta, \theta, \gamma, d, c$,
	we can apply \refandname{lemma:embeddingtrunk}, with $T'_1$, $H$, $R$, $\min\{ ( \theta \gamma^k / (2 k!)), \delta^{k-1} \}$, $\Delta^{\ell+k-1}$, $2$, $\ell+k-3$ playing the roles of $T_I$, $H$, $U$, $\delta$, $M$, $t$ and $\ell$, respectively.
	By doing so, this gives an embedding $\phi'_0: V(T'_1) \rightarrow V(H)$.
	By construction, the union of $\phi_0$ and $\phi'_0$ gives an embedding $\phi_1$ of $T_1$ satisfying \ref{item:f1}--\ref{item:f4}.
	\medskip
	
	\noindent \emph{Step 3: Embedding the crown of the tree.}
	We need to extend the embedding $\phi_1$ of $T_1$ to an embedding of all of $T$ in $H'[W_1,\dots, W_k]$.
	By~\ref{item:f4}, we know that every $f\in \phi_1(\partial T[V_{\ell+1,\ell+k-1}])$ has codegree at least $\delta n$ in $H'[W_{\sigma(1)},\dots,W_{\sigma(k-1)}]$. Since $\beta \ll d, c, 1/k$ and the definition of $\delta$, we can assume $\beta \leq \delta/2$.
	Thus, we can use \refandname{proposition:emb extension}, with $H', W_{\sigma(1)}, \dotsc, W_{\sigma(k)}$ playing the role of $H, W_1, \dotsc, W_k$, to find an embedding $\phi$ of $T$ which extends $\phi_1$ and such that, for each $i \in [m - \ell]$, $L_{\ell + i}$ is embedded in $W_j$, where $i \equiv \sigma(j) \bmod k$.
	
	Now we verify that $\phi$ satisfies  \ref{keylemma:1}--\ref{keylemma:4}.
	Since $\phi$ extends $\phi_1$ and since $r \subseteq V(T_1)$, \ref{keylemma:1} follows from \ref{item:f1}.
	Since all of $V(T_2) = V(T) \setminus V(T_1)$ was embedded in $W_1 \cup \dotsb \cup W_k$, which is disjoint from $R$,
	we have $\phi^{-1}(R) = \phi^{-1}_1(R)$.
	From \ref{item:f2}--\ref{item:f3}, we deduce that \ref{keylemma:2} holds.
	The properties of $\phi$ imply that, for each $i \in [k]$, $\phi(\overline{L}_i)\subseteq W_{\sigma(i)}$, and therefore $\phi^{-1}(W_\sigma(i)) = \overline{L}_i$.
	By the choice of $\sigma$, \ref{keylemma:3} holds.
	Finally, since $H'[W_1, \dotsc, W_k]$ only contains $\theta$-extensible edges, \ref{keylemma:4} holds.
\end{proof}

\section{Absorption}\label{section:absorption}

In this section, we state and prove lemmas which will allow us to complete the embedding of an almost spanning tree.
This technique is similar to the one used in~\cite{BHKMPP2018,BMPP18}.
Let us first define useful structures both for the $k$-tree we want to embed, and for the host graph which is used to embed the $k$-tree.

\begin{definition}[Absorbing $X$-tuple]
Let $k\ge 3$ and let $X$ be a $(k-1)$-tree on $h\ge k-1$ vertices, with a fixed valid ordering $x_1, \dotsc, x_{h}$.
For a $k$-tree $T
$, we say that an $(h + 1)$-tuple $(v_1, \dotsc, v_{h}, v^\ast)$ of vertices of $T$ is an \emph{$X$-tuple} if
\begin{enumerate}[\normalfont{(\roman*)}]
	\item $V(T(v^\ast)) = \{ v_1, \dotsc, v_{h} \}$, and
	\item the map $x_i \mapsto v_i$ is a hypergraph isomorphism between $X$ and $T(v^\ast)$.
\end{enumerate}
Let $H$ be a $k$-graph on $n$ vertices and let $(v_1, \dotsc, v_k)$ be a $k$-tuple of distinct vertices of $H$. An \emph{absorbing $X$-tuple for $(v_1, \dotsc, v_k)$} is an $(h + 1)$-tuple $(u_1, \dotsc, u_h, u^\ast)$ of vertices of $H$ such that
\begin{enumerate}[(A)]
	\item \label{Xtuple:1} $\{ v_1, \dotsc, v_{k-1}, u^\ast \} \in H$, and
	\item \label{Xtuple:2} there exists a copy $\tilde{X}$ of $X$ on $\{ u_1, \dotsc, u_h \}$ such that $\tilde{X} \subseteq H(v_k) \cap H(u^\ast)$.
\end{enumerate}
Furthermore, we write $\Lambda_X(v_1, \dotsc, v_k)$ for the set of absorbing $X$-tuples for $(v_1, \dotsc, v_k)$, and we let $\Lambda_X$ denote the set of all absorbing $X$-tuples in $H$, that is, $\Lambda_X$ is the union of $\Lambda_X(v_1, \dotsc, v_k)$ over all $k$-tuples $(v_1, \dotsc, v_k)$ of distinct vertices of $V(H)$. Suppose there exists an embedding $\phi: V(T) \rightarrow V(H)$ and let $(u_1, \dotsc, u_{h}, u^\ast)$ be an $(h+ 1)$-tuple of vertices of $H$.
We say $(u_1, \dotsc, u_{h}, u^\ast)$ is \emph{$X$-covered by $\phi$} if there exists an $X$-tuple $(v_1, \dotsc, v_h, v^\ast)$ of vertices of $T$ such that $\phi(v^\ast) = u^\ast$ and $\phi(v_i) = u_i$ for all $i \in [h]$. \end{definition}
The idea behind the definition of this gadget (the $X$-tuple) is that it will allow us to extend the embedding of a tree by iteratively adding leaves. 
We illustrate an $X$-tuple and how the extension step works in Figure~\ref{figure:Xtuple}.

 \begin{figure}[h!]
	\centering
		\begin{tikzpicture}[scale=1.1]
		
		\coordinate (v2l) at (2, 3.9);
		\coordinate (v1l) at (3, 3.7);
		\coordinate (v3l) at (0, 3.9);
		\coordinate (u2l) at (0, 2.2);
		\coordinate (u1l) at (-2.2, 2);
		\coordinate (u3l) at (-1, 0);
		\coordinate (u4l) at (1, 0);
		\coordinate (u0l) at (2.5, 1.3);
		
		\draw (v2l) node {$v_2$};
		\draw (v1l) node {$v_1$};
		\draw (v3l) node {$v_3$};
		\draw (u2l)  node  {$u_2$};
		\draw (u1l)  node {$u_1$};
		\draw (u3l) node  {$u_3$};
		\draw (u4l) node {$u_4$};
		\draw (u0l) node {$u^*$};
		
		\draw (8,0)+(v2l) node {$v_2$};
		\draw (8,0)+(v1l) node {$v_1$};
		\draw (8,0)+(v3l) node {$v_3$};
		\draw (8,0)+(u2l)  node  {$u_2$};
		\draw (8,0)+(u1l)  node {$u_1$};
		\draw (8,0)+(u3l) node  {$u_3$};
		\draw (8,0)+(u4l) node {$u_4$};
		\draw (8,0)+(u0l) node {$u^*$};
		
		\draw [->,dashed, thick] (3.5,1.5) -- node[below] {adding $v_3$} (5,1.5); 
		
		\coordinate (v3) at (0,3.5);
		\coordinate (v2) at (2,3.5);
		\coordinate (v1) at (3,3.3);
		\coordinate (u1) at (-2,1.5);
		\coordinate (u2) at (0,1.8);
		\coordinate (u0) at (2,1.5);
		\coordinate (u3) at (-1,0.5);
		\coordinate (u4) at (1,0.5);
		
		\tikzstyle{every node}=[circle, draw, fill, inner sep=0pt, minimum width=4pt]
		\draw (v1) node {};
		\draw (v2) node {};
		\draw (v3) node {};
		\draw (u1) node {};
		\draw (u2) node {};
		\draw (u0) node {};
		\draw (u3) node {};
		\draw (u4) node {};		
		
		\draw (8,0)+(v1) node {};
		\draw (8,0)+(v2) node {};
		\draw (8,0)+(v3) node {};
		\draw (8,0)+(u1) node {};
		\draw (8,0)+(u2) node {};
		\draw (8,0)+(u0) node {};
		\draw (8,0)+(u3) node {};
		\draw (8,0)+(u4) node {};

		\triple{(u1)}{(u2)}{(u0)}{6pt}{1pt}{DarkDesaturatedBlue,opacity=0.8}{VerySoftBlue,opacity=0.2};
		\triple{(u3)}{(u2)}{(u0)}{6pt}{1pt}{DarkDesaturatedBlue,opacity=0.8}{VerySoftBlue,opacity=0.2};
		\triple{(u4)}{(u2)}{(u0)}{6pt}{1pt}{DarkDesaturatedBlue,opacity=0.8}{VerySoftBlue,opacity=0.2};
		\draw[-, thick](u1) -- (u2);
		\draw[-,thick](u3) -- (u2);
		\draw[-,thick](u4) -- (u2);
		\draw[-, thick](8,0)+(u1) -- ($(8,0)+(u2)$);
		\draw[-,thick](8,0)+(u3) -- ($(8,0)+(u2)$);
		\draw[-,thick](8,0)+(u4) -- ($(8,0)+(u2)$);
		
		\triple{(u1)}{(v3)}{(u2)}{6pt}{1pt}{DarkDesaturatedBlue,opacity=0.2}{VerySoftBlue,opacity=0.05};
		\triple{(u3)}{(v3)}{(u2)}{6pt}{1pt}{DarkDesaturatedBlue,opacity=0.2}{VerySoftBlue,opacity=0.05};
		\triple{(u2)}{(v3)}{(u4)}{6pt}{1pt}{DarkDesaturatedBlue,opacity=0.2}{VerySoftBlue,opacity=0.05};
		\triple{(v2)}{(v1)}{(u0)}{6pt}{1pt}{DarkDesaturatedBlue,opacity=0.2}{VerySoftBlue,opacity=0.05};
		
		
		\triple{($(8,0)+(u1)$)}{($(8,0)+(v3)$)}{($(8,0)+(u2)$)}{6pt}{1pt}{DarkDesaturatedBlue,opacity=0.8}{VerySoftBlue,opacity=0.2};
		\triple{($(8,0)+(u3)$)}{($(8,0)+(v3)$)}{($(8,0)+(u2)$)}{6pt}{1pt}{DarkDesaturatedBlue,opacity=0.8}{VerySoftBlue,opacity=0.2};
		\triple{($(8,0)+(u2)$)}{($(8,0)+(v3)$)}{($(8,0)+(u4)$)}{6pt}{1pt}{DarkDesaturatedBlue,opacity=0.8}{VerySoftBlue,opacity=0.2};
		\triple{($(8,0)+(v2)$)}{($(8,0)+(v1)$)}{($(8,0)+(u0)$)}{6pt}{1pt}{DarkDesaturatedBlue,opacity=0.8}{VerySoftBlue,opacity=0.2};
	\end{tikzpicture}

\caption{An illustration of an $X$-tuple and how this is used to extend an embedding of a $k$-tree.
	In this case, $k=3$, $H$ is a $3$-graph, $(v_1,v_2,v_3)$ is a tuple of distinct vertices of $H$, and $X$ is a star with $3$ leaves (which is a $2$-tree).
	The tuple $(u_1,u_2,u_3,u_4,u^*)$ is an $X$-tuple for $(v_1,v_2,v_3)$, so $\{u_1,u_2,u_3,u_4\}$ induces a copy of $X$ in $H(u^\ast) \cap H(v_2)$, and $\{v_1,v_2,u^*\}$ is an edge of $H$, as shown in the left picture. 
	Suppose, in addition, that $(u_1,u_2,u_3,u_4,u^*)$ is $X$-covered by a embedding $\varphi$ of some $3$-uniform tree $T$, and 
	$\{v_1,v_2\}\in \partial \varphi(T)$ and $v_3\not\in V(\varphi(T))$.
	To find an embedding of the tree $T+v$ obtained by attaching a new vertex to $\{v_1,v_2\}$, we can modify $\varphi$ by switching $u^*$ with $v_3$ and then adding the edge $\{v_1,v_2,u^*\}$, as shown in the right picture.
	This gives an embedding of $T+v$ where $v$ is copied to $u^*$.}
	\label{figure:Xtuple}
\end{figure}

The following lemma is the heart of our absorbing method.

\begin{lemma}[Absorbing Lemma] \label{lemma:finishingembedding} Let $n\ge h\ge k\ge 3$ and let $0<\delta <\alpha$.
	Let $T$ be a $k$-tree on $n$ vertices with a valid ordering of $V(T)$ given by $v_1, \dotsc, v_n$,  and let $T_0 = T[ \{ v_1, \dotsc, v_{n'} \}]$ be a $k$-subtree of $T$ on $n' \ge (1 - \delta)n$ vertices.
	Let $H$ be a $k$-graph on $n$ vertices, and suppose there exists an embedding $\phi_0:V(T_0)\to V(H)$, a $(k-1)$-tree $X$ on $h$ vertices and a family $\calA \subseteq \Lambda_X$ of $(h + 1)$-tuples of vertices of $H$ with the following properties:
	\begin{enumerate}[\normalfont{(\roman*)}]
		\item the tuples in $\calA$ are pairwise vertex-disjoint, \label{itm:absorbingtuples-disjoint}
		\item every tuple in $\calA$ is $X$-covered by $\phi_0$, and \label{itm:absorbingtuples-covered}
		\item $|\Lambda_X(v_1, \dotsc, v_k) \cap \calA| \ge \alpha n$ for every $k$-tuple of distinct vertices of $H$ such that $\{ v_1, \dotsc, v_{k-1} \} \in \partial H$. \label{itm:absorbingtuples-many}
	\end{enumerate}	
	Then there exists an embedding of $T$ in $H$.
\end{lemma}

\begin{proof}
	Let $m = n - n'$ and let $\{ x_1, \dotsc, x_m \}$ be an arbitrary enumeration of $V(H) \setminus V( \phi_0(T_0) )$.
	For every $i \in [m]$, we set $T_i := T[ \{ v_1, \dotsc, v_{n' + i} \} ]$. Iteratively, for each $0\le i\le  m$, we will find an embedding $\phi_i: V(T_i) \rightarrow V(H)$ and subset $\calA_i \subseteq \calA$ with the following properties:
	\begin{enumerate}[(a{$_i$})]
		\item $\phi_i(V(T_i)) = \phi_0(T_0) \cup \{ x_1, \dotsc, x_i \}$, \label{item:embedding1}
		\item $|\calA_i| \leq i$, and \label{item:embedding2}
		\item for every $(u_1, \dotsc, u_h, u^\ast) \in \calA \setminus \calA_i$, $\phi^{-1}_i(u^\ast) = \phi^{-1}_0(u^\ast)$ and $\phi^{-1}_i(u_j) = \phi^{-1}_0(u_j)$ for every $j \in [h]$. \label{item:embedding3}
	\end{enumerate}
	It is very easy to see that for $i = 0$ the properties hold for $\phi_0$ and $\calA_0 := \emptyset$.
	Suppose that for some $0\le i\le m-1$ we have defined $\phi_i$ and $\calA_i$ satisfying \ref{item:embedding1}--\ref{item:embedding3}. We shall construct $\phi_{i+1}$ and $\calA_{i+1}$ satisfying (a$_{i+1}$)--(c$_{i+1}$).
	
	Since $v_1, \dotsc, v_n$ is a valid ordering for $T$, there exists a unique $(k-1)$-set $\{v_{i_1}, \dotsc, v_{i_{k-1}} \} \subseteq V(T_i)$ such that $\{v_{i_1}, \dotsc, v_{i_{k-1}}, v_{n'+i+1} \} \in T$.
	Let $w_1, \dotsc, w_{k-1} \in V(H)$ be an arbitrary labelling of $\phi_i( \{ v_{i_1}, \dotsc, v_{i_{k-1}} \} )$ and define $w_k := x_{i+1}$.
	Note that $w_1, \dotsc, w_{k-1} \in \partial H$, so by assumption $|\Lambda_X(w_1, \dotsc, w_k) \cap \calA| \ge \alpha n$, $i \leq m \leq \delta n$ and $\delta < \alpha$.
	Thus, by \ref{item:embedding2}, we have \begin{align*}
		|\Lambda_X(w_1, \dotsc, w_k) \cap \calA \setminus \calA_i| \ge \alpha n - |\calA_i| = \alpha n - i \ge (\alpha - \delta)n > 0.
	\end{align*}
	Now we can select an arbitrary absorbing $X$-tuple $(u_1, \dotsc, u_{h}, u^\ast)\in \calA \setminus \calA_i$ for $(w_1, \dotsc, w_k)$, and define $\calA_{i+1} := \calA_i \cup \{ (u_1, \dotsc, u_{h}, u^\ast) \}$.
	Note this definition of $\calA_{i+1}$ satisfies (b$_{i+1}$).
	Since $(u_1, \dotsc, u_{h}, u^\ast)$ is an $X$-tuple for $(w_1, \dotsc, w_k)$ in $\calA \setminus \calA_i$, then $\{u_1,\dots,u_h,u^*\}$  is $X$-covered by $\phi_0$ and, because of~\ref{item:embedding3}, it is also $X$-covered by $\phi_i$.
	For every $x \in V(T_{i+1})$, define
	\begin{align*}
		\phi_{i+1}(x) := \begin{cases}
			w_k & \text{if $x = \phi^{-1}_{i}(u^\ast)$}, \\
			u^\ast & \text{if $x = v_{n'+i+1}$}, \\
			\phi_i(x) & \text{otherwise.}
		\end{cases}
	\end{align*}
	Note that the function $\phi_{i+1}$ is injective and it satisfies~(a$_{i+1}$) and~(c$_{i+1}$).
	To finish, we check that $\phi_{i+1}$ is an embedding of $V(T_{i+1})$.
	Indeed, if $e \in T_{i+1}$ does not contain $v_{n'+i+1}, \phi^{-1}_{i}(u^\ast)$, then $e \in T_i$ and $\phi_{i+1}(e) = \phi_i(e) \in H$ since $\phi_i$ is an embedding of $T_i$.
	If $e \in T_{i+1}$ contains $v_{n'+i+1}$, then $e = \{v_{i_1}, \dotsc, v_{i_{k-1}}, v_{n'+i+1} \}$ and because of~\ref{Xtuple:1} (in the definition of absorbing $X$-tuples) we know that $\phi_{i+1}(e) = \{ w_1, \dotsc, w_{k-1}, u^\ast \} \in H$.
	If $e \in T_{i+1}$ contains $\phi^{-1}_{i}(u^\ast)$, then there exists a $(k-1)$-edge $e'$ in $T(\phi^{-1}_{i}(u^\ast))$ such that $e = e' \cup \{ \phi^{-1}_{i}(u^\ast) \}$ and $u^\ast \cup \phi_i(e') \in H$.
	Note that $\phi_{i+1}(e) = w_k \cup \phi_i(e')$.
	Since $(u_1, \dotsc, u_h, u^\ast)$ is an $X$-tuple for $(w_1, \dotsc, w_k)$, then by~\ref{Xtuple:2} we know that $w_k \cup \phi(e') \in H$ and therefore $\phi_{i+1}(e) \in H$, as desired.
	Thus $\phi_{i+1}$ is an embedding of $T_{i+1}$.
	
	Following this process for $m$ steps, we find an embedding $\varphi_m$ of $T_m = T$, as desired.
\end{proof}

In the remainder of this section we will prove a series of lemmas that allow us to build a partial embedding of a tree $T$ in which properties \ref{itm:absorbingtuples-disjoint}--\ref{itm:absorbingtuples-many} of Lemma~\ref{lemma:finishingembedding} are fulfilled.

\subsection{Finding separated tuples in $k$-trees}\label{subsection:separatedtuples}

Let $\calT_{k, [h]}$ be the family of all non-labelled $k$-trees on at most $h$ vertices, up to isomorphism. For our purposes, we need to bound $|\calT_{k, [h]}|$ in terms of $h$ and $k$. Let $\calT_{k, h}$ be the family of all $k$-trees on exactly $h$ vertices.
We will bound $|\calT_{k, h}|$ by the number of \emph{labelled} $k$-trees on $h$ vertices. For all such labelled trees $T$, recall that all but the first $k$ vertices have an anchor.
Since for each vertex in $T$ we have at most $\binom{h}{k-1}$ options for its anchor, we thus have $|\calT_{k, h}| \leq \tbinom{h}{k-1}^{h-k} \leq h^{h (k-1)}$,
which in turn implies
\begin{align}
	|\calT_{k, [h]}| & \leq h^{h k}. \label{equation:boundontrees}
\end{align}
Recall that the distance between $(k-1)$-tuples of the shadow of a $k$-tree was given in Definition~\ref{definition:distance}.
Given a $k$-tree $T$, a $(k-1)$-tree $X$, and $\ell \ge 0$, we say that a set $\calB$ of $X$-tuples of $T$ is \emph{$\ell$-separated} if they are pairwise at a distance at least $\ell$,
that is, for each distinct $B_i, B_j \in \calB$,
and each $f_i \in \partial T[B_i]$, $f_j \in \partial T[B_j]$, we have $d_T(f_i, f_j) \ge \ell$.

We now show that every bounded-degree tree contains a large $\ell$-separated set of $X$-tuples, for some $(k-1)$-tree~$X$.

\begin{proposition} \label{proposition:findingfamilyinT}
	Suppose $0 < \mu \ll 1/\Delta, 1/k, 1/\ell$ and $k \ge 2$.
	Let $T$ be a $k$-tree on $n$ vertices such that $\Delta_1(T) \leq \Delta$.
	Then there exists a $(k-1)$-tree $X \in \calT_{k-1, [\Delta + k - 1]}$ and an $\ell$-separated set $\calB$ of $X$-tuples of $T$ with $|\calB| \ge \mu n$.
\end{proposition}

\begin{proof}
	By \cref{proposition:linkgraph} and the bound $\Delta_1(T)\le \Delta$, for every vertex $v \in V(T)$, $T(v)$ is a $(k-1)$-tree which is in $\calT_{k-1, [\Delta + k - 1]}$.
	By the pigeon-hole principle, there exists a $(k-1)$-tree $X \in \calT_{k-1, [\Delta + k - 1]}$ and a subset $W' \subseteq V(T)$ of size at least $n/|\calT_{k-1, [\Delta + k - 1]}|$ such that $T(w) \cong X$ for every $w \in W'$.
	Note that each $w \in W'$ yields an $X$-tuple $B_w$ in $T$ which uses precisely the vertices in $T(w)$.
	Let $W \subseteq W'$ be maximal so that for all distinct $w_1, w_2 \in W$, the $(k-1)$-tuples in $V(T(w_1))$ and $V(T(w_2))$ are at distance at least $\ell$.  
	Set $\calB :=\{B_w:w\in W\}$.
	It remains to show that $|\calB| \ge \mu n$.
	
	The assumption $\Delta_1(T) \leq \Delta$ implies that for every $w'_1 \in W'$, there are at most $(\Delta + k)^{\ell+1}$ other vertices $w'_2 \in W'$ such that $T(w'_1)$ and $T(w'_2)$ have distance less than $\ell$.
	We deduce that \begin{align*}
		|\calB| = |W| \ge \frac{|W'|}{(\Delta + k)^{\ell+1}} \ge \frac{n}{|\calT_{k-1, [\Delta + k - 1]}| (\Delta + k)^{\ell+1}} \ge \mu n,
	\end{align*} where the last inequality follows from \eqref{equation:boundontrees} and the assumption that $\mu \ll 1/\Delta, 1/k, 1/\ell$.
\end{proof}

\subsection{Finding absorbing tuples in the host graph}\label{subsec:absorbing tuples}

Many copies of an $X$-tuple in a tree $T$ will indicate   parts of $T$ that are `flexible enough' to be interchanged. This can be used to extend a partial embedding of $T$ into an embedding of all of $T$.
In the following proposition we will show the existence of many absorbers for each $k$-tuple.
We remark that here the condition of $(\gamma n)$-large is not enough, as we cannot construct absorbers if there are isolated vertices; but forbidding isolated vertices in addition to being $(\gamma n)$-large will be enough.

\begin{proposition} \label{proposition:absorbingtuples}
	Let $1/n \ll \beta \ll \gamma, 1/h, 1/k$ with $h \ge k \ge 2$.
	Let $H$ be a $k$-graph on $n$ vertices which is $\gamma n$-large and has no isolated vertices.
	Let $X$ be a $(k-1)$-tree on $h$ vertices
	and let $(v_1, \dotsc, v_k)$ a $k$-tuple of distinct vertices of $H$ such that $\{v_1, \dotsc, v_{k-1}\} \in \partial H$.
	Then $|\Lambda_X(v_1, \dotsc, v_k)| \ge \beta n^{h + 1}$.
\end{proposition}

\begin{proof}
	We construct an absorbing $X$-tuple for $(v_1, \dotsc, v_k)$ by choosing vertices one by one.
	First, select an arbitrary vertex $u_{h + 1} \in N_H( \{ v_1, \dotsc, v_{k-1} \} ) \setminus \{ v_k \}$,
	and note that, since $\{v_1, \dotsc, v_{k-1}\} \in \partial H$, there are at least $\gamma n$ possible choices for $u_{h + 1}$.
	
	Define the $(k-1)$-graph $H' := H( v_k ) \cap H( u_{h + 1} )$.
	First, we show that $H'$ is non-empty.
	Since each vertex is contained at least in one edge, there must exist $k$-edges $e_1, e_2$ in $H$ which contain $v_k$ and $u_{h+1}$ respectively.
	Among all pairs $(e_1, e_2)$ of edges in $H$ such that $e_1$, $e_2$ contain $v_k$ and $u_{h+1}$ respectively, choose a pair such that $|e_1 \cap e_2|$ is maximum.
	If $|e_1 \cap e_2| < k-1$, then select $Y_1 \subseteq e_1$ of size $k-1$ containing $(e_1 \cap e_2) \cup \{v_k\}$, and $Y_2 \subseteq e_2$ of size $k-1$ containing $(e_1 \cap e_2) \cup \{u_{h+1}\}$.
	Then $Y_1, Y_2 \in \partial H$, and since $H$ is $\gamma n$-large, there exists $u' \in N(Y_1) \cap N(Y_2)$.
	Then $u' \cup Y_1$ and $u' \cup Y_2$ are two edges in $H$ containing $v_k$ and $u_{h+1}$ respectively and with larger intersection than $e_1, e_2$, a contradiction.
	Thus $|e_1 \cap e_2| = k-1$, which implies that $e_1 \cap e_2 \in H'$, and therefore $H'$ is non-empty, as desired.
	
	Now, observe that every $(k-2)$-set in $\partial H'$ has at least $\gamma n$ neighbours (which follows since $H$ is $\gamma n$-large).
	First, select an arbitrary $(k-1)$-set $\{x_1, \dotsc, x_{k-1}\} \in H'$.
	Using this, we can select $u_1, \dotsc, u_h$ iteratively in increasing order, as follows.
	First, fix a valid ordering $\{ t_1, \dotsc, t_{h} \}$ of the vertices of $X$.
	For each $1 \leq i \leq k-1$, having chosen $u_1, \dotsc, u_{i-1}$ already, we select $u_i$ as a neighbour of the $(k-2)$-set $\{x_{i+1}, \dotsc, x_{k-1}, u_1, u_2, \dotsc, u_{i-1}\} \in \partial H'$, with the additional assumpation that $u_i \notin \{ u_1, \dotsc, u_{i-1} \}$.
	Then, successively for $i=k-1, ...,  h+1$,  select $u_{i+1}$ in the following way: if $\{ t_{j_1}, \dotsc, t_{j_{k-2}}\}$ is the anchor of $t_{i+1}$, then  choose $u_{i+1}\in N_{H'}( \{ u_{j_1}, \dotsc, u_{j_{k-2}} \}  ) \setminus \{ u_1, \dotsc, u_i \}$  arbitrarily.
	By construction, $(u_1, \dotsc, u_{\ell + 1})$ is an absorbing $X$-tuple for $(v_1, \dotsc, v_k)$.
	
	Finally, as in each step there are at least $\gamma n/2$ possibilities to choose the next vertex~$u_i$, there are at least $( \gamma n/2)^{h + 1} \ge \beta n^{h + 1}$ absorbing $X$-tuples for $(v_1, \dotsc, v_k)$, as desired.
\end{proof}

Now we would like to find a linear-sized family $\calA \subseteq \Lambda_X$ of vertex-disjoint $X$-tuples such that every $k$-tuple has many absorbing $X$-tuples in $\calA$ (as to satisfy properties \ref{itm:absorbingtuples-disjoint} and \ref{itm:absorbingtuples-many} of Lemma~\ref{lemma:finishingembedding}).
The following lemma can be proved by selecting independently each $(\ell + 1)$-tuple in $\Lambda_X$ with probability $p := \Omega(n^{-\ell})$ and showing it satisfies the required properties with positive probability, which can be done using Chernoff's inequality (\cref{theorem:chernoff}) and Markov's inequality.
As this strategy is standard by now and has appeared in many other absorption-based proofs (e.g.~\cite[Claim 3.2]{RRS08}),
 we leave the details of the proof to the reader.

\begin{lemma} \label{lemma:usingchernoff}
	Let $1/n \ll \alpha\ll\beta, 1/k, 1/h$ with $h \ge k \ge 2$.
	Let $H$ be a $k$-graph  on $n$ vertices and let $X$ be a $(k-1)$-tree on $h$ vertices.
	Suppose $|\Lambda_X(v_1, \dotsc, v_k)| \ge \beta n^{h + 1}$ for every $k$-tuple $(v_1, \dotsc, v_k)$ of vertices of $H$ such that $\{v_1, \dotsc, v_{k-1}\} \in \partial H$.
	Then there is a set $\calA \subseteq \Lambda_X$ of at most $\alpha n$ disjoint $(h + 1)$-tuples of vertices of~$H$ such that $|\Lambda_X(v_1, \dotsc, v_k) \cap \calA| \ge \beta \alpha n / 8$ for every $k$-tuple $(v_1, \dotsc, v_k)$ of distinct vertices of~$H$ such that $\{v_1, \dotsc, v_{k-1}\} \in \partial H$.
	\hfill \qedsymbol
\end{lemma}

\subsection{Covering $X$-tuples with a partial tree embedding}
Our final step in order to use the Absorbing Lemma is to show that we can cover a large family of absorbing tuples.

\begin{lemma}[Embedding pseudopaths]\label{lemma:embedding_pseudopaths}
Let $1/n \ll 1/\Delta, 1/k, 1/\ell, \gamma$ satisfying $\Delta,k\ge 2$, and also $\ell \ge (2k+1)\lfloor{k/2}\rfloor+2k$.
	Let $H$ be a $\gamma n$-large $k$-graph on $n$ vertices.
	Let $P$ be a $k$-uniform $(f,f')$-pseudopath $P$, and let $\mathbf f, \mathbf f'$ be any ordering of $f$ and $f'$ respectively.
	Moreover, suppose $e(P) \ge \Delta k(\ell+3k)$ and $\Delta_1(P)\le \Delta$.
	Then, given any pair of disjoint $(k-1)$-tuples $\mathbf x,\mathbf y\in\partial^\circ H$,  
	there exists an embedding $\varphi:V(P)\to V(H)$ such that $\varphi(\mathbf f) = \mathbf x$ and $\varphi(\mathbf f') = \mathbf y$.
\end{lemma}

\begin{proof}	
	Suppose $P$ has $t$ edges $e_1, \dotsc, e_t$ such that $f \subseteq e_1$ and $f' \subseteq e_t$.	
	Now let $\mathcal L=(L_1,\dots,L_m)$ be the (unique) layering for $(P, \mathbf{f})$.
	By Lemma~\ref{lemma:layeringpseudopath}\ref{item:layerpseudo-bounded}, $|L_i| \leq k \Delta$ for all $i\in[m]$.
	Thus the number $m$ of layers of $\mathcal{L}$ satisfies $ m \ge |V(P)|/(k \Delta) \ge |E(P)|/(k \Delta) \ge \ell+3k$.

	We start our construction of the embedding by setting $\varphi(\mathbf f) = \mathbf x$ and $\varphi(\mathbf f') = \mathbf y$.
	As a next step, we will embed greedily the first $k$ layers of $\mathcal{L}$ into $R$.
	Recall the definition of $r(j)$ from Lemma~\ref{lemma:layeringpseudopath}\ref{item:layerpseudo-increasing}, and let $j_1$ be the maximum $j$ such that $r(j) \leq k$.
	Lemma~\ref{lemma:layeringpseudopath}\ref{item:layerpseudo-increasing} implies that  $r(j_1) = k$, $r(j) \le k$ for all $j \le j_1$ and $r(j) > k$ for all $j > j_1$.
	Let $P_1$ be the subgraph of $P$ spanned by the edges $e_1, \dotsc, e_{j_1}$.
	Then $P_1$ is the `minimum' subpath of $P$ which contains all the edges touching the first $k$ layers.
	
	Let $\mathbf s_{1} = e_{j_1} \setminus (L_1 \cup \dotsb \cup L_k)$, ordered according to the increasing layering order.
	Since $r(j_1) = k$, $\mathbf s_1$ has size $k-1$.
	Now we embed $P_{1} \setminus \mathbf s_1$ making sure there are `many' possible extensions available for $\mathbf s_1$.
	First, in a greedy fashion, we count the number of embeddings of $e_1$ which extend $\varphi(\mathbf f)$ simply by completing $\mathbf f$ to $e_1$ to an unused vertex outside $\mathbf x \cup \mathbf y$. Since $H$ is $\gamma n$-large, this can be done in at least $\gamma n - |\mathbf x \cup \mathbf y| \ge \gamma n / 2$ ways.
	Next, we  iteratively count the extensions of the embedding of $e_1$ to an embedding of $P_{1}$, which can be done again in a greedy fashion,  again having $\gamma n / 2$  choices for  an unused vertex each time.
	Letting $n_1 = |V(P_{1})| - (k-1)$, we deduce that there are at least $(\gamma n / 2)^{n_1}$ embeddings of $P_{1}$ which extend $\varphi(\mathbf f)$.
	Now, an averaging argument entails that there is an embedding of $P_{1} \setminus \mathbf s_1$ which extends $\varphi(\mathbf f)$ and can be extended to at least $((\gamma n / 2)^{n_1})/n^{n_1-|s_1|}=(\gamma/ 2)^{n_1} n^{k-1}$ embeddings of $P_{1}$.
	Let $\phi(P_{1} \setminus \mathbf s_1)$  be such an extension of $\varphi(\mathbf f)$, and let $\mathbf F_1 \subseteq V(H)^{k-1}$ be the set of all ordered $(k-1)$-tuples which are valid extensions of $\phi(P_{1} \setminus \mathbf s_1)$ to an embedding of $P_{1}$.
	Then $|\mathbf F_1| \ge (\gamma/ 2)^{n_1} n^{k-1}$.
	
Observe that if we consider the edges of $P$ in reverse ordering (i.e.~we consider $e_t$ to be the first, $e_1$ to be the last edge), the resulting $k$-tree is an $(f',f)$-pseudopath.
	So, defining $j_2$ as the minimum $j \leq t$ such that $r(j) \ge m - 2k + 2$, and defining $\mathbf s_2$ and $P_2$ accordingly, we can proceed as in the previous paragraph, to obtain an embedding $\varphi(P_2 \setminus \mathbf s_2)$ which extends $\varphi( \mathbf f')$ and can be extended to at least $(\gamma/n)^{n_2} n^{k-1}$ embeddings of $P_2$.
	Letting $\mathbf F_2 \subseteq V(H)^{k-1}$ denote the set of all ordered $(k-1)$-tuples  giving valid extensions of $\phi(P_{2} \setminus \mathbf s_2)$ to an embedding of $P_{2}$, we have $|\mathbf F_2| \ge (\gamma/ 2)^{n_1} n^{k-1}$.	

	We complete the embedding by using Lemma~\ref{lemma:embeddingtrunk}.
	Let $\delta > 0$ be such that $1/n \ll \delta \ll \gamma$.
	Since $m - 3k \ge \ell \ge (2k+1)\lfloor{k/2}\rfloor+2k$ and $H$ is $(\gamma n, V(H))$-large,
	Lemma~\ref{lemma:strengthenedconnectinglemma} outputs $q \leq m-3k$ such that
	for each pair of disjoint tuples $\mathbf v_1 \in \mathbf F_1$ and $\mathbf v_2 \in \mathbf F_2$ there are at least $( \gamma n/2 )^q$ walks of length $m - 3k$ connecting $\mathbf v_1$ and $\mathbf v_2$, each having $q$ internal vertices, and internally disjoint from $\mathbf v_1\cup \mathbf v_2$.
	By the choice of $\delta$ we have $|\mathbf F_1|, |\mathbf F_2| \ge \delta n^{k-1}$,
	and note that the remaining vertices to be embedded correspond to $P \setminus (P_1 \cup P_2)$, whose set of vertices is completely in $V_{[k+1, m-k]}$, which is $k \Delta$-bounded.
	So, we can use Lemma~\ref{lemma:embeddingtrunk}
	with $P, P \setminus (P_1 \cup P_2), k+1, m-2k-1, k\Delta$ playing the roles of $T, T_I, t, \ell, M$
	to obtain an embedding $\varphi'$ of the remaining vertices from $L_{k+1} \cup \dotsb \cup L_{m-k}$,
	which extends the embedding $\varphi$ to an embedding of all of $P$.
\end{proof}

\begin{lemma}[Covering Lemma]\label{lemma:coveringtuples}
	Let $1/n \ll \alpha \ll \mu \ll \nu \ll \gamma, 1/h$ with $h, \Delta, k \ge 2$ and $\ell \ge (2k+1)\lfloor{k/2}\rfloor+2k$.
	Let $X$ be a $(k-1)$-tree on $h$ vertices.
	Let $H$ be a $\gamma n$-large $k$-graph on $n$ vertices, and let $\calA \subseteq \Lambda_X$ be a set of at most $\alpha n$ pairwise disjoint $(h+1)$-tuples of vertices of $H$.
	Let $T$ be a $k$-tree on $\nu n$ vertices, with $\Delta_1(T)\leq \Delta$, and let $\calB$ be a $2\Delta k(\ell+3k)$-separated set of size at least $\mu n$ of $X$-tuples of vertices of $T$.
	Then, for any $\mathbf y\in \partial^\circ H$, $\mathbf x\in \partial^\circ T$, there is an embedding $\phi: V(T) \rightarrow V(H)$ such that $\phi(\mathbf x)=\mathbf y$ and every tuple in $\calA$ is $X$-covered by~$\phi$.
\end{lemma}

\begin{proof}
	Write $\mathcal A=\{A_1,\dots,A_t\}$, where $1\le t\le \alpha n$.
	We will abuse notation by treating $X$-tuples $B \in \mathcal{B}$ as subgraphs of $T$, consisting of the corresponding edges forming the $X$-tuple.
	We claim that there are $B_1,\dotsc, B_{t}\in \mathcal B$ such that, defining $d_T(\mathbf x, B_i)$ as the minimum of $d_T(\mathbf x, b)$ over all $(k-1)$-sets $b \in \partial T [B_i]$, we have
	\begin{align}
		\label{lem:cov:distance}
		\Delta k(\ell+3k) \leq d_T(r,B_j) \leq d_T(r, B_i)\text{ for all $1\le j < i\le t$}.
	\end{align}
	To see this, order the elements of  $\mathcal{B}$ as $B'_1, B'_2, \dotsc$ so that  $d_T(r, B'_i)$ increases.
	As $\mathcal{B}$ is $2\Delta k(\ell+3k)$-separated, and using the triangle inequality, we see that 
	\[2\Delta k(\ell+3k)\le d_T(B'_1,B'_2) \le d_T(B'_1,\mathbf x)+d_T(B'_2,\mathbf \mathbf x) \leq 2 d_T(B'_2, \mathbf x)\]
	and thus $d_T(B'_2, \mathbf x) \ge \Delta k(\ell+3k)$.
	Since $|\mathcal{B}|\ge \mu n \gg \alpha n  \ge t$,
	we can delete $B'_1$ from $\mathcal{B}$ if necessary, and delete more sets $B'_i$ until size exactly $t$ is reached.
	After relabelling, we arrive at sets $B_1, \dotsc, B_t$ satisfying~\eqref{lem:cov:distance}.
	
	Set $T_0:=\emptyset$.
	Given $i \in [t]$ and $T_{i-1} \subseteq T$,  define $T_i$ as follows.
	Let $t_i \in \partial T_{i-1}$ and $b_i \in \partial B_i$ such that $d_T(t_i, b_i)$ is minimised (if $i = 1$, select $t_1 = r$ instead).
	Let $P_i$ be the unique $(t_i, b_i)$-pseudopath in $T$, and let $T_{i} = T_{i-1} \cup P_i \cup B_i$.
	Then $T_0, T_1, \dotsc, T_t$ satisfy the following properties.
	\begin{claim}
		For all $1 \leq i \leq t$, $T_i$ is a subtree of $T$, and
		\begin{enumerate}[{\upshape(Q1)}]
			\item $T_i \subseteq T_{i+1}$ if $i < t$,
			\item $(B_1 \cup \dotsb \cup B_i) \subseteq T_i$, and
			\item $T_i \cap ( B_{i+1} \cup \dotsb \cup B_t ) = \emptyset$. \label{item:coveringtuples-after}
		\end{enumerate}
	\end{claim}
	Indeed, the only property which is not immediate from construction is~\ref{item:coveringtuples-after}.
	Suppose the property failed, and let $i$ be a minimum integer for which it fails.
	Then there exists a $j > i$ such that $B_j \cap T_i \neq \emptyset$.
	By minimality of $i$, $B_j \cap T_{i-1} = \emptyset$, and since $\mathcal{B}$ are pairwise disjoint subgraphs, $B_j \cap B_i = \emptyset$.
	Thus $B_j \cap (P_i \setminus B_i) \neq \emptyset$.
	Since $P_i$ was the unique minimum-length pseudopath between $T_{i-1}$ and $B_i$ in $T$, this implies $d_T(\mathbf x, B_j) < d_T(\mathbf x, B_i)$, contradicting~\eqref{lem:cov:distance}.
	
	For $i\ge 0$, we will now construct embeddings $\phi_i:V(T_i)\to V(H)\setminus (A_{i+1}\cup\dots\cup A_t)$ such that
	for $i \ge 1$, $\phi_i$ extends $\phi_{i-1}$, and such that $A_i=\phi_i(B_i)$ is $X$-covered by $\phi_i$.
	We start by setting $\phi_0(\mathbf x)=\mathbf y$.	
	Now assume that $i\ge 1$, and suppose we have embedded $T_{i-1}$ with the embedding $\phi_{i-1}$.
	By~\ref{item:coveringtuples-after}, the image of $B_i$ has not been defined in $\phi_{i-1}$.
	We begin by setting $\phi'_i(B_i)=A_i$,	in a way that $A_i$ is $X$-covered by $\phi'_1$.
	Recall that, by definition, $T_i$ is the union of $T_{i-1}$,
	$B_i$, and a $(t_i, b_i)$-pseudopath $P_i$, for some $b_i \in \partial B_i$ and some $t_i \in \partial T_i$ (or $t_i = \mathbf x$ if $i = 1$). Note that $H_i := H \setminus (A_{i+1} \cup \dotsb \cup A_{t} )$ has at least $n' = n - (t-1)(h+1) \ge (1 - \alpha(h+1))n \ge (1 - \gamma/3) n$ vertices, as by assumption, $\alpha h \ll \gamma$.
	A similar calculation entails that $H_i$ is $(\gamma n'/2)$-large.
	We claim that 
	\begin{equation}\label{disss}
	d_T(t_i, b_i)\ge \Delta k(\ell+3k).
	\end{equation}
	Then, we can use Lemma~\ref{lemma:embedding_pseudopaths} to find an embedding $\phi_i:V(P_i)\to V(H_i)$ which extends both $\phi_{i-1}$ and $\phi'_i$, and this completes step $i$.
	
	So let us show~\eqref{disss}. Assume otherwise, and let $1 \leq j < i-1$ be the minimum index such that $t_i \in \partial T_j$.
	By minimality, we have $t_i \in \partial (P_j \cup B_j)$.
	If $t_i \in \partial B_j$, then $P_i$ is a pseudopath from $B_i$ to $B_j$,
	and since $\mathcal{B}$ is $2 \Delta k (\ell + 3k)$-separated we then have $d_T(b_i,t_i) \ge 2 \Delta k (\ell + 3k)$, and we are done.
	Assume instead that $t_i \in \partial P_j \setminus \partial B_j$.	
	Let $b_j \in \partial T_j$ be such that $d_T(b_j, \mathbf x) = d_T(B_j, \mathbf x)$.
	Note that $t_i$ must lie on the unique pseudopath in $T$ from $b_j$ to $\mathbf x$, and thus
	\begin{align*}
		 d_T(B_j, t_i) + d_T(t_i, \mathbf x) = d_T(B_j, \mathbf x)
		  \le d_T(B_i, \mathbf x)
		  \le d_T(b_i,t_i)+d_T(t_i, \mathbf x),
		 \label{equation:distances}
	\end{align*}
	where the second inequality comes from~\eqref{lem:cov:distance}.
	As furthermore
	\[2 \Delta k(\ell+3k) \le d_T(B_{j},B_i) \leq d_T(b_i,t_i)+d_T(B_{j},t_i) < \Delta k(\ell+3k)+d_T(B_{j},t_i),\]
	we obtain $d_T(b_i,t_i)\ge d_T(B_{j},t_i)>\Delta k(\ell+3k)$, contrary to our assumption.
	Thus~\eqref{disss} holds.
	
	Having defined all $\phi_i$, we note that  each absorbing tuple in $\mathcal A$ is  $X$-covered by $\phi_t$.
	We extend~$\phi_t$ to an embedding $\phi$ of all of $V(T)$.
	We can get from $T_t$ to $T$ by iteratively adding leaves,
	thus the embedding can be found in a greedy fashion since $\delta_{k-1}(H) \ge \gamma n$, and at each step there are at most $|V(T)| \leq \nu n \leq \gamma n / 2$ used vertices.
\end{proof}

\section{Proof of the Main Theorem}\label{section:main-theorem} 
We now assemble all results from the previous sections
 in order to prove  Theorem~\ref{theorem:spanning-codegree}.
The proof is divided into three main steps.

We start  the first step by finding a small subtree $T'\subseteq T$ of linear size which we use to build the absorbing structures. First, using Proposition~\ref{proposition:findingfamilyinT} we find a $(k-1)$-tree $X$ such that $T'$ contains linearly many well-separated $X$-tuples, and then, with the help of Proposition~\ref{proposition:absorbingtuples} and Lemma~\ref{lemma:usingchernoff}, we find a family $\mathcal A$ of disjoint absorbing $X$-tuples in the host graph $H$ such that every $k$-tuple in $H$ has linearly many absorbing $X$-tuples in $\mathcal A$.
We then use the \nameandref{lemma:coveringtuples} to 
embed $T'$ in $H$, covering every tuple in $\mathcal A$.

In the second step, we will find an almost spanning subtree $T'' \subseteq T-T'$ and embed it  following the regularity method.
The \nameandref{theorem:weakregularity} gives a regular partition of the vertices of $H$ and  using Lemma~\ref{lemma:matching}, we find an almost spanning matching $\mathcal M$ in the corresponding reduced graph.
We find  a \nameandref{beta-decomp} of $T''$
and use the \nameandref{lem:embedding} to map the small parts of this decomposition into edges of $\mathcal M$\footnote{Actually, for the second step, we will consider a slightly more general setting than the one given by a regular partition.
This will allow us to treat in an unified way the proof of Theorem~\ref{theorem:spanning-codegree} and Theorem~\ref{theorem:spanning-typical}.}. In the third and last step, we  finish the embedding by using the  \nameandref{lemma:finishingembedding}.

We begin with a lemma which essentially covers all of the second step outlined above. First, we need a definition.
Let $H$ be a $k$-graph on $m$ vertices and let $\eps,d>0$. We say that $H$ is \emph{$(\eps,d)$-uniformly dense} if for all pairwise disjoint sets $W_1,\dotsc,W_k\subseteq V(H)$ with $|W_i| = h \geq \eps m$, $i\in [k]$, we have
\begin{align}\label{bound:uniformly:dense}e(W_1,\dots,W_k)\ge d h^k.\end{align}
If $H$ is $k$-partite, with parts $V_1,\dots, V_k$, we say that $H$ is \emph{$k$-partite $(\eps,d)$-uniformly dense} if for all sets $W_1\subseteq V_1,\dotsc, W_k\subseteq V_k$ with $|W_i| = h \geq \eps m$, $i\in [k]$, the bound~\eqref{bound:uniformly:dense} holds.
\begin{definition}[Uniformly dense perfect matching]\label{def:udpm}
	For $\eps,d>0$ and $t\in\mathbb N$,
	we say $\mathcal M=\{(V_{1}^{i},\dots,V_k^{i})\}_{i\in[t]}$ is an \emph{$(\eps,d)$-uniformly dense perfect matching} of a $k$-graph~$H$ if
	\begin{enumerate}[(M1)]
		\item { $\{V_{a}^{i}\}_{a\in [k], i\in[t]}$ partitions $V(H)$,}
		\item $|V_a^{i}|=|V_{b}^{j}|$ for all $i,j\in [t]$, $a,b\in[k]$, and
		\item $H[V_1^{i},\dots,V_k^{i}]$ is $k$-partite $(\eps,d)$-uniformly dense for each $i\in [t]$.
	\end{enumerate}
\end{definition}
The following proposition will be used in the second step of the proof of Theorem~\ref{theorem:spanning-codegree}, providing us with an embedding of an almost spanning tree in any graph having an $(\eps,d)$-uniformly dense perfect matching.

\begin{proposition}[Embedding almost spanning trees]\label{proposition:almost-spanning}
	Let $\Delta, k \ge 2$,
	let $\theta \ll d, 1/T_0$, and 
	let $1/n \ll \mu \ll \theta \ll \eps, \delta \ll \alpha, \gamma, 1/k, 1/\Delta$.
	Let $H$ be a $\gamma n$-large $k$-graph on $n$ vertices with a $(\delta,\mu,2)$-reservoir $R$. Let $\mathcal M$ be
	 an $(\eps,d)$-uniformly dense perfect matching of $H-R$ with $|\mathcal M|\le T_0$.
	Let $(T, \mathbf x)$ be a rooted $k$-tree with $|V(T)|\le (1-\alpha)n$  and $\Delta_1(T) \leq \Delta$.
	Then $H$ contains a copy of $T$. Moreover, for any $\theta$-extensible edge $e$, and for any $f\in\partial H$ with $ f\subseteq e$, any ordering  $\mathbf f\in\partial^\circ H$ of $f$ can be chosen as the image of $\mathbf x$. 
\end{proposition}

\begin{proof}
Let $\ell = \lfloor k/2 \rfloor (2k+1) + 2k + 1$.
Introduce  new constants $c, \beta$ satisfying $\mu \ll \beta \ll \eps \ll c \ll \alpha, 1/k$ and $\beta \ll 1/T_0, d$.

Let $\mathcal L$ be the layering for $(T,\mathbf x)$, which exists by \cref{lemma:flattening}.
We invoke \cref{lemma:cut:trees}, with parameters $\beta$, $\Delta$ and $\ell$, to obtain a $(\beta,\ell)$-decomposition of $(T,\mathbf x,\mathcal L)$ into $p\le 2\Delta^{\ell}/\beta$ parts.
That is, we find a collection of rooted trees $\{(D_i,\mathbf s_i)\}_{1 \leq i \leq p}$ such that
\begin{enumerate}[(D1)]
	\item\label{decom:1} $E(T)=\bigcupdot_{i\in[p]}E(D_i)$,
	\item\label{decom:2} $|E(D_i)|\leq \beta n$ for each $i\in[p]$,
	\item \label{decomp:5} $\mathbf s_1=\mathbf x$ and $\mathbf s_i$ is $\mathcal L$-layered,
	\item\label{decom:3} $(V(D_j)\setminus \mathbf s_j)\cap V(D_i)=\emptyset$ for all $1\le i<j\le p$, and
	\item\label{decom:4} for each $2\le j\le p$ there is a unique $i<j$ such that $\mathbf s_j\in \partial D_i$ and the rank of $\mathbf s_j$ in~$D_i$ is at least $\ell$ (in the inherited layering of $D_i$ from $\mathcal L$).
\end{enumerate}
Let $e\in H$ be an arbitrary $\theta$-extensible edge in $H$, let $f\subseteq e$ and let $\mathbf f $ be an ordering of $f$. Say $\mathcal M=\{(V_1^{h},\dots, V_k^{h})\}_{h\in[t]}$ are the clusters of the given uniformly dense perfect matching and set $m:=|V_1^1|$.

We begin our embedding by setting $\phi_0(\mathbf x):= \mathbf f$.
Now we will construct successively, for all $i \in [p]$,
an embedding $\phi_i:V(D_1\cup\dots\cup D_i)\to V(H)$ such that $\phi_{i}$ extends $\phi_{i-1}$, and 
\begin{enumerate}[(P1)]
	\item\label{almost:emb:1} if defined, $\phi_i(\mathbf s_j)$ is contained in a $\theta$-extensible edge for all $j>i$, 
	\item\label{almost:emb:2} $|\phi_i^{-1}(R)|\le i\Delta^{\ell+1}$, and
	\item\label{almost:emb:3} for every $j\in[t]$,
	$\max_{a, b \in [k]}\big ||\phi_i^{-1}(V_a^{j})|-|\phi_i^{-1}(V_b^{j})|\big|\le c m. $
\end{enumerate}
Having done this, then $\phi_p$ will be the desired embedding of $T$.

In step $i$, assume we have constructed $\phi_{i-1}$ satisfying \ref{almost:emb:1}--\ref{almost:emb:3}.
Our aim is to embed $(D_{i},\mathbf s_{i})$.
We claim that $s_{i}$ is already embedded into some $(k-1)$-tuple $\phi_{i-1}(\mathbf s_{i})$ contained in some $\theta$-extensible edge.
Indeed, if $i = 1$, we have $\mathbf s_1 = \mathbf x$ by~\ref{decomp:5} and we are done by the choice of $\phi_0$.
Otherwise, if $i \ge 2$, then~\ref{decom:4} implies $\phi_{i-1}(\mathbf s_i)$ is defined and we are done by \ref{almost:emb:1}.

We claim that there is a $j\in[t]$ such that for each $a\in[k]$,
\begin{align}
	\label{eq:size}|V_a^{j}|\ge|\phi_{i-1}^{-1}(V_a^{j})|+cm.
\end{align}
Indeed, otherwise, for each $j\in[t]$ there is an  $a\in [k]$ such that~\eqref{eq:size} does not hold.
Then, by~\ref{almost:emb:3}, also all other $V^j_b$ are almost full, and we calculate
\begin{eqnarray*}
	\textstyle|\phi_{i-1}^{-1}(V(H))| \ge \displaystyle\sum_{j\in[t], a\in[k]}|\phi_{i-1}^{-1}(V_a^j)| > \sum_{j\in[t], a\in[k]}(|V_{a}^j|-2cm)
	 = n-|R|-2cmtk
	  \geq (1-\alpha)n,
\end{eqnarray*}
a contradiction as $|V(T)|\le (1-\alpha)n$. So,~\eqref{eq:size} holds for, say, index $j$.
 
Let $\sigma$ be a bijection satisfying
\begin{equation}\label{eq:order2}
|W_1|\geq\dotsb\geq|W_k|,
\end{equation} 
where $W_a:=V_{\sigma (a)}^{j}\setminus \phi_{i-1}(V(T))$.
Because of \eqref{eq:size}, $|W_i|\ge cm$ for each $i\in [k]$, so we can select $W'_i \subseteq W_i$ of size exactly $cm$ for each $i$.
Therefore, using that $H[V_1^j,\dots,V_k^j]$ is $k$-partite $(\eps,d)$-uniformly dense on $km$ vertices and $\eps \ll c, d, 1/k$, we have
\begin{align}
	e(H[W_1,\dots,W_k])
	& \ge e(H[W'_1,\dots,W'_k]) \ge d(cm)^k.
\end{align}
Also, note that since $m = |V(H-R)|/(k T_0)$ and $R$ is a $(\delta, \mu, 2)$-reservoir, we have $m \ge n/(2 k T_0)$ and thus $|W_a| \ge cn/(2 k T_0)$ for each $1 \leq a \leq k$.

Set $R_{i-1}:=R\setminus \phi_{i-1}(D_1 \cup \dotsm \cup D_{i-1})$.
Recall that $i \leq p \leq 2 \Delta^{\ell} \beta^{-1}$.
Using~\ref{almost:emb:2}, we get $|R_{i-1}|\ge|R|-p\Delta^{\ell+1} \ge |R| - 2 \Delta^{2\ell + 1} \beta^{-1}$.
Since $R$ is a $(\delta, \mu, 2)$-reservoir for $H$ and $1/n \ll 1/\Delta, 1/k, \beta, \mu, \delta$,
we deduce that $R_{i-1}$ is a $(\delta, 2 \mu,2)$-reservoir for $H$.
Let $\mathcal{L}_i = (L_1, \dotsc, L_u)$ be the inherited layering $\mathcal{L}^{\mathbf s_i}$ for $(D_i, \mathbf s_i)$ from $\mathcal{L}$, this is well-defined since $\mathbf s_i$ is $\mathcal{L}$-layered by \ref{decomp:5}.

Use \refandname{lem:embedding} with
\begin{center}
\begin{tabular}{c|c|c|c|c|c|c|c|c|c|c|c|c}
	object/parameter & $\ell - 1$ & $R_{i-1}$ & $\delta$ & $2 \mu$ & $c/(2 k T_0)$ & $d/2$ & $D_i$ & $\mathbf s_i$ & $\mathcal{L}_i$ & $\phi_{i-1}(\mathbf s_i)$  \\ 	\hline
	playing the role of & $\ell$ & $R$ & $\gamma$ & $\mu$ & $c$ & $d$ & $T$ & $\mathbf x$ & $\mathcal{L}$ & $\mathbf f$  \\
\end{tabular}
\end{center}
to find an embedding $\phi'_i$ of $D_i$ such that
\begin{enumerate}[(E1)]
	\item \label{item:almostspanning-iteration-root} $\phi'_i(\mathbf{s_i}) = \phi_{i-1}(\mathbf{s_i})$,
	\item \label{item:almostspanning-iteration-reservoir} $(\phi'_i)^{-1}(R_{i-1} \cup \phi_{i-1}(\mathbf s_i) ) = \bigcup_{r=1}^{\ell-1} L_r$,
	\item \label{item:almostspanning-iteration-order} $\phi'_i(V_{[\ell, u]}) \subseteq W_1 \cup \dotsb \cup W_k$, with $|(\phi'_i)^{-1}(W_1)| \ge \dotsb \ge |(\phi'_i)^{-1}(W_k)|$,
	\item \label{item:almostspanning-iteration-nextroots} $\phi'_i(e')$ is $\theta$-extensible, for each $e'$ in $E(D_i[V_{[\ell, u]}])$.
\end{enumerate}
By \ref{decom:3}--\ref{decom:4} and~\ref{item:almostspanning-iteration-root},
we get that $\phi_i = \phi_{i-1} \cup \phi'_i$ is an extension of $\phi_{i-1}$ which embeds $T[V(D_1 \cup \dotsb \cup D_i)]$.
It is only left to check that $\phi_{i}$ satisfies \ref{almost:emb:1}--\ref{almost:emb:3}.

We check~\ref{almost:emb:1} holds for $\phi_i$.
Since~\ref{almost:emb:1} holds for $i-1$, we only need to consider those $j > i$ with $\mathbf s_j \in \partial D_i$.
By~\ref{decom:4}, each $\mathbf s_j \in \partial D_i$ with $j>i$ has rank at least $\ell$ in $\mathcal{L}_i$, namely it must be contained in $V_{[\ell, u]}$.
Thus, by \ref{item:almostspanning-iteration-nextroots}, we know that all such $\phi_i(\mathbf s_j)$ are contained in a $\theta$-extensible edge, as required.

To see~\ref{almost:emb:2} holds, we note that since~\ref{almost:emb:2} holds for $i-1$, $|\phi_i^{-1}(R \setminus R_{i-1})| = |\phi_{i-1}^{-1}(R \setminus R_{i-1})| \leq (i-1) \Delta^{\ell + 1}$,
so it only remains to show that $|\phi_i^{-1}(R_{i-1})| \leq \Delta^{\ell+1}$.
Note that by \cref{proposition:boundedclusters} we have that $|L_r| \leq \Delta^{r - 1}$ holds for all $1 \leq r \leq \ell - 1$,
and together with~\ref{item:almostspanning-iteration-reservoir} we get $|\phi_i^{-1}(R_{i-1})| \leq \sum_{r=1}^{\ell-1} |L_r| \leq \sum_{r=1}^{\ell-1} \Delta^{r - 1} \leq \Delta^{\ell+1}$, as required.

Finally, for \ref{almost:emb:3}, observe that \ref{item:almostspanning-iteration-order} implies that
\begin{align*}
	|V(D_i)| \ge |\varphi_i^{-1}(W_1)\setminus \varphi_{i-1}^{-1}(W_1)|\ge \dotsb \ge |\varphi_i^{-1}(W_k)\setminus \varphi_{i-1}^{-1}(W_k)|.
\end{align*}
Using~\eqref{eq:order2}, we get, for any $a, b \in [k]$ with $a < b$,
\[\Big||\phi_{i}^{-1}(V_{a}^{j})|-|\phi_{i}^{-1}(V_{b}^{j})|\Big|\le\max\Big\{\Big||\phi_{{i-1}}^{-1}(V_{a}^{j})|-|\phi_{{i-1}}^{-1}(V_{b}^{j})|\Big|, |V(D_i)| \Big\}.\]
Since~\ref{almost:emb:3} holds for $i-1$, the first term in the maximum is bounded from above by $cn$,
and $|V(D_i)| \leq \beta n \leq c n$, follows from~\ref{decom:2} and $\beta \ll c$.
\end{proof}

Now we are ready to prove Theorem~\ref{theorem:spanning-codegree}.

\begin{proof}[Proof of Theorem~\ref{theorem:spanning-codegree}]
	We begin by noting that $\delta_{k-1}(H) \ge (1/2 + \gamma)n$ implies that $H$ is $(2 \gamma n)$-large, and also that every vertex is contained in an edge. \medskip
	
	\noindent\emph{Step 1: Finding the absorbing structures.}
	For the remainder of the proof, we choose $\ell = (2k+1)\lfloor{k/2}\rfloor+2k$ and also
	\[1/n_0 \ll \nu_3 \ll \nu_2 \ll \delta, \nu_1 \ll \alpha \ll \theta_0 \ll \gamma, 1/k, 1/\Delta, 1/\ell.\]
	Let $T$ be any given tree on $n \ge n_0$ vertices such that $\Delta_1(T)\le\Delta$.
	We choose an arbitrary tuple $\mathbf x\in\partial^\circ T$ as the root of $T$ and, using Lemma~\ref{lemma:flattening}, we find a  layering $\mathcal L=(L_1,\dots, L_m)$ for $(T,\mathbf x)$.
	Let $\mathbf x'\in\partial T$ be an $\mathcal L$-layered tuple such that $|E(T_{\mathbf x'})|\ge \alpha n/(2\Delta)$ with the highest possible rank.
	Then, using \cref{proposition:cut1} for $(T_{\mathbf x'},\mathbf x',\mathcal L^{\mathbf x'})$, and by the maximality of~$\mathbf x'$ we deduce that
	\begin{align}\label{alllpha}
		\frac{\alpha n}{2 \Delta} & \le |E(T_{\mathbf x'})|\le\alpha n.
	\end{align}
	Since $\delta \ll 1/k, 1/\Delta, 1/\ell$,
	we can use Proposition~\ref{proposition:findingfamilyinT} with $2\Delta k(\ell+3k)$ and $\delta$ in place of $\ell$ and $\mu$, respectively.
	This yields a $(k-1)$-tree $X$ with $h\le\Delta+k-1$ vertices,
	and a $2\Delta k(\ell+3k)$-separated set $\mathcal B$ of $X$-tuples of $T_{\mathbf x'}$ such that $|\mathcal B|\ge \delta n$.
	Recall that $H$ is $(2\gamma n)$-large, there are no isolated vertices, and additionally $1/n \ll \nu_1 \ll \gamma, 1/k, 1/\Delta$ and $h \leq \Delta+k-1$.
	Thus, Proposition~\ref{proposition:absorbingtuples} (with $\nu_1$ playing the role of $\beta$) tells us that $|\Lambda_X(v_1,\dots,v_k)|\ge \nu_1 n^{h+1}$ for every $k$-tuple of distinct vertices $(v_1,\dots,v_k) \in V(H)^k$.
	The hierarchy $\nu_2 \ll \nu_1, 1/k, 1/h$ allows us to use Lemma~\ref{lemma:usingchernoff}, with parameters $\nu_1$, $\nu_2$ in place of $\beta, \alpha$ respectively, to deduce the existence of a family $\mathcal A \subseteq \Lambda_X$ of size at most $\nu_2 n$, such that 
	\begin{align}
		\text{for every $k$-tuple $(v_1,\dots,v_k)$ we have }|\Lambda_X(v_1,\dots,v_k)\cap \mathcal A| \ge \frac{\nu_1 \nu_2}{8} n \ge 2 \nu_3 n, \label{equation:final-absorber}
	\end{align}
	where the last inequality follows from $\nu_3 \ll \nu_2, \nu_1$.
	
	Bounding very crudely, we deduce that $H$ has at least $\gamma \binom{n}{k}$ edges.
	Since $\theta_0 \ll \gamma, 1/k$, an application of Lemma~\ref{lem:extensible} implies the existence of a $\theta_0$-extensible edge $e_0$ which is vertex-disjoint of all the subgraphs in $\mathcal A$.
	Let $f_0 \subseteq e_0$ be an arbitrary $(k-1)$-set and let $\mathbf f_0$ be an ordering of $f_0$.
	Finally, since $1/n \ll \nu_2 \ll \delta \ll \alpha \ll \gamma, 1/h$, we can use Lemma~\ref{lemma:coveringtuples}, with $\nu_2, \delta, \alpha, 2 \gamma, \mathbf x', \mathbf f_0$ playing the roles of $\alpha, \mu, \nu, \gamma, \mathbf x, \mathbf y$ respectively, to find an embedding $\phi':V(T_{\mathbf x'})\to V(H) \setminus (e_0 \setminus \mathbf f_0)$ which $X$-covers each tuple in $\mathcal A$,
	such that $\phi'(\mathbf x') = \mathbf f_0 \subseteq e_0$, and $e_0 \setminus \mathbf f_0$ is disjoint from $\phi'(V(T_{\mathbf x'}))$. \medskip
		
	\noindent \emph{Step 2: Finding an almost spanning tree.}
	We introduce new constants by letting
	\[1/n \ll \mu \ll \theta \ll 1/T_0 \ll 1/t_0 \ll \eps'\ll \eta \ll \eps \ll \gamma'\ll d\ll \nu_3.\]
	Let $H'= H - \phi'(V(T_{\mathbf x'})) + e_0$ and $n' = |V(H')|$.
	Then, by~\eqref{alllpha} we have $|V(T_{\mathbf x'})| \leq \alpha n + k-1 \leq 2 \alpha n$, and thus $n' \ge (1-2\alpha)n$.
	Note that $e_0 \subseteq V(H')$.
	
	The choice of $\alpha \ll \gamma$ ensures that $H'$ has minimum codegree at least $(1/2+2\gamma/3)n'$.
	Using $1/n \ll \mu \ll \gamma'$, Lemma~\ref{lem:reservoir} provides us with a $(\gamma',\mu,2)$-reservoir $R$ for $H'$.
	Now set $H'' = H'- R$ and $n'' = |V(H'')|$.
	Since $\gamma' \ll \alpha$, we deduce that $n'' \ge (1 - 3 \alpha)n$.	
	
	Now we prepare the setup to apply regularity tools.
	Since $1/n, 1/T_0 \ll 1/t_0, 1/k, \eps'$,
	an application of the \nameandref{theorem:weakregularity} to $H''$, with parameters $\eps'$ and $t_0$ as input,
	yields an $\eps'$-regular partition $\mathcal{P} = \{V_0, V_1, \dotsc, V_t\}$ of $V(H'')$, for some $t_0\le t\le T_0$.
	Using Lemma~\ref{lemma:matching}, $1/t_0 \ll \eps' \ll 1/k, \gamma, \eta$ and $2d\ll \gamma$, we know that the $(\eps',2d)$-reduced graph $R_{2d}(H'')$ contains a matching $\mathcal M$ covering at least $(1-\eta)t$ clusters.
	Since each edge $(V_{i_1},\dots,V_{i_k})\in\mathcal M$ is $(\eps',d')$-regular for some $d' \ge 2d$,
	by our choice of $\eps' \ll \eps$, we deduce that $H''[V_{i_1},\dots,V_{i_k}]$ is $k$-partite $(\eps,2d - \eps)$-uniformly dense,
	and thus $(\eps,d)$-uniformly dense.
	Let $V_{\mathcal M} \subseteq V(H'')$ consist of the union of clusters covered by $\mathcal{M}$ in the reduced graph.
	Thus, $H''[V_{\mathcal M}]$ has an $(\eps,d)$-uniformly dense perfect matching with $p \leq T_0/k$ edges.
	
	We wish to apply Proposition~\ref{proposition:almost-spanning}.	For this we need a $k$-graph having a reservoir and a uniformly dense perfect matching, which we plan to be $R$ and $V_{\mathcal M}$, respectively.
	We also need the $k$-graph to contain our desired root $e_0$, which is why we set $R' = (R \cup e_0) \setminus V_{\mathcal M}$.
	Let $H^\ast$ be the induced subgraph of $H'$ on $R' \cup V_{\mathcal M}$ and set $n^\ast = |V(H^\ast)|$.

	We now check that the requirements of Proposition~\ref{proposition:almost-spanning} are satisfied with this choice of $H^\ast, R', V_{\mathcal M}$.	
	Clearly $\{R', V_{\mathcal M}\}$ partitions $V(H^\ast)$ and $e_0 \subseteq V(H^\ast)$.
	As $\mathcal{P}$ is an $\eps'$-regular partition of $H'' = H' - R$,
	and $\mathcal{M}$ leaves at most $\eta t$ clusters uncovered, 
	each having size at most $n''/t$,
	we see that
	\[n' - n^\ast \leq |V_0| + \sum_{V_i \notin V(\mathcal{M})} |V_i| \leq \eps' n'' + \eta t n'' / t \leq \eps n',\]
	where we used $\eps', \nu \ll \eps$ and $n'' \leq n'$ in the last inequality.
	This implies  $n^\ast \ge (1 - \eps)n'$.
	Since $\eps \ll \gamma$ and since $\delta_{k-1}(H') \ge (1/2 + 2 \gamma / 3)n'$,
	we deduce that $H^\ast$ has minimum codegree at least $(1/2+\gamma/2)n^\ast$.
	In particular, $H^\ast$ is $\gamma n^\ast$-large.
	Since $R$ is a $(\gamma', \mu, 2)$-reservoir for $H'$,
	we use $n^\ast \ge (1 - \eps)n'$ and $1/n \ll \mu, 1/k$ to deduce that
	$R'$ has size $\gamma'n' \pm (\mu n' + k)$ and then $|R'|= (\gamma^\ast \pm 2 \mu)n^\ast$, for some $\gamma^\ast \leq 2 \gamma'$.
	Thus $R'$ is a $(\gamma^\ast, 2\mu, 2)$-reservoir for $H^\ast$.
	Finally, recall that $e_0$ is $\theta_0$-extensible in~$H$.
	As $1/n \ll \theta \ll \eps \ll \alpha \ll \theta_0$,
	we have $n - n^\ast \leq 4 \alpha n$,
	implying that $e_0$ is $\theta$-extensible in $H^\ast$.
	
	We set $T'=T-(T_{\mathbf x'}-\mathbf x')$ and root $T'$ at $\mathbf x'$.
	Let $v_1,\dots,v_{|V(T')|}$ be a valid ordering of $(T',\mathbf x')$ with $\mathbf x' = ( v_1, \dotsc, v_{k-1})$.
	Let $m^\ast = \lceil (1-\nu_2)n^\ast \rceil$ and $T^\ast=T'[v_1,\dots,v_{m^\ast}]$, such that $|V(T^\ast)| = m^\ast$.
	An application of Proposition~\ref{proposition:almost-spanning} with
	\begin{center}
		\begin{tabular}{c|c|c|c|c|c|c|c|c|c|c}
			object/parameter & $n^\ast$ & $2 \mu$ & $\gamma^\ast$ & $\nu_3/2$ & $d$ & $H^\ast$ & $R'$ & $\mathbf x'$&$e_0$ & $\mathbf f_0$ \\ 	\hline
			playing the role of & $n$ & $\mu$ & $\delta$ & $\alpha$ & $d$ & $H$ & $R$ & $\mathbf x$& $e$ & $\mathbf f$ \\
		\end{tabular}
	\end{center}
	yields an embedding $\phi^\ast :V(T^\ast)\to V(H^\ast)$ with $\phi^\ast (\mathbf x')= \mathbf f_0 = \phi'(\mathbf x')$.
	\medskip
	
	\noindent \emph{Step 3: Finishing the embedding.}
	The embedding $\phi'\cup\phi^\ast$  of $T_{\mathbf x'}\cup T^\ast$ fulfils all the conditions of Lemma~\ref{lemma:finishingembedding} (as $|V(T_{\mathbf x'}\cup T^\ast)| \ge m^\ast + |V(T_{\mathbf x'})| - k -1 \ge (1 - \nu_3)n$ and by~\eqref{equation:final-absorber}).
	So, we can use the Absorbing lemma (Lemma~\ref{lemma:finishingembedding}) to embed the remaining vertices $v_{m^\ast + 1},\dots,v_{|V(T')|}$.
\end{proof}

\section{Spanning trees in quasirandom hypergraphs}\label{section:quasirandom}

In this section, we prove Theorem~\ref{theorem:spanning-typical}.
Actually, we prove a slightly stronger statement since we only need the weaker notion of being {\it uniformly dense} (given right before Definition~\ref{def:udpm}).

\begin{theorem} \label{theorem:spanning-uniformlydense}
	Let $0 < 1/n \ll \delta \ll \eta, \gamma, 1/\Delta, 1/k$ and let $H_1$ and $H_2$ be $k$-graphs on the same vertex set of size $n$.
	Assume that $H_1$ is $(\eta, \delta)$-uniformly dense, and that $H_2$ is $\gamma n$-large and has no isolated vertices.
	Then, $H := H_1 \cup H_2$ contains a copy of every $k$-tree $T$ on $n$ vertices with $\Delta_1(T) \leq \Delta$.
\end{theorem}

\begin{proof}[Proof (sketch)]
	Since the proof is essentially the same as the proof of Theorem~\ref{theorem:spanning-codegree}, we only outline the major differences.
	In particular,  step 1 (construction of absorbers) is virtually the same, noting that the only property used there is that the host graph was that $H$ is $2 \gamma n$-large and has no isolated vertices, while here $H_2$ is $\gamma n$-large and has no isolated vertices.
	
	The difference in step 2 is that instead of applying regularity to obtain the uniformly dense perfect matching, we use that $H_1$ is uniformly dense.
	Indeed, suppose we have already found the equivalent of~$H'$ with a reservoir $R$, and we wish to find a uniformly dense matching covering most vertices in $H'' = H' - R$.
	Move at most $k-1$ vertices from $H''$ to $R$ we can assume $|V(H'')|$ is divisible by~$k$.
	Fix any partition $V_1, \dotsc, V_k$ of $V(H'')$ into equal-sized parts.
	Since $H_1$ is $(\eta, \delta)$-uniformly dense, $\mathcal M = \{(V_1, \dotsc, V_k)\}$ is an $(\eta,\delta)$-uniformly dense perfect matching with only one edge.
	Having found $\mathcal M$, then we can proceed with the remainder of step 2 and step 3 exactly as before.
\end{proof}

Let us now show how \cref{theorem:spanning-uniformlydense} implies the same result for 
host hypergraphs  satisfing certain {\it quasirandom properties}, namely {\it weak quasirandomness} (\cref{corollary:spanning-quasirandom}) and {\it typicality} (\cref{theorem:spanning-typical}).

The study of quasirandomness for graphs evolved in the late 1980's, a milestone being the seminal result of Chung, Graham and Wilson~\cite{ChungGrahamWilson1989} relating uniform edge distribution to other `random-like' properties. For the history of 
quasirandomness for hypergraphs we refer to the exposition in~\cite{ACHPS2018}.
The weakest  form of quasirandomness studied in the literature is as follows.
Call a $k$-graph $H$  \emph{weakly $(\eta, \delta)$-quasirandom} if for every $U \subseteq V(H)$, the number $e(U)$ of edges entirely contained in $U$ satisfies
\begin{align}
\Big|e(U) - \eta \binom{|U|}{k} \Big| \leq \delta n^{k}. \label{equation:weakquasirandom}
\end{align}
An inclusion-exclusion argument shows that every weakly $(\eta, \delta)$-quasirandom $k$-graph is $(\eta, 2^k \delta)$-uniformly dense.
Thus the following is an immediate corollary of \cref{theorem:spanning-uniformlydense}.

\begin{corollary} \label{corollary:spanning-quasirandom}
	Let $0 < 1/n \ll \delta \ll \eta, \gamma, 1/\Delta, 1/k$ and
	let $H_1$ and $H_2$ be $k$-graphs on the same vertex set of size $n$.
	Assume that $H_1$ is weakly $(\eta, \delta)$-quasirandom,
	and that $H_2$ is $ \gamma n$-large and has no isolated vertices.
	Then, $H := H_1 \cup H_2$ contains a copy of every $k$-tree $T$ on $n$ vertices with $\Delta_1(T) \leq \Delta$.
\end{corollary}

As mentioned in the introduction, being $\gamma n$-large in Corollary~\ref{corollary:spanning-quasirandom} cannot be  replaced with a lower bound of type $\Omega(n^{k-1})$ on the minimum degree of $H$, as evidenced by an example by Ara\'ujo, Piga and Schacht~\cite[Section 8.2]{AraujoPigaSchacht2020}.

Let us now turn  to the notion of {\it typicality}, which was defined in the introduction.
Although weakly quasirandomness does not imply typicality, one can show\footnote{Consider the $k$-graph $F$ with $V(F)=[k]\times\{0,1\}$ and edges of the forms $\{(1,0),(2,x_2),\dotsc,(k,x_k)\}$ and $\{(1,1),(2,x_2),\dotsc,(k,x_k)\}$, where $x_i\in\{0,1\}$ for $i\in\{2,\dotsc,k\}$. It is clear that in any $(\eta,2,\delta)$-typical $k$-graph $H$ one can estimate the density of $H$ and the number of $F$-copies, up to an error term depending on $\delta$. Hence~$H$ satisfies the property $\text{disc}_{{\mathcal Q},\eta}$ for $\mathcal Q=\{\{1\},\{2,\dotsc,k\}\}$, implying that $H$ is weakly quasirandom (see~\cite{ACHPS2018, Towsner2017} for details).} that the converse is true: any  $(\eta, 2, \delta)$-typical $k$-graph is weakly $(\eta, \delta')$-quasirandom, for some constant $\delta'$.
This is the only ingredient we need in order to prove Theorem~\ref{theorem:spanning-typical}.

\begin{proof}[Proof of Theorem~\ref{theorem:spanning-typical}]
Let $1/n_0 \ll \eps' \ll \eps$.
As discussed above, $H$ is weakly $(\varrho, \eps')$-quasirandom.
Also, note that every $(\varrho, 2, \eps)$-typical graph on $n$ vertices is $(\varrho^2 - \eps)n$-large, and cannot have isolated vertices.
Thus the theorem follows from \cref{corollary:spanning-quasirandom}, by taking $H = H_1 = H_2$.
\end{proof}

\section{Further questions}\label{section:remarks}

\subsection{Degree variations}
For any $1\le j<k-1$, 
one can define the {\it minimum  $j$-degree} $\delta_{j}(H)$ of a graph $H$ in analogy to the minimum codegree 
$\delta_{k-1}(H)$ as defined in the introduction. 
Also, define $\st_j(k)$ as the smallest $\delta > 0$ such that for every $\Delta, k \ge 2$ and $\mu > 0$,
every large enough $k$-graph $H$ with $\delta_j(H) \ge (\delta + \mu) \binom{n - j}{k - j}$ contains every $k$-tree $T$ of the same order and with $\Delta_1(T) \leq \Delta$.
Then, by Theorem~\ref{theorem:spanning-codegree}, we have $\st_{k-1}(k) \le 1/2$ for all $k \ge 2$, and by Proposition~\ref{proposition:lowerbound} this is tight.

It would be interesting to understand $\st_j(k)$ for $j < k-1$. If instead of a spanning tree we are looking for a spanning cycle, then some results are known for the analogous problem.
In $k$-graphs, recall that a tight Hamilton cycle in a $k$-graph $H$ on $n$ vertices is a sequence of distinct vertices $v_1, \dotsc, v_n$ such that, for each $i \in \{1, \dotsc, n \}$, the edge $\{ v_i, v_{i+1}, \dotsc, v_{i+k-1} \}$ is present in $H$ (indices understood modulo~$n$).
Then $\hc_j(k)$ is defined accordingly.
It is known that $\hc_{k-2}(k) = 5/9$ for all $k \ge 3$~\cite{PRRS2020, LangSanhueza2020}.
For general $j, k$ satisfying $1 \leq j \leq k-3$, the current best lower~\cite{HZ2016} and upper~\cite{LangSanhueza2020} bounds are \[ 1-\frac{1}{\sqrt{k-j}} \leq \hc_j(k) \leq 1 - \frac{1}{2(k-j)}. \]
Furthermore, the result of R\"odl, Ruci\'nski and Szemer\'edi~\cite{RRS08} mentioned in the introduction states that $\hc_{k-1}(k) = 1/2$ for all $k \ge 2$.
Thus $\hc_{k-1}(k) = \st_{k-1}(k) = 1/2$ for all $k \ge 2$, and we believe the same should be true in general.
\begin{conjecture}
	For all $k > j \ge 1$, we have $\hc_{j}(k) = \st_j(k)$.
\end{conjecture}

\subsection{Trees of larger maximum degree}

Another possible generalisation of Theorem~\ref{theorem:spanning-codegree} would be to relax the condition on $\Delta_1(T) = O(1)$, allowing it to grow with $n$.
In the graph case, Koml\'os, S\'ark\"ozy, and Szemer\'edi~\cite{KomlosSarkozySzemeredi2001} strengthened their earlier result, Theorem~\ref{thm:KSSoriginal}, considerably by showing that, with the same minimum degree conditions, one can  find all trees of maximum degree at most $c n / \log n$, for a fixed $c > 0$ depending on the approximation $\gamma$ only. They also gave an example showing this bound is tight up to a multiplicative factor.

A natural adaptation of their example shows that, for $k \ge 3$, the bound on $\Delta_{1}(T)$ in Theorem~\ref{theorem:spanning-codegree} cannot be larger than $O(n / \log n)$.
Indeed, let $T$ be the $k$-tree consisting of a vertex $v$ whose link graph is a $(k-1)$-tight path $P$ of length $c\log n$, for some sufficiently small constant $c>0$. For each set $S$ of $k-1$ consecutive vertices in $P$, we add $n/(c\log n)$ new vertices adjacent to $S$. Then $T$ has maximum degree $\Theta(n/\log n)$, and a straightforward calculation shows that, with high probability,  the binomial random $k$-graph of edge density $p=0.9$ does not contain an ordered set $U$ of size $c\log n$ such that each vertex outside $U$ is adjacent to some $k-1$ consecutive vertices in $U$.

\section*{Acknowledgment}
The authors are very grateful to Esteban Quiroz Camarasa for many valuable discussions in the beginning of this project,
and to Yanitsa Pehova for useful comments.

\sloppy\printbibliography

\end{document}